%% file: ms.tex
\documentclass[a4paper,final]{siamonline171218}

\usepackage{ifthen}
\newboolean{ispreprint}
\setboolean{ispreprint}{true}
\ifthenelse{\boolean{ispreprint}}%
{}{}

\usepackage{definitions}

\newcommand{\TheTitle}{Discrete Total Variation with Finite Elements and Applications to Imaging} 
 
\newcommand{\TheAuthors}{M. Herrmann, R. Herzog, S. Schmidt, J. Vidal and G. Wachsmuth}

\title{{\TheTitle}\thanks{This work was supported by DFG grants HE~6077/10--1 and SCHM~3248/2--1 within the \href{https://spp1962.wias-berlin.de}{Priority Program SPP~1962} (\emph{Non-smooth and Complementarity-based Distributed Parameter Systems: Simulation and Hierarchical Optimization}), which is gratefully acknowledged.}}

\author{
	Marc Herrmann\thanks{Julius-Maximilians-Universität Würzburg, Faculty of Mathematics and Computer Science, Lehrstuhl für Mathematik~VI, Emil-Fischer-Straße~40, 97074 Würzburg, Germany (MH: \email{marc.herrmann@mathematik.uni-wuerzburg.de}; STS: \email{stephan.schmidt@mathematik.uni-wuerzburg.de}, \url{https://www.mathematik.uni-wuerzburg.de/\string~schmidt}).}
	\and
	Roland Herzog\thanks{Technische Universität Chemnitz, Faculty of Mathematics, Professorship Numerical Mathematics (Partial Differential Equations), 09107 Chemnitz, Germany (RH: \email{roland.herzog@mathematik.tu-chemnitz.de}, \url{https://www.tu-chemnitz.de/herzog}, JV: \email{jose.vidal-nunez@mathematik.tu-chemnitz.de}, \url{https://www.tu-chemnitz.de/mathematik/part_dgl/people/vidal}).}
	\and
	Stephan Schmidt\footnotemark[2]
	\and
	Jos{\'e} Vidal\footnotemark[3]
	\and
	Gerd Wachsmuth\thanks{Brandenburgische Technische Universität Cottbus-Senftenberg, Institute of Mathematics, Chair of Optimal Control, Platz der Deutschen Einheit~1, 03046 Cottbus, Germany (\email{gerd.wachsmuth@b-tu.de}, \url{https://www.b-tu.de/fg-optimale-steuerung}).}
}

\ifpdf
	\hypersetup{
		pdftitle={\TheTitle},
		pdfauthor={\TheAuthors}
	}
\fi

\begin{document}

\maketitle

\input{main.tex}


\IfFileExists{World.bib}
{
	\ifthenelse{\boolean{ispreprint}}{\bibliographystyle{plain}}{\bibliographystyle{erae}}
	\bibliography{World}
}
{
	\ifthenelse{\boolean{ispreprint}}{\bibliographystyle{plain}}{\bibliographystyle{erae}}
	\bibliography{paper}
}

\end{document}

%% file: main.tex
 \begin{abstract} 
	The total variation (TV)-seminorm is considered for piecewise polynomial, globally discontinuous (DG) and continuous (CG) finite element functions on simplicial meshes.
	A novel, discrete variant (DTV) based on a nodal quadrature formula is defined.
	DTV has favorable properties, compared to the original TV-seminorm for finite element functions.
	These include a convenient dual representation in terms of the supremum over the space of Raviart--Thomas finite element functions, subject to a set of simple constraints.
	It can therefore be shown that a variety of algorithms for classical image reconstruction problems, including TV-$L^2$ and TV-$L^1$, can be implemented in low and higher-order finite element spaces with the same efficiency as their counterparts originally developed for images on Cartesian grids.
\end{abstract}

\ifthenelse{\boolean{ispreprint}}%
	{%
	}%
	{
		\keywords{%
			discrete total variation,
			dual problem,
			image reconstruction,
			numerical algorithms
		}
	}

\ifthenelse{\boolean{ispreprint}}%
{}{
\subclass{%
	94A08, 
	68U10, 
	49M29, 
	65K05  
}
}

\section{Introduction}
\label{sec:Introduction}

The total-variation (TV)-seminorm $\abs{\,\cdot\,}_{TV}$ is ubiquitous as a regularizing functional in image analysis and related applications; see for instance \cite{RudinOsherFatemi1992,EsedogluOsher2004,ChanEsedoglu2005,ChambolleCasellesCremersNovagaPock2010}.
When $\Omega \subset \R^2$ is a bounded domain, this seminorm is defined as
\begin{multline}
	\label{eq:continuous_TV}
	\ifthenelse{\boolean{ispreprint}}{\hfill}{}
	\abs{u}_{TV(\Omega)} 
	\ifthenelse{\boolean{ispreprint}}{}{\\}
	\coloneqq
	\sup \left\{ \int_\Omega u \div \bp \, \dx: \bp \in C^\infty_c(\Omega;\R^2), \; \abs{\bp}_{s^*} \le 1 \right\},
	\ifthenelse{\boolean{ispreprint}}{\hfill}{}
\end{multline}
where $s \in [1,\infty]$, $s^* = \frac{s}{s-1}$ denotes the conjugate of $s$ and $\abs{\,\cdot\,}_{s^*}$ is the usual $s^*$-norm of vectors in $\R^2$.
Frequent choices include $s = 2$ (the isotropic case) and $s = 1$, see \cref{fig:illustration_of_anisotropy_DG0}.
\begin{figure*}[htb]
	\centering
	\raisebox{10mm}{\includegraphics[height=30mm]{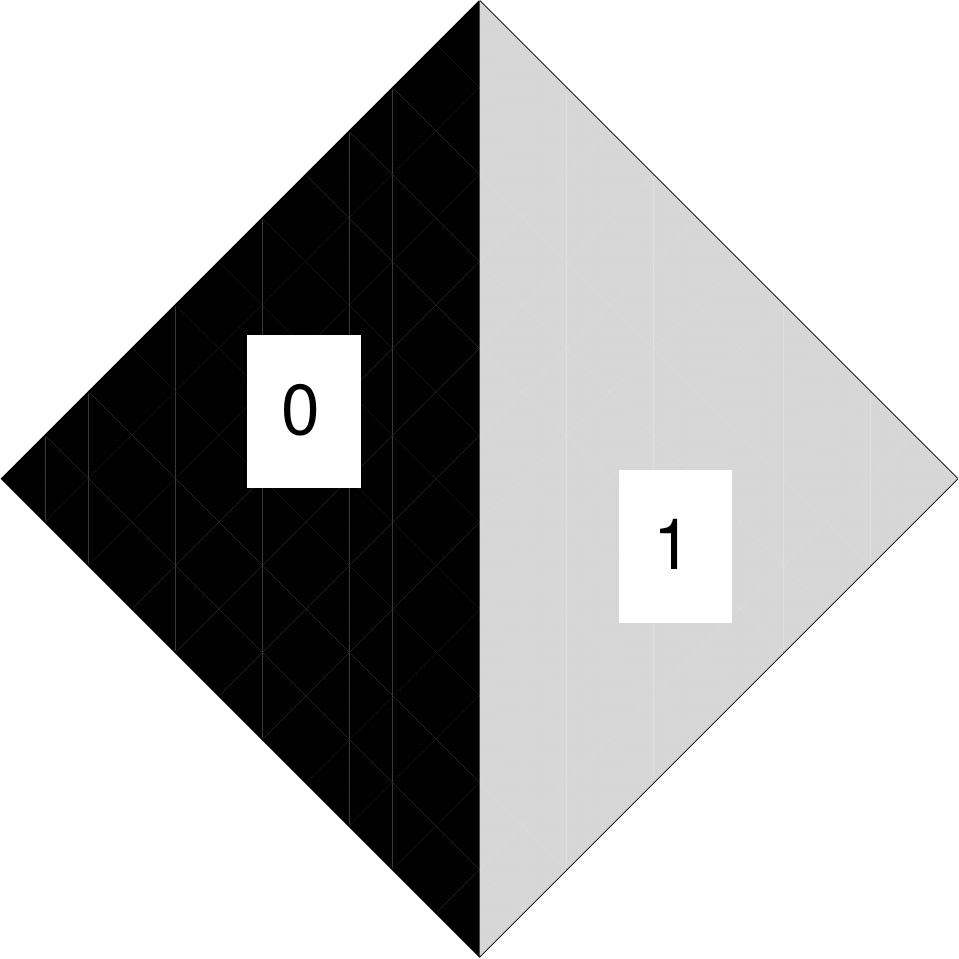}}
	\quad
	\includegraphics[height=50mm]{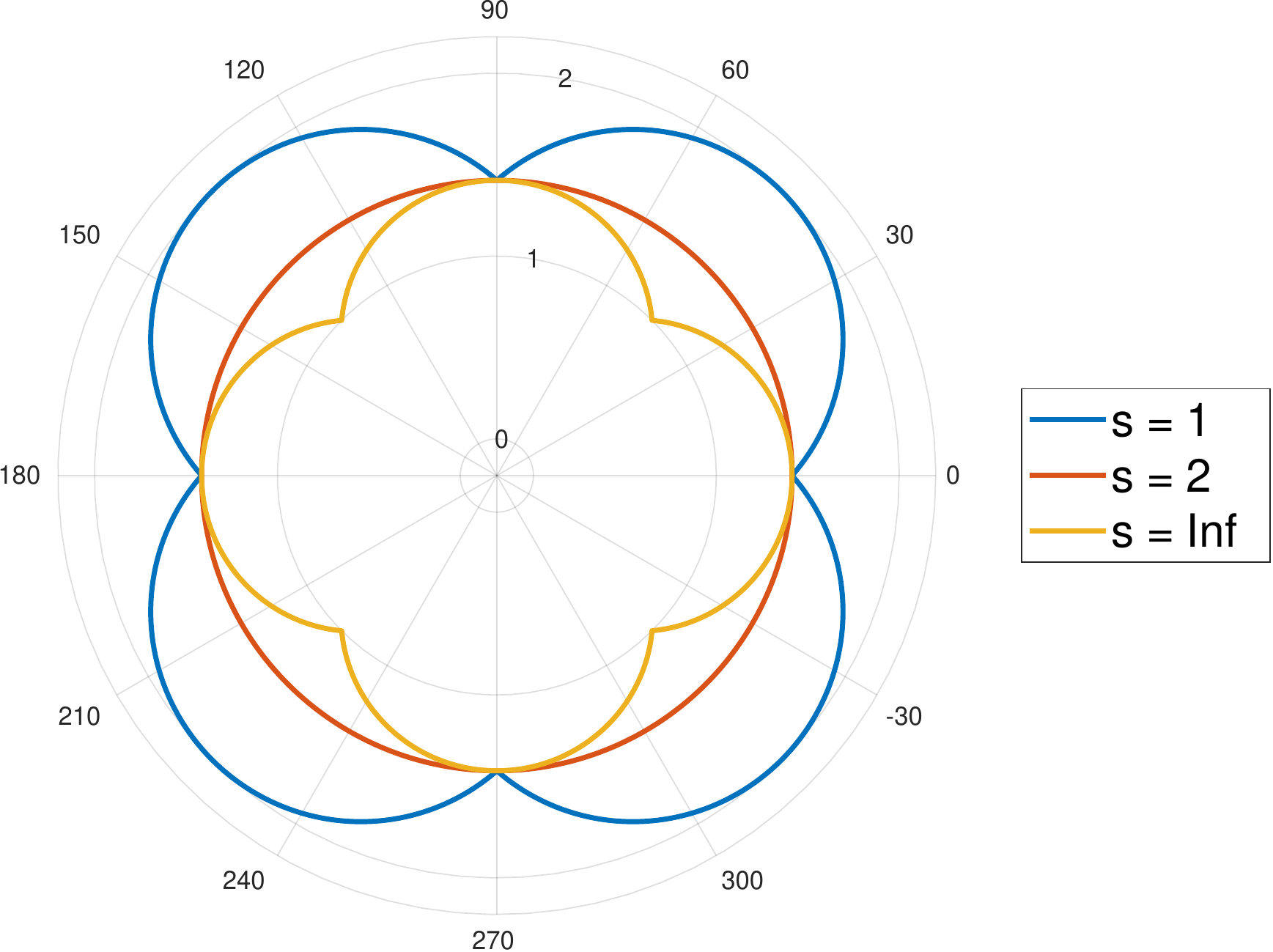}
\caption{A $\DG{0}(\Omega)$ function $u$ with values~0 and 1 on two triangles forming the unit square $\Omega$ (left), and the value of the associated TV-seminorm $\abs{u}_{TV(\Omega)} = \abs{u}_{DTV(\Omega)}$ as a function of the rotation angle of the mesh.}
	\label{fig:illustration_of_anisotropy_DG0}
\end{figure*}

It has been observed in \cite{Condat2017} that \eqq{the rigorous definition of the TV for discrete images has received little attention.}
In this paper we propose and analyze a discrete analogue of \eqref{eq:continuous_TV} for functions $u$ belonging to a space $\DG{r}(\Omega)$ or $\CG{r}(\Omega)$ of globally discontinuous or continuous finite element functions of polynomial degree\footnote{It will become clear in \cref{sec:Discrete_total_variation} why the discussion is restricted to polynomial degrees at most~4. Although this should be sufficient for most practical purposes, we briefly discuss extensions in \cref{sec:Conclusion_outlook}.} $0 \le r \le 4$ on a geometrically conforming, simplicial triangulation of $\Omega$, consisting of triangles $T$ and interior edges $E$.\footnote{\added{While we mainly discuss the case of $\Omega \subset \R^2$, an extension to 3D is detailed in \cref{subsec:3D}.}}

In this case, it is not hard to see that the TV-seminorm \eqref{eq:continuous_TV} can be evaluated as
\begin{equation}
	\label{eq:continuous_TV_for_DG}
	\abs{u}_{TV(\Omega)} 
	=
	\sum_T \int_T \abs{\nabla u}_s \, \dx
	+
	\sum_E \int_E \bigabs{\vjump{u}}_s \, \ds,
\end{equation}
where $\vjump{u}$ denotes the vector-valued jump of a function in normal direction across an interior edge of the triangulation.

It is intuitively clear that when $u$ is confined to a finite element space such as $\DG{r}(\Omega)$ or $\CG{r}(\Omega)$, then it ought to be sufficient to consider the supremum in \eqref{eq:continuous_TV} over all vector fields $\bp$ from an appropriate finite dimensional space as well.
Indeed, we show that this is the case, provided that the TV-seminorm \eqref{eq:continuous_TV_for_DG} is replaced by its discrete analogue
\begin{multline}
	\label{eq:discrete_TV_for_DG}
	\ifthenelse{\boolean{ispreprint}}{\hfill}{}
	\abs{u}_{DTV(\Omega)} 
	\ifthenelse{\boolean{ispreprint}}{}{\\}
	\coloneqq
	\sum_T \int_T \II_T \big\{ \abs{\nabla u}_s \big\} \, \dx
	+
	\sum_E \int_E \II_E \big\{ \bigabs{\vjump{u}}_s \big\} \, \ds,
	\ifthenelse{\boolean{ispreprint}}{\hfill}{}
\end{multline}
which we term the \emph{discrete TV-seminorm}.
Here $\II_T$ and $\II_E$ are local interpolation operators into the polynomial spaces $\PP_{r-1}(T)$ and $\PP_r(E)$, respectively.
Therefore, \eqref{eq:discrete_TV_for_DG} amounts to the application of a nodal quadrature formula for the integrals appearing in \eqref{eq:continuous_TV_for_DG}.
We emphasize that both \eqref{eq:continuous_TV_for_DG} and \eqref{eq:discrete_TV_for_DG} are isotropic when $s = 2$, i.e., invariant w.r.t.\ rotations of the coordinate system.
In the lowest-order case ($r = 0$) of piecewise constant functions, the first sum in \eqref{eq:discrete_TV_for_DG} is zero and only edge contributions appear.
Moreover, in this case \eqref{eq:continuous_TV_for_DG} and \eqref{eq:discrete_TV_for_DG} coincide since $\vjump{u}$ is constant on edges.
In general, we will show that the difference between \eqref{eq:continuous_TV_for_DG} and \eqref{eq:discrete_TV_for_DG} is of the order of the mesh size, see \cref{prop:error_estimate_DTV}.

Using \eqref{eq:discrete_TV_for_DG} in place of \eqref{eq:continuous_TV_for_DG} in optimization problems in imaging offers a number of significant advantages.
Specifically, we will show in \cref{theorem:dual_representation_of_DTV} that \eqref{eq:discrete_TV_for_DG} has a discrete dual representation
\begin{align}
	\abs{u}_{DTV(\Omega)} 
	=
	\max \Bigg\{ \int_\Omega u \div \bp \, \dx: \bp \in \RT{r+1}^0(\Omega) \Bigg.
	\phantom{\Bigg\}}
	\notag
	\ifthenelse{\boolean{ispreprint}}{}{\\[-2ex]}
	\Bigg.
	\text{s.t.\ a number of simple constraints} 
	\Bigg\}
	\label{eq:discrete_TV_abstract}
	\raisetag{8mm}
\end{align}
for $u \in \DG{r}(\Omega)$, where $\RT{r+1}(\Omega)$ denotes the space of Raviart--Thomas finite element functions of order $r+1$, and $\RT{r+1}^0(\Omega)$ is the subspace defined by $\bp \cdot \bn = 0$ (where $\bn$ is the outer normal of unit Euclidean length) on the boundary of $\Omega$.
In the lowest-order case $r = 0$ in particular, one obtains
\begin{multline}
	\label{eq:discrete_TV_r=0}
	\abs{u}_{DTV(\Omega)} 
	=
	\max \Bigg\{ \int_\Omega \! u \div \bp \, \dx: \bp \in \RT{1}^0(\Omega) \Bigg., 
	\\[-2ex]
	\Bigg. \int_E \abs{\bp \cdot \bn_E} \, \ds \le \abs{E} \, \abs{\bn_E}_s \text{ on interior edges} \Bigg\}
	.
\end{multline}
Here $\bn_E$ denotes a normal vector of arbitrary orientation and unit Euclidean length, i.e., $\abs{\bn_E}_2 = 1$, on an interior edge $E$, and $\abs{E}$ denotes the (Euclidean) edge length.
Since the expressions $\int_E \abs{\bp \cdot \bn_E} \, \ds$ are exactly the degrees of freedom typically used to define the basis in $\RT{1}(\Omega)$, the constraints in \eqref{eq:discrete_TV_r=0} are in fact simple \emph{bound constraints on the coefficient vector} of $\bp$.
For comparison, the pointwise restrictions $\abs{\bp}_{s^*} \le 1$ appearing in \eqref{eq:continuous_TV} are nonlinear unless $s^* \in \{1,\infty\}$.
For the case of higher-order finite elements, i.e., $1 \le r \le 4$, further constraints in \eqref{eq:discrete_TV_abstract} impose an upper bound on the $\abs{\,\cdot\,}_{s^*}$-norm of \emph{pairs of coefficients} of $\bp$, see \cref{theorem:dual_representation_of_DTV}.
Consequently, these constraints are likewise linear in the important special case $s = 1$.
In any case, each coefficient of $\bp$ is constrained only once.

As a consequence of \eqref{eq:discrete_TV_abstract}, we establish that optimization problems utilizing the discrete TV-seminorm \eqref{eq:discrete_TV_for_DG} as a regularizer possess a discrete dual problem with very simple constraints.
This applies, in particular, to the famous TV-$L^2$ and TV-$L^1$ models; see \cite{RudinOsherFatemi1992} and \cite{Nikolova2004:1,EsedogluOsher2004,ChanEsedoglu2005}, respectively.
The structure of the primal and dual problems is in turn essential for the efficient implementation of appropriate solution algorithms.
As one of the main contributions of this paper, we are able to show that a variety of popular algorithms for TV-$L^2$ and TV-$L^1$, originally developed in the context of finite difference discretizations on Cartesian grids, apply with little or no changes to discretizations with low or higher-order finite elements. 
\ifthenelse{\boolean{ispreprint}}%
	{Specifically, we consider the split Bregman algorithm \cite{GoldsteinOsher2009}, the primal-dual method of \cite{ChambollePock2011}, Chambolle's projection method \cite{Chambolle2004}, a primal-dual active set method similar to \cite{HintermuellerKunisch2004:2} for TV-$L^2$ for denoising and inpainting problems, as well as the primal-dual method and the ADMM of \cite{TaoYang2009:1_preprint} for TV-$L^1$.
	}%
	{Specifically, we consider the split Bregman algorithm \cite{GoldsteinOsher2009} \added{and} the primal-dual method of \added{Chambolle and Pock} \cite{ChambollePock2011} \added{for TV-$L^2$ for denoising and inpainting problems.} 
		\added{We mention that Chambolle's projection method \cite{Chambolle2004} and a primal-dual active set method similar to \cite{HintermuellerKunisch2004:2} can also be considered; we refer the reader to the extended preprint \cite{HerrmannHerzogSchmidtVidalNunezWachsmuth2018:1_preprint} for details.}
		\added{Moreover, we discuss the primal-dual method \cite{ChambollePock2011} as well as the ADMM of \cite{TaoYang2009:1_preprint} for TV-$L^1$.}
	}
A `Huberized' version of \eqref{eq:discrete_TV_for_DG} can also be considered with minor modifications to the algorithms.

There are multiple motivations to study finite element discretizations of the TV-seminorm\added{, in imaging and beyond}.
First, finite element discretizations lend themselves in applications whenever the data is not represented on a Cartesian grid.
\added{While we focus in this paper mainly on the mathematical theory on triangular grids, we mention, for instance, that honeycombed octagonal CCD sensor layouts are in use in consumer cameras, e.g., the Fujifilm SuperCCD sensor.
	Furthermore, non-rectangular sub-pixel configurations appear to be promising for spatially varying exposure (SVE) sensors for high-dynamic-range (HDR) imaging, see \cite{LiBaiLinYu2016}, and super-resolution applications, see \cite{Ben-EzraLinWilburnZhang2011,SasaoHiuraSato2013,YueShenLiYuanZhangZhang2016}.
	Image processing problems on non-regular pixel layouts have been previously considered in~\cite{ColemanScotneyGardiner2016,Jensen2015,SugathanScariaJames2015,KnaupSteckmannBockenbachKachelriess2007}.
	Further applications of higher-order discretizations in imaging arise when the image data to be reconstructed is not a priori quantized into piecewise constant pixel values.
}

Second, \eqref{eq:continuous_TV} is popular as a regularizer in inverse coefficient problems for partial differential equations; see for instance \cite{ChanTai2004,BachmayrBurger2009,ClasonKruseKunisch2017}.
In this situation, a discretization by finite elements of both the state and the unknown coefficient is often the natural choice, in particular on non-trivial geometries.
Third, finite element discretizations generalize easily to higher-order simply by increasing the poly\added{n}omial degree.
It is well known that higher-order discretizations can outperform mesh refinement approaches when the function to be approximated is sufficiently smooth.
Finally, we anticipate that our approach can be extended to total \emph{generalized} variation (TGV) introduced in \cite{BrediesKunischPock2010} as well, \added{and imaging problems on surfaces as in \cite{LopezPerez2006,HerrmannHerzogKroenerSchmidtVidalNunez2017:1}}, although this is not the subject of the present paper.

The vast majority of all publications to date dealing with the TV-seminorm use a (lowest order) finite difference approximation of \eqref{eq:continuous_TV} on Cartesian grids, where the divergence is approximated by one-sided differences.
We are aware of only a few contributions including \cite{FengProhl2003,ElliottSmitheman2009,WuZhangDuanTai2012,Bartels2012,BartelsNochettoSalgado2014,AlkaemperLanger2017,BerkelsEfflandRumpf2017,ClasonKruseKunisch2017} using lowest-order ($r = 1$) \emph{continuous} finite elements, i.e., $u \in \CG{1}(\Omega)$.
In this case the edge jump contributions in \eqref{eq:continuous_TV_for_DG} and \eqref{eq:discrete_TV_for_DG} vanish, and since $\nabla u \in \DG{0}(\Omega)$ holds, formulas \eqref{eq:continuous_TV_for_DG} and \eqref{eq:discrete_TV_for_DG} coincide.
Moreover, the case $u \in \DG{0}(\Omega)$ on uniform, rectangular grids\added{, i.e., pixel images,} is discussed in \cite{StammWihler2015,LeeParkPark2018_preprint}.
\added{Recently, \cite{ChambollePock2018:1_preprint} proposed a different discrete approximation of the total variation over the Crouzeix-Raviart finite element space for the image data $u$, which lies in between $\DG{1}(\Omega)$ and $\CG{1}(\Omega)$.}

To the best of our knowledge, the definition of the discrete TV-seminorm \eqref{eq:discrete_TV_for_DG} as well as role of \added{the} Raviart--Thomas finite element space to establish the dual representation \eqref{eq:discrete_TV_abstract} are novel contributions of the present work.

This paper is structured as follows.
We collect some background material on finite elements in \cref{sec:Finite_Element_Spaces}.
In \cref{sec:Discrete_total_variation} we establish the dual representation \eqref{eq:discrete_TV_for_DG} of the discrete TV-seminorm \eqref{eq:discrete_TV_abstract}.
We also derive an estimate of the error between \eqref{eq:discrete_TV_for_DG} and \eqref{eq:continuous_TV_for_DG}.
We present discrete TV-$L^2$ and TV-$L^1$ models along with their duals in \cref{sec:Discrete_dual_problems}.
In \cref{sec:Algorithms} we show that \ifthenelse{\boolean{ispreprint}}{a variety of}{\added{two}} well known algorithms for TV-$L^2$ image denoising \added{and inpainting} can be applied in our (possibly higher-order) finite element setting with little or no changes compared to their classical counterparts in the \added{Cartesian} finite difference domain.
Further implementation details in the finite element framework \fenics\ are given in \cref{sec:Implementation_Details} and numerical results for TV-$L^2$ \added{denoising and inpainting} are presented in \cref{sec:Numerical_results_DTV-L2}.
In \cref{sec:DTV-L1} we briefly also consider two methods for the TV-$L^1$ case.
In \cref{sec:Extensions} we comment on extensions such as Huber regularized variants of TV-$L^2$ and TV-$L^1$\ifthenelse{\boolean{ispreprint}}{, as well as on the simplifications that apply when images belong to globally \emph{continuous} finite element spaces $\CG{r}(\Omega)$.}{\deleted{, as well as on the simplifications that apply when images belong to globally \emph{continuous} finite element spaces $\CG{r}(\Omega)$}.}
\added{Moreover, an extension to $\Omega \subset \R^3$ is discussed.}
We conclude with an outlook in \cref{sec:Conclusion_outlook}.

\subsection*{Notation}

Let $\Omega \subset \R^2$ be a bounded domain with polygonal boundary.
We denote by $L^2(\Omega)$ and $H^1(\Omega)$ the usual Lebesgue and Sobolev spaces.
$H^1_0(\Omega)$ is the subspace of $H^1(\Omega)$ of functions having zero trace on the boundary $\partial \Omega$.
The vector valued counterparts of these spaces as well as all vector valued functions will be written in bold-face notation.
Moreover, we define
\begin{equation*}
	\bH(\div;\Omega) \coloneqq \big\{ \bp \in \bL^2(\Omega): \div \bp \in L^2(\Omega) \big\}
\end{equation*}
and $\bH_0(\div;\Omega)$ is the subspace of functions having zero normal trace on the boundary, i.e., $\bp \cdot \bn = 0$.

\section{Finite Element Spaces}
\label{sec:Finite_Element_Spaces}

Suppose that $\Omega$ is triangulated by a geometrically conforming mesh (no hanging nodes) consisting of non-degenerate triangular cells $T$ and interior edges $E$.
Recall that on each interior edge, $\bn_E$ denotes the unit normal vector (of arbitrary but fixed orientation).
Throughout, $r \ge 0$ denotes the degree of certain polynomials.

\subsection*{Lagrangian Finite Elements}

Let $\PP_r(T)$ denote the space of scalar, bivariate polynomials on $T$ with total maximal degree $r$.
The dimension of $\PP_r(T)$ is $(r+1) \, (r+2) / 2$.
Let $\{\LagrangeBasisTr{k}\}$ denote the standard nodal basis of $\PP_r(T)$ with associated Lagrange nodes $\{\LagrangeNodesTr{k}\}$, $k=1,\ldots, (r+1)\,(r+2)/2$.
In other words, each $\LagrangeBasisTr{k}$ is a function in $\PP_r(T)$ satisfying $\LagrangeBasisTr{k}(\LagrangeNodesTr{k'}) = \delta_{kk'}$, see \cref{fig:illustration_of_DG_basis_functions} in \cref{sec:Examples_of_FE_basis_functions}.
We denote by
\begin{align}
	\label{eq:DGr}
	\DG{r}(\Omega) & \coloneqq \big\{ u \in L^2(\Omega): \restr{u}{T} \in \PP_r(T) \big\}, \quad r \ge 0,
	\\
	\label{eq:CGr}
	\CG{r}(\Omega) & \coloneqq \big\{ u \in C(\Omega): \restr{u}{T} \in \PP_r(T) \big\}, \quad r \ge 1,
\end{align}
the standard finite element spaces of globally discontinuous ($L^2$-conforming) or continuous ($H^1$-conforming) piecewise polynomials of degree $r$.
A finite element function $u \in \DG{r}(\Omega)$ or $\CG{r}(\Omega)$, restricted to $T$, is represented by its coefficient vector w.r.t.\ the basis $\{\LagrangeBasisTr{k}\}$, which is simply given by point evaluations.
We use the notation
\begin{equation*}
	u_{T,k}
	= 
	\restr{u}{T}(\LagrangeNodesTr{k})
\end{equation*}
to denote the elements of the coefficient vector of a function $u \in \DG{r}(\Omega)$ or $\CG{r}(\Omega)$.

Frequently we will also work with the space $\PP_{r-1}(T)$, whose standard nodal basis and Lagrange nodes we denote by $\{\LagrangeBasisTrmone{i}\}$ and $\{\LagrangeNodesTrmone{i}\}$, $i=1,\ldots, r\,(r+1)/2$.
The interpolation operator into this space (used in the definition \eqref{eq:discrete_TV_for_DG} of $\abs{u}_{DTV(\Omega)}$) is defined by 
\begin{equation*}
	\II_T \{ v \} \coloneqq \sum_{i=1}^{\mrep{r \, (r+1)/2}{r+1}} v(\LagrangeNodesTrmone{i}) \, \LagrangeBasisTrmone{i}.
\end{equation*}
Similarly, $\PP_r(E)$ denotes the space of univariate scalar polynomials on $E$ of maximal degree $r$, which has dimension $r+1$.
Let $\{\LagrangeBasisEr{j}\}$ denote the standard nodal basis of $\PP_r(E)$ with associated Lagrange nodes $\{\LagrangeNodesEr{j}\}$, $j=1,\ldots, r+1$, see \cref{fig:illustration_of_Pr_basis_functions} in \cref{sec:Examples_of_FE_basis_functions}.
The associated interpolation operator becomes 
\begin{equation*}
	\II_E \{ v \} \coloneqq \sum_{j=1}^{r+1} v(\LagrangeNodesEr{j}) \, \LagrangeBasisEr{j}.
\end{equation*}

Finally, we address the definition of the jump of a $\DG{r}(\Omega)$ function across an interior edge $E$ connecting two cells $T_1$ and $T_2$ with their respective outer normals $\bn_1$ and $\bn_2 = - \bn_1$ of unit length.
We recall that the edge normal $\bn_E$ coincides either with $\bn_1$ or $\bn_2$ and we distinguish between the
\begin{subequations}
	\label{eq:definition_of_jump}
	\begin{alignat}{2}
	\label{eq:definition_of_vector_jump}
	\text{vector-valued jump} & & \quad
	\vjump{u} 
	& 
	= 
	\restr{u}{T_1} \bn_1 + \restr{u}{T_2} \bn_2
	\\
	\label{eq:definition_of_scalar_jump}
	\text{and scalar jump} & & \quad
	\jump{u} 
	& 
	= 
	\vjump{u} \cdot \bn_E
	.
	\end{alignat}
\end{subequations}
Notice that the sign of $\jump{u}$ depends on the orientation of $\bn_E$, while $\vjump{u}$ does not.
For instance when $\bn_E = \bn_1$, then $\jump{u} \coloneqq \restr{u}{T_1} - \restr{u}{T_2}$ holds.
Moreover, we point out that $\vjump{u} = \jump{u} \, \bn_E$ holds.

\subsection*{Raviart--Thomas Finite Elements}

For $r \ge 0$, we denote by
\begin{multline}
	\label{eq:RTr+1}
	\ifthenelse{\boolean{ispreprint}}{\hfill}{}
	\RT{r+1}(\Omega) 
	\ifthenelse{\boolean{ispreprint}}{}{\\}
	\coloneqq 
	\big\{ \bp \in \bH(\div;\Omega): \restr{\bp}{T} \in \PP_r(T)^2 + \bx \, \PP_r(T) \big\}
	\ifthenelse{\boolean{ispreprint}}{\hfill}{}
\end{multline}
the ($\bH(\div;\Omega)$-conforming) Raviart--Thomas finite element space of order $r+1$.\footnote{Notice that while denote the lowest-order $\RT{}$ space by $\RT{1}$, some authors use $\RT{0}$ for this purpose.}
Moreover, $\RT{r+1}^0(\Omega)$ is the subspace of functions satisfying $\bp \cdot \bn = 0$ along the boundary of $\Omega$.
The dimension of the polynomial space on each cell is $(r+1) \, (r+3)$.
Notice that several choices of local bases for $\RT{r+1}(T)$ are described in the literature, based on either point evaluations or integral moments as degrees of freedom (dofs).
Clearly, a change of the basis does not alter the finite element space but only the representation of its members, which can be identified with their coefficient vectors w.r.t.\ a particular basis.
For the purpose of this paper, it is convenient to work with the following global degrees of freedom of integral type for $\bp \in \RT{r+1}(\Omega)$; see \cite[Ch.~3.4.1]{LoggMardalWells2012:1}:
\begin{subequations}
	\label{eq:RTr+1_dofs}
	\begin{align}
		\label{eq:RTr+1_dofs_T}
		\RTDofsT{i}(\bp)
		&
		\coloneqq
		\int_T \LagrangeBasisTrmone{i} \, \bp \, \dx, 
		\quad i = 1, \ldots, \textstyle r\,(r+1)/2,
		\\
		\label{eq:RTr+1_dofs_E}
		\RTDofsE{j}(\bp)
		&
		\coloneqq
		\int_E \LagrangeBasisEr{j} \,(\bp \cdot \bn_E) \, \ds, 
		\quad j = 1, \ldots, r+1.
	\end{align}
\end{subequations}
We will refer to \eqref{eq:RTr+1_dofs_T} as triangle-based, or interior, dofs and to \eqref{eq:RTr+1_dofs_E} as edge-based dofs.
Notice that while the edge-based dofs are scalar, the triangle-based dofs have values in $\R^2$ for notational convenience.
The global basis functions for the space $\RT{r+1}(\Omega)$ are denoted by $\RTBasisT{i}$ and $\RTBasisE{j}$, respectively.
Notice that $\RTBasisT{i}$ is $\R^{2 \times 2}$-valued.
As is the case for all finite element spaces, any dof applied to any of the basis functions evaluates to zero except
\begin{equation}
	\label{eq:RTr+1_duality_of_basis}
	\RTDofsT{i}(\RTBasisT{i'}) = \begin{psmallmatrix} 1 & 0 \\ 0 & 1 \end{psmallmatrix} \, \delta_{ii'} 
	\quad
	\text{and}
	\quad
	\RTDofsE{j}(\RTBasisE{j'}) = \delta_{jj'}.
\end{equation}
Some basis functions of type $\RTBasisE{j}$ are shown in \cref{fig:illustration_of_RT_basis_functions} in \cref{sec:Examples_of_FE_basis_functions}.
Let us emphasize that for any function $\bp \in \RT{r+1}^0(\Omega)$, the dof values \eqref{eq:RTr+1_dofs} are precisely the coefficients of $\bp$ w.r.t.\ the basis, i.e.,
\begin{equation}
	\label{eq:representation_of_p_in_RTr+1}
	\bp 
	= 
	\sum_T \sum_{i=1}^{\mrep{r \, (r+1)/2}{r+1}} \RTDofsT{i}(\bp) \, \RTBasisT{i} 
	+ 
	\sum_E \sum_{j=1}^{r+1} \RTDofsE{j}(\bp) \, \RTBasisE{j}.
\end{equation}

\subsection*{Index Conventions}

In order to reduce the notational overhead, we are going to associate specific ranges for any occurence of the indices $i$, $j$ and $k$ in the sequel:
\begin{equation*}
	\begin{aligned}
		&
		i \in \{1, \ldots, r \, (r+1)/2 \}
		\text{ as in the basis functions}
		\\
		&
		\qquad
		\text{$\LagrangeBasisTrmone{i}$ of $\PP_{r-1}(T)$ and dofs $\RTDofsT{i}$ in $\RT{r+1}(\Omega)$}
		,
		\\
		&
		j \in \{1, \ldots, r+1 \}
		\text{ as in the basis functions}
		\\
		&
		\qquad
		\text{$\LagrangeBasisEr{j}$ of $\PP_r(E)$ and dofs of $\RTDofsE{j}$ in $\RT{r+1}(\Omega)$}
		,
		\\
		&
		k \in \{1, \ldots, (r+1)(r+2)/2 \}
		\text{ as in the basis functions}
		\\
		&
		\qquad 
		\text{$\LagrangeBasisTr{k}$ of $\PP_r(T)$}
		.
	\end{aligned}
\end{equation*}
For instance, \eqref{eq:representation_of_p_in_RTr+1} will simply be written as 
\begin{equation*}
	\bp = \sum_{T,i} \RTDofsT{i}(\bp) \, \RTBasisT{i} + \sum_{E,j} \RTDofsE{j}(\bp) \, \RTBasisE{j} 
\end{equation*}
in what follows.
For convenience, we summarize the notation for the degrees of freedom and basis functions needed throughout the paper in \cref{tab:spaces_and_basis_functions}.

\begin{table*}[htbp]
	\centering
	\begin{tabular}{@{}lllll@{}}
		\toprule
		FE space           & local dimension & dofs                                       & basis functions               & global dimension \\
		\midrule
		$\CG{r}(\Omega)$   & $(r+1)(r+2)/2$  & eval.\ in $\LagrangeNodesTr{k}$            & $\{\LagrangeBasisTr{k}\}$     & $N_T \, (r-2)^+ (r-1)/2$
		\\
		($r \ge 1$)        &                 &                                            &                               & ${}+ N_E \, (r-1)^+ + N_V$
		\\
		\midrule
		$\DG{r}(\Omega)$   & $(r+1)(r+2)/2$  & eval.\ in $\LagrangeNodesTr{k}$            & $\{\LagrangeBasisTr{k}\}$     & $N_T \, (r+1)(r+2)/2$ 
		\\
		\midrule
		$\DG{r-1}(\Omega)$ & $r\,(r+1)/2$    & eval.\ in $\LagrangeNodesTrmone{i}$        & $\{\LagrangeBasisTrmone{i}\}$ & $N_T \, r\,(r+1)/2$
		\\
		\midrule
		$\DG{r}(\cup E)$   & $r+1$           & eval.\ in $\LagrangeNodesEr{j}$            & $\{\LagrangeBasisEr{j}\}$     & $N_E \, (r+1)$
		\\
		\midrule
		$\RT{r+1}^0(\Omega)$ & $(r+1)(r+3)$  & $\RTDofsT{i}$, see \eqref{eq:RTr+1_dofs_T} & $\{\RTBasisT{i}\}$            & $N_T \, r\,(r+1)$ 
		\\
		                   &                 & $\RTDofsE{j}$, see \eqref{eq:RTr+1_dofs_E} & $\{\RTBasisE{j}\}$            & ${}+N_E \, (r+1)$
		\\
		\midrule 
		\bottomrule \\
	\end{tabular}
	\caption{Finite element spaces, their degrees of freedom and corresponding bases. Here $N_T$, $N_E$ and $N_V$ denote the number of triangles, interior edges and vertices in the triangular mesh. A term like $(r-a)^+$ should be understood as $\max\{r-a,0\}$.}
	\label{tab:spaces_and_basis_functions}
\end{table*}

\section{Properties of the Discrete Total Variation}
\label{sec:Discrete_total_variation}

In this section we investigate the properties of the discrete total variation seminorm
\begin{multline*}
	\ifthenelse{\boolean{ispreprint}}{\hfill}{}
	\abs{u}_{DTV(\Omega)} 
	\ifthenelse{\boolean{ispreprint}}{}{\\}
	\coloneqq
	\sum_T \int_T \II_T \big\{ \abs{\nabla u}_s \big\} \, \dx
	+
	\sum_E \int_E \II_E \big\{ \bigabs{\vjump{u}}_s \big\} \, \ds
	\ifthenelse{\boolean{ispreprint}}{\hfill}{}
\end{multline*}
for functions $u \in \DG{r}(\Omega)$.
Recall that $\II_T$ and $\II_E$ are local interpolation operators into the polynomial spaces $\PP_{r-1}(T)$ and $\PP_r(E)$, respectively.
In terms of the Lagrangian bases $\{\LagrangeBasisTrmone{i}\}$ and $\{\LagrangeBasisEr{j}\}$ of these spaces, we have
\begin{subequations}
	\label{eq:interpolation_terms_in_DTV}
	\begin{align}
		\int_T \II_T \big\{ \abs{\nabla u}_s \big\} \, \dx
		&
		= 
		\sum_{i=1}^{\mrep{r \, (r+1)/2}{r+1}} \; \bigabs{\nabla u(\LagrangeNodesTrmone{i})}_s \, c_{T,i}
		,
		\label{eq:interpolation_terms_in_DTV_T}
		\\
		\int_E \II_E \big\{ \bigabs{\vjump{u}}_s \big\} \, \ds
		&
		= 
		\sum_{j=1}^{r+1} \; \bigabs{\jump{u}(\LagrangeNodesEr{j})} \, \abs{\bn_E}_s \, c_{E,j},
		\label{eq:interpolation_terms_in_DTV_E}
	\end{align}
\end{subequations}
where the weights are given by
\begin{equation}
	\label{eq:dual_representation_of_DTV_upper_bounds}
	c_{T,i} \coloneqq \int_T \LagrangeBasisTrmone{i} \, \dx
	\quad
	\text{and}
	\quad
	c_{E,j} \coloneqq \int_E \LagrangeBasisEr{j} \, \ds.
\end{equation}
\Cref{fig:illustration_of_edge_contributions_TV_vs_DTV} provides an illustration of the difference between the contributions
\begin{align*}
\int_E \bigabs{\vjump{u}}_s \, \ds \text{ and } \int_E \II_E \big\{ \bigabs{\vjump{u}}_s \big\} \, \ds
\end{align*}
to $\abs{u}_{TV(\Omega)}$ and $\abs{u}_{DTV(\Omega)}$.

\begin{figure*}[htp]
	\centering
	\includegraphics[height=45mm,width=50mm]{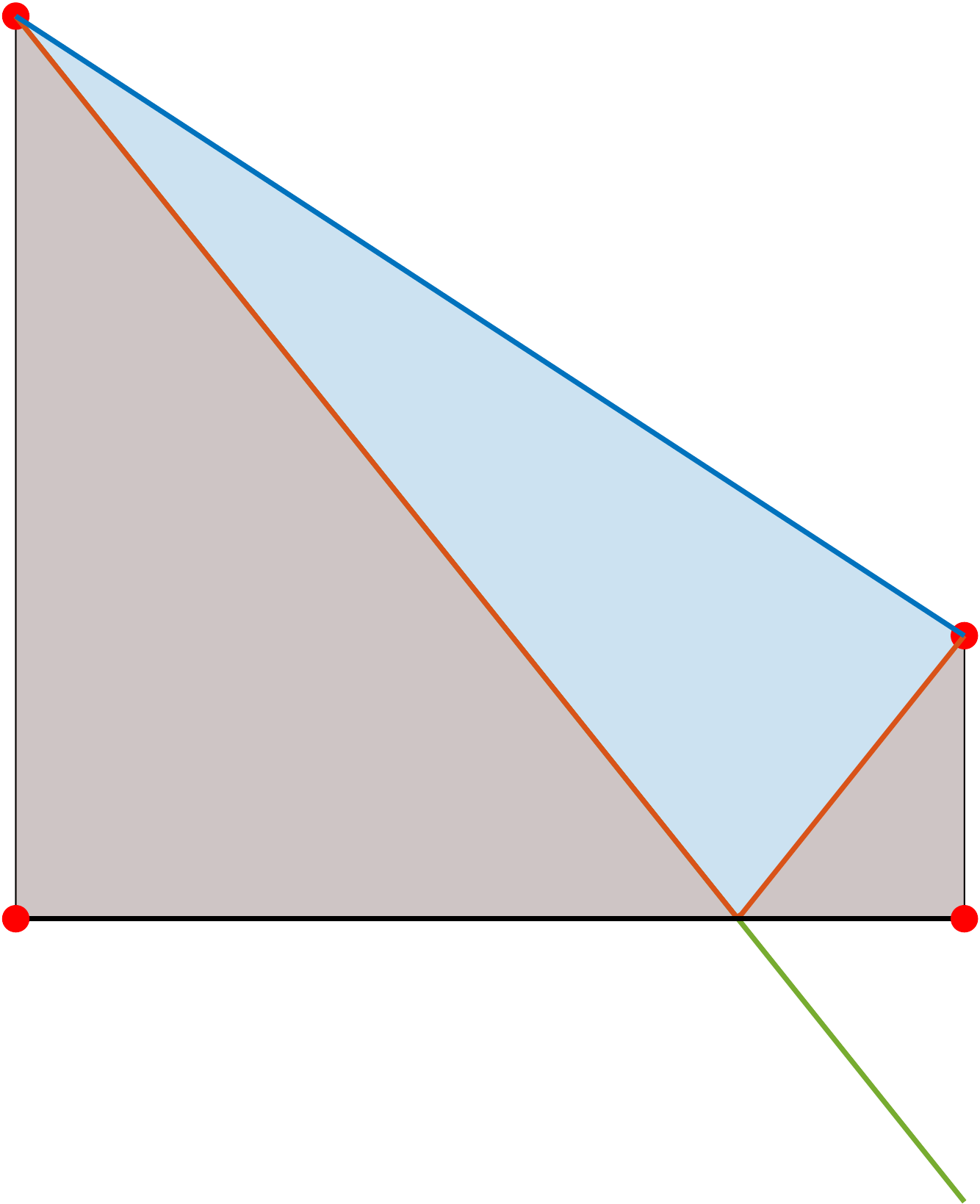}
	\quad
	\includegraphics[height=45mm,width=50mm]{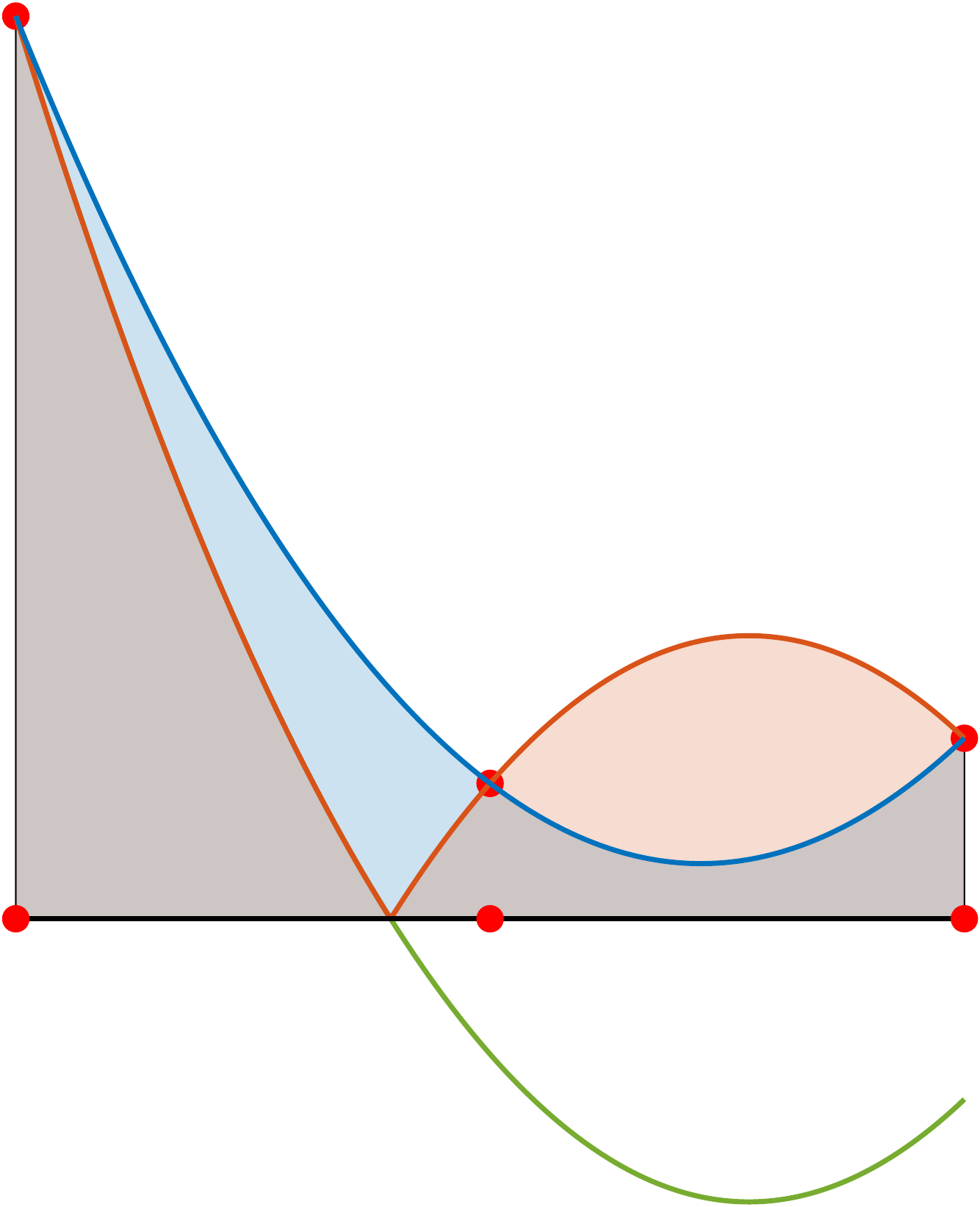}
	\caption{Illustration of typical edge-jump contributions to $\abs{u}_{TV(\Omega)}$ and to $\abs{u}_{DTV(\Omega)}$. The green and red curves show $\jump{u}$ and $\abs{\jump{u}}$, respectively, and the blue curve shows $\II_E \big\{ \abs{\jump{u}} \big\}$ for polynomial degrees $r = 1$ (left) and $r = 2$ (right). The \added{left} picture also confirms $\abs{u}_{TV(\Omega)} \le \abs{u}_{DTV(\Omega)}$ when $r = 1$, see \cref{cor:error_estimate_DTV}, while $\abs{u}_{TV(\Omega)}$ may be larger or smaller than $\abs{u}_{DTV(\Omega)}$ when $r \in \{2,3,4\}$.}
	\label{fig:illustration_of_edge_contributions_TV_vs_DTV}
\end{figure*}

In virtue of the fact that $\restr{\nabla u}{T} \in \PP_{r-1}(T)^2$ and $\jump{u} \in \PP_r(E)$, it is clear that $\abs{\,\cdot\,}_{DTV(\Omega)}$ is indeed a seminorm on $\DG{r}(\Omega)$, provided that all weights $c_{T,i}$ and $c_{E,j}$ are non-negative.
The following lemma shows that this is the case for polynomial degrees $0 \le r \le 4$.

\begin{lemma}[Lagrange basis functions with positive integrals]
	\label{lemma:basis_functions_positive_integrals}
	\begin{enumerate}[label=$(\alph*)$]
		\item 
			\label{item:closed_Newton-Cotes_T}
			Let $T \subset \R^2$ be a triangle and $1 \le r \le 4$.
			Then $c_{T,i} \ge 0$ holds for all $i = 1, \ldots, r \, (r+1)/2$.
			When $r \neq 3$, then all $c_{T,i} > 0$.
		\item 
			\label{item:closed_Newton-Cotes_E}
			Let $E \subset \R^2$ be an edge and $0 \le r \le 7$.
			Then $c_{E,j} > 0$ holds for all $j = 1, \ldots, r+1$.
	\end{enumerate}
\end{lemma}
\begin{proof}
	Given that the Lagrange points form a unifom lattice on either $T$ or $E$, the values of $c_{T,i}$ and $c_{E,j}$ are precisely the integration weights of the closed Newton--Cotes formulas.
	For triangles, these weights are tabulated, e.g., in \cite[Tab.~I]{Silvester1970} for orders $0 \le r \le 8$, and they confirm \cref{item:closed_Newton-Cotes_T}.
	For edges (intervals), we refer the reader to, e.g., \cite[Ch.~2.5]{DavisRabinowitz1984} or \cite[Ch.~5.1.5]{DahlquistBjoerck2008}, which confirms \cref{item:closed_Newton-Cotes_E}.
	\qed
\end{proof}

We can now prove the precise form of the dual representation \eqref{eq:discrete_TV_abstract} of the discrete TV-seminorm \eqref{eq:discrete_TV_for_DG}.
\begin{theorem}[Dual Representation of $\abs{u}_{DTV(\Omega)}$]
	\label{theorem:dual_representation_of_DTV}
	Suppose $0 \le r \le 4$.
	Then for any $u \in \DG{r}(\Omega)$, the discrete TV-seminorm \eqref{eq:discrete_TV_for_DG} satisfies
	\begin{align}
		\MoveEqLeft
		\abs{u}_{DTV(\Omega)} 
		=
		\sup \Bigg\{ 
			\int_\Omega u \div \bp \, \dx: \bp \in \RT{r+1}^0(\Omega),
			\notag
			\\
			&
			\abs{\RTDofsT{i}(\bp)}_{s^*} \le c_{T,i} \text{ for all $T$, } i = 1, \ldots, r \, (r+1)/2,
			\notag
			\\
			&
			\abs{\RTDofsE{j}(\bp)} \le \abs{\bn_E}_s \, c_{E,j} \text{ for all $E$, } j = 1, \ldots, r+1 
		\Bigg\}
		.
		\label{eq:dual_representation_of_DTV}
	\end{align}
\end{theorem}
\begin{proof}
	We begin with the observation that integration by parts yields
	\begin{multline}
		\ifthenelse{\boolean{ispreprint}}{\hfill}{}
		\label{eq:dual_representation_of_DTV_1}
		-
		\int_\Omega u \div \bp \, \dx
		= 
		-
		\sum_T \int_T u \div \bp \, \dx
		\ifthenelse{\boolean{ispreprint}}{}{\\}
		=
		\sum_T \int_T \nabla u \cdot \bp \, \dx
		+
		\sum_E \int_E \jump{u} \, (\bp \cdot \bn_E) \, \ds
		\ifthenelse{\boolean{ispreprint}}{\hfill}{}
	\end{multline}
	for any $u \in \DG{r}(\Omega)$ and $\bp \in \RT{r+1}^0(\Omega)$, i.e., $\bp \cdot \bn = 0$ on the boundary $\partial \Omega$.

	Let us consider one of the edge integrals first.
	Notice that $\jump{u} \in \PP_r(E)$ holds and thus $\jump{u} = \sum_j v_j \, \LagrangeBasisEr{j}$ with coeffients $v_j = \jump{u}(\LagrangeNodesEr{j})$.
	By the duality property \eqref{eq:RTr+1_duality_of_basis} of the basis of $\RT{r+1}(\Omega)$,  we obtain
	\begin{multline*}
		\ifthenelse{\boolean{ispreprint}}{\hfill}{}
		\int_E \jump{u} \, (\bp \cdot \bn_E) \, \ds
		\ifthenelse{\boolean{ispreprint}}{}{\\}
		= 
		\sum_j v_j \int_E \LagrangeBasisEr{j} \, (\bp \cdot \bn_E) \, \ds
		= 
		\sum_j v_j \, \RTDofsE{j}(\bp).
		\ifthenelse{\boolean{ispreprint}}{\hfill}{}
	\end{multline*}
	The maximum of this expression w.r.t.\ $\bp$ verifying the constraints in \eqref{eq:dual_representation_of_DTV} is attained when $$\RTDofsE{j}(\bp) = \sgn(v_j) \, \abs{\bn_E}_s \, c_{E,j}$$ holds.
	Here we are using the fact that $c_{E,j} > 0$ holds; see \cref{lemma:basis_functions_positive_integrals}.
	Choosing $\bp$ as the maximizer yields
	\begin{multline*}
		\ifthenelse{\boolean{ispreprint}}{\hfill}{}
		\int_E \jump{u} \, (\bp \cdot \bn_E) \, \ds
		= 
		\sum_j \abs{v_j} \, \abs{\bn_E}_s \, c_{E,j} 
		\ifthenelse{\boolean{ispreprint}}{}{\\}
		= 
		\sum_j \int_E \abs{v_j} \, \LagrangeBasisEr{j} \, \abs{\bn_E}_s \, \ds
		= 
		\int_E \II_E \big\{ \bigabs{\vjump{u}}_s \big\} \, \ds,
		\ifthenelse{\boolean{ispreprint}}{\hfill}{}
	\end{multline*}
	where we used $\abs{v_j} = \bigabs{\jump{u}(\LagrangeNodesEr{j})} = \bigabs{\jump{u}}(\LagrangeNodesEr{j})$ and thus $\abs{v_j} \, \abs{\bn_E}_s = \bigabs{\vjump{u}}_s(\LagrangeNodesEr{j})$ in the last step.

	Next we consider an integral over a triangle, which is relevant only when $r \ge 1$.
	Since $u \in \PP_r(T)$ holds, we have $\nabla u \in \PP_{r-1}(T)^2$ and thus $\nabla u = \sum_i \LagrangeBasisTrmone{i} \, \bw_i$ with vector-valued coefficients $\bw_i = \nabla u(\LagrangeNodesTrmone{i})$.
	Using again the duality property \eqref{eq:RTr+1_duality_of_basis} of the basis of $\RT{r+1}(\Omega)$, we obtain
	\begin{equation*}
		\int_T \nabla u \cdot \bp \, \dx
		=
		\sum_i \bw_i \cdot \int_T \LagrangeBasisTrmone{i} \, \bp \, \dx
		= 
		\sum_i \bw_i \cdot \RTDofsT{i}(\bp).
	\end{equation*}
	By virtue of Hölder's inequality, the maximum of this expression w.r.t.\ $\bp$ verifying the constraints in \eqref{eq:dual_representation_of_DTV} can be characterized explicitly.
	When $\bw_i \neq \bnull$ and $1 \le s < \infty$, then the maximum is attained when
	\begin{equation*}
		\RTDofsT{i}(\bp) 
		=
		\begin{pmatrix}
			(\sgn \bw_{i,1}) \, \abs{\bw_{i,1}}^{s-1}
			\\
			(\sgn \bw_{i,2}) \, \abs{\bw_{i,2}}^{s-1}
		\end{pmatrix}
		\frac{c_{T,i}}{\abs{\bw_i}_s^{s-1}}.
	\end{equation*}
	Similarly, in case $\bw_i \neq \bnull$ and $s = \infty$, we choose
	\begin{equation*}
		\RTDofsT{i}(\bp) 
		=
		\begin{cases}
			c_{T,i} \, (\sgn \bw_{i,\ell}) & \text{for exactly one component} 
			\ifthenelse{\boolean{ispreprint}}{}{\\ &}
			\text{$\ell \in \{1,2\}$ s.t.\ } \abs{\bw_{i,\ell}} = \abs{\bw_i}_\infty, \\
			0 & \text{otherwise}.
		\end{cases}
	\end{equation*}
	When $\bw_i = \bnull$ holds, $\RTDofsT{i}(\bp)$ can be chosen arbitrarily but subject to $\abs{\RTDofsT{i}(\bp)}_{s^*} \le c_{T,i}$.
	In any case, we arrive at the optimal value $\bw_i \cdot \RTDofsT{i}(\bp) = c_{T,i} \, \abs{\bw_i}_s$.
	As before, we are using here the fact that $c_{T,i} \ge 0$ holds; see again \cref{lemma:basis_functions_positive_integrals}.
	For an optimal $\bp$, we thus have 
	\begin{multline*}
		\ifthenelse{\boolean{ispreprint}}{\hfill}{}
		\int_T \nabla u \cdot \bp \, \dx
		= 
		\sum_i \abs{\bw_i}_s \, c_{T,i} 
		\ifthenelse{\boolean{ispreprint}}{}{\\}
		= 
		\sum_i \int_T \abs{\bw_i}_s \, \LagrangeBasisTrmone{i} \, \dx
		= 
		\int_T \II_T \big\{ \abs{\nabla{u}}_s \big\} \, \dx,
		\ifthenelse{\boolean{ispreprint}}{\hfill}{}
	\end{multline*}
	where we used $\abs{\bw_i}_s = \abs{\nabla u(\LagrangeNodesTrmone{i})}_s = \abs{\nabla u}_s(\LagrangeNodesTrmone{i})$ in the last step.

	Finally, we point out that each summand in \eqref{eq:dual_representation_of_DTV_1} depends on $\bp$ only through the dof values $\RTDofsT{i}(\bp)$ or $\RTDofsE{j}(\bp)$ associated with one particular triangle or edge.
	Consequently, the maximum of \eqref{eq:dual_representation_of_DTV_1} is attained if and only if each summand attains its maximum subject to the constraints on the dof values set forth in \eqref{eq:dual_representation_of_DTV}.
	Since $-\bp$ verifies the same constraints as $\bp$, the maxima over $\pm \int_\Omega u \div \bp \, \dx$ coincide and \eqref{eq:dual_representation_of_DTV} is proved.
	\qed
\end{proof}

\begin{remark}[The lowest-order case $r = 0$] 
	\label{remark:dual_representation_of_DTV}
	In the lowest-order case $r = 0$, the only basis function on any interior edge $E$ is $\LagrangeBasisEr{1} \equiv 1$ so that $c_{E,1} = \abs{E}$ holds.
	Consequently, \eqref{eq:dual_representation_of_DTV} reduces to \eqref{eq:discrete_TV_r=0}.
\end{remark}

It may appear peculiar that the constraints for the edge dofs in \eqref{eq:dual_representation_of_DTV} are scalar and linear, while the constraints for the pairwise triangle dofs $\RTDofsT{i}(\bp) \in \R^2$ are generally nonlinear.
Notice, however, that it becomes evident in the proof of \cref{theorem:dual_representation_of_DTV} that the edge dofs are utilized to measure the contributions in $\abs{u}_{DTV(\Omega)}$ associated with the edge jumps of $u$, while the triangle dofs account for the contributions attributed to the gradient $\nabla u$.
Since the edge jumps are maximal in the direction normal to the edge, scalar dofs suffice in order to determine the unknown jump height.
On the other hand, both the norm and direction of the gradient are unknown and must be recovered from integration against suitable functions $\bp$.
To this end, a variation of $\RTDofsT{i}(\bp)$ within a two-dimensional ball (w.r.t.\ the $\abs{\,\cdot\,}_{s^*}$-norm) is required, leading to constraints $\abs{\RTDofsT{i}(\bp)}_{s^*} \le c_{T,i}$ on pairs of coefficients of $\bp$.
Notice that those constraints appear for polynomial degrees $r \ge 1$ and they are nonlinear unless $s^* \in \{1,\infty\}$, which correspond to variants of the TV-seminorm with maximal anisotropy; compare \cref{fig:illustration_of_anisotropy_DG0}.

We conclude this section by comparing the TV-seminorm \eqref{eq:continuous_TV_for_DG} with our discrete variant \eqref{eq:discrete_TV_for_DG} for $\DG{r}(\Omega)$ functions.
For the purpose of the following result, let us denote by $\jump{u}'$ the tangential derivative (in arbitrary direction of traversal) of the scalar jump of $u$ along an edge~$E$.
The symbol
\begin{multline*}
	\ifthenelse{\boolean{ispreprint}}{\hfill}{}
	\abs{u}_{W^{2,\infty}(T)}
	=
	\max \Big\{
		\max_{x \in T} \left\{ \abs{u_{x_1x_1}(x)} \right\}
		, \; 
		\ifthenelse{\boolean{ispreprint}}{}{\\}
		\max_{x \in T} \left\{ \abs{u_{x_1x_2}(x)} \right\}
		, \; 
		\max_{x \in T} \left\{ \abs{u_{x_2x_2}(x)} \right\}
	\Big\}
	\ifthenelse{\boolean{ispreprint}}{\hfill}{}
\end{multline*}
is the $W^{2,\infty}$-seminorm of $u$ on $T$.
Moreover, we recall that the aspect ratio $\gamma_T = h_T/\varrho_T$ of a triangle $T$ is the ratio between its diameter (longest edge) $h_T$ and the diameter $\varrho_T$ of the maximal inscribed circle; see for instance \cite[Definition~1.107]{ErnGuermond2004}.

\begin{proposition}
	\label{prop:error_estimate_DTV}
	There is a constant $C > 0$
	such that
	\begin{multline}
		\ifthenelse{\boolean{ispreprint}}{\hfill}{}
		\label{eq:error_estimate_TV_and_DTV}
		\bigabs{
			\abs{u}_{TV(\Omega)} 
			-
			\abs{u}_{DTV(\Omega)} 
		}
		\ifthenelse{\boolean{ispreprint}}{}{\\}
		\le
		C \, h \,
		\Bigl(
			\max_T \, \abs{u}_{W^{2,\infty}(T)}
			+
			\sum_E \, \bignorm{ \jump{u}' }_{L^1(E)}
		\Bigr)
		\ifthenelse{\boolean{ispreprint}}{\hfill}{}
	\end{multline}
	holds for all $u \in \DG{r}(\Omega)$,
	$0 \le r \le 4$, where $h \coloneqq \max_T h_T$ is the mesh size.
	The constant $C$ depends only on $r$, $s$, the maximal aspect ratio $\max_T \gamma_T$ and the area $\abs{\Omega}$.
\end{proposition}
\begin{proof}
	We use \eqref{eq:interpolation_terms_in_DTV} to interpret
	the discrete TV-semi\-norm
	as a quadrature rule applied to
	the TV-seminorm \eqref{eq:continuous_TV_for_DG}.
	\added{%
		Note that no volume terms appear in the piecewise constant case $r = 0$.
		In case $r \ge 1$,
		we use \cite[Lem.~8.4]{ErnGuermond2004}
		with $d = 2$, $p = \infty$, $k_q = 0$, and $s = 1$ therein,
		for the volume terms in \eqref{eq:interpolation_terms_in_DTV_T}.
		This result yields the existence of a constant $C > 0$
		such that
	}
	\begin{equation*}
		\added{%
			\biggabs{
				\int_T v \, \dx
				-
				\sum_i \, v(\LagrangeNodesTrmone{i}) \, c_{T,i}
			}
			\le
			C \, h_T^3 \, \abs{v}_{W^{1,\infty}(T)}
		}
	\end{equation*}
	\added{%
		holds for all $v \in W^{1,\infty}(T)$.
		Using this estimate for
		$v = \abs{\nabla u}_s$
		shows
	}
	\begin{align*}
		\biggabs{
			\int_T \abs{\nabla u}_s \, \dx
			-
			\sum_i \, \bigabs{\nabla u(\LagrangeNodesTrmone{i})}_s& \, c_{T,i}
		}
		\ifthenelse{\boolean{ispreprint}}{}{\\}
		\le
		\ifthenelse{\boolean{ispreprint}}{}{&}
		C \, h_T^3 \, \bigabs{\abs{\nabla u}_s}_{W^{1,\infty}(T)}.
	\end{align*}
	(During the proof, $C$ denotes a generic constant which may change from instance to instance.)
	Summing over $T$ and using $\sum_T h_T^2 \le C$ (depending on $\abs{\Omega}$ and the maximal aspect ratio $\max_T \gamma_T$),
	we find
	\begin{align*}
		\MoveEqLeft
		\sum_T \,
		\biggabs{
			\int_T \Bigl( \abs{\nabla u}_s - \II_T \big\{ \abs{\nabla u}_s \big\} \Bigr) \, \dx
		}
		\\
		& 
		=
		\sum_T \,
		\biggabs{
			\int_T \abs{\nabla u}_s \, \dx
			-
			\sum_i \, \bigabs{\nabla u(\LagrangeNodesTrmone{i})}_s \, c_{T,i}
		}
		\\
		&
		\le
		C \, h \, \max_T \, \bigabs{\abs{\nabla u}_s}_{W^{1,\infty}(T)}.
	\end{align*}
	Since $\bv \mapsto \abs{\bv}_s$ is globally Lipschitz continuous,
	we find that
	\begin{equation*}
		\max_T \, \bigabs{\abs{\nabla u}_s}_{W^{1,\infty}(T)}
		\le
		C
		\,
		\max_T \, \abs{u}_{W^{2,\infty}(T)}.
	\end{equation*}

	Similarly, for each edge $E$,
	we will apply \cite[Lem.~8.4]{ErnGuermond2004}
	\added{in \eqref{eq:interpolation_terms_in_DTV_E}}
	(using $d = 1$, $p = 1$, $k_q = 0$, and $s = 1$ therein); note that the proof carries
	over to this limit case with $p = 1$ and $d = s$.
	\added{This implies the existence of $C > 0$ such that}
	\begin{equation*}
		\added{%
			\biggabs{
				\int_E v \, \ds
				-
				\sum_j \, v(\LagrangeNodesEr{j}) \, c_{E,j}
			}
			\le
			C \, h \, \norm{ v' }_{L^1(E)}
		}
	\end{equation*}
	\added{%
		holds for all $v \in W^{1,1}(E)$,
		where $v'$ denotes the tangential derivative of $v$.
		Using $v = \bigabs{\jump{u}}$
		yields the estimate
	}
	\begin{equation*}
		\biggabs{
			\int_E \bigabs{\jump{u}} \, \ds
			-
			\sum_j \, \bigabs{\jump{u}(\LagrangeNodesEr{j})} \, c_{E,j}
		}
		\le
		C \, h \, \bignorm{ \abs{\jump{u}}' }_{L^1(E)}.
	\end{equation*}
	Here, $\abs{\jump{u}}'$ is the tangential derivative of the
	absolute value of the jump of $u$ on $E$.
	Notice that
	$\bignorm{ \abs{\jump{u}}' }_{L^1(E)} = \bignorm{ \jump{u}' }_{L^1(E)}$ holds.
	Summing over $E$ yields
	\begin{align*}
		\MoveEqLeft
		\sum_E \,
		\biggabs{
			\int_E \bigabs{\jump{u}} - \II_E \big\{ \bigabs{\jump{u}} \big\} \, \ds
		}
		\\
		&
		=
		\sum_E \,
		\biggabs{
			\int_E \bigabs{\jump{u}} \, \ds
			-
			\sum_j \, \bigabs{\jump{u}(\LagrangeNodesEr{j})} \, c_{E,j}
		}
		\\
		&
		\le
		C \, h \, \sum_E \, \bignorm{ \jump{u}' }_{L^1(E)}.
	\end{align*}
	By using $\bigabs{\vjump{u}}_s = \abs{\jump{u}} \, \abs{\bn_E}_s$ on each edge,
	and combining the above estimates,
	we obtain the announced error bound.
	\qed
\end{proof}
\begin{corollary}[Low Order Polynomial Degrees]
	\label{cor:error_estimate_DTV}
	\begin{enumerate}[label=$(\alph*)$]
		\item
			When $r = 0$, we have
			$\abs{u}_{TV(\Omega)} = \abs{u}_{DTV(\Omega)}$
			for all $u \in \DG{r}(\Omega)$.
		\item 
			When $r = 1$, then
			$\abs{u}_{TV(\Omega)} \le \abs{u}_{DTV(\Omega)}$
			for all $u \in \DG{r}(\Omega)$.
	\end{enumerate}
\end{corollary}
\begin{proof}
	In case $r = 0$, the right-hand side of the estimate in \cref{prop:error_estimate_DTV}
	vanishes.
	In case $r = 1$, $\nabla u$ is piecewise constant and the corresponding terms in \eqref{eq:continuous_TV_for_DG} and \eqref{eq:discrete_TV_for_DG} coincide. 
	Moreover, for affine functions $v : E \to \R$ it is easy to check that
	\begin{equation*}
		\int_E \abs{v} \, \ds
		\le
		\frac12 \, \Bigl( \bigabs{v(\LagrangeNodesEr{1})} + \bigabs{v(\LagrangeNodesEr{2})} \Bigr) \, \int_E 1 \, \ds,
	\end{equation*}
	where $\LagrangeNodesEr{1}$ and $\LagrangeNodesEr{2}$ are the two end points of $E$.
	This yields the claim in case $r = 1$.
	\qed
\end{proof}

We also mention that the boundary perimeter formula
\begin{equation*}
	\operatorname{Per}(E)
	\coloneqq
	\abs{\chi_E}_{TV(\Omega)} 
	=
	\abs{\chi_E}_{DTV(\Omega)} 
	=
	\operatorname{length}(E) 
\end{equation*}
holds when $E$ is a union of triangles and thus the characteristic function $\chi_E$ belongs to $\DG{0}(\Omega)$.

\section{Discrete Dual Problems}
\label{sec:Discrete_dual_problems}

In this section we revisit the classical image denoising \added{and inpainting} problems, 
\begin{align}
	\label{eq:TV-L2}
	\tag{TV-L2}
	\text{Minimize} \quad & \frac{1}{2} \norm{u-f}_{L^2(\added{\Omega_0})}^2 + \beta \, \abs{u}_{TV(\Omega)},
	\\
	\label{eq:TV-L1}
	\tag{TV-L1}
	\text{Minimize} \quad & \norm{u-f}_{L^1(\added{\Omega_0})} + \beta \, \abs{u}_{TV(\Omega)},
\end{align}
see \cite{RudinOsherFatemi1992,Nikolova2004:1,EsedogluOsher2004,ChanEsedoglu2005,ChambolleCasellesCremersNovagaPock2010}.
We introduce their discrete counterparts and establish their Fenchel duals.
\added{Here $\Omega_0 \subset \Omega$ is the domain where data is available, and $\beta$ is a positive parameter. For simplicity, we assume that the inpainting region $\Omega \setminus \Omega_0$ is the union of a number of triangles in the discrete problems.}

\subsection{The TV-$L^2$ Problem}
\label{subsec:Discrete_dual_problem_TV-L2}

The discrete counterpart of \eqref{eq:TV-L2} we consider is
\begin{equation}
	\label{eq:DTV-L2}
	\tag{DTV-L2}
	\text{Minimize} \quad \frac{1}{2} \norm{u-f}_{L^2(\added{\Omega_0})}^2 + \beta \, \abs{u}_{DTV(\Omega)}.
\end{equation}
The reconstructed image $u$ is sought in $\DG{r}(\Omega)$ for some $0 \le r \le 4$.
We can assume that the given data $f$ belongs to $\DG{r}(\added{\Omega_0})$ as well, possibly after applying interpolation or quasi-interpolation.
Notice that we use the discrete TV-seminorm as regularizer.

The majority of algorithms considered in the literature utilize either the primal or the dual formulations of the problems at hand.
The \emph{continuous} (pre-)dual problem for \eqref{eq:TV-L2} is well known, see for instance \cite{HintermuellerKunisch2004:2}:
\begin{equation}
	\label{eq:dualTV-L2}
	\tag{TV-L2-D}
	\begin{aligned}
		\text{Minimize} \quad & \frac{1}{2} \norm{\div \bp + f}_{L^2(\added{\Omega_0})}^2 
		\\
		\text{s.t.\ } \quad & \abs{\bp}_{s^*} \le \beta, 
	\end{aligned}
\end{equation}
with $\bp \in \bH_0(\div;\Omega)$. 
Our first result in this section shows that the dual of the discrete problem \eqref{eq:DTV-L2} has a very similar structure as \eqref{eq:dualTV-L2}, but with the pointwise constraints replaced by coefficient-wise constraints\added{ as in \eqref{eq:dual_representation_of_DTV}}. 
For future reference, we denote the \added{associated admissible set by} 
\begin{align}
	\ifthenelse{\boolean{ispreprint}}{}{\MoveEqLeft}
	\bP 
	\coloneqq 
	\Big\{ 
		\bp \in \RT{r+1}^0(\Omega): 
	\Big.
	\notag
	\ifthenelse{\boolean{ispreprint}}{}{\\}
		& 
		\abs{\RTDofsT{i}(\bp)}_{s^*} \le c_{T,i} \text{ for all $T$ and all } i, 
		\notag
		\\
		&
		\abs{\RTDofsE{j}(\bp)} \le \abs{\bn_E}_s \, c_{E,j} \text{ for all $E$ and all } j 
	\Big\}.
	\label{eq:constraint_set}
\end{align}

\begin{theorem}[Discrete dual problem for \eqref{eq:DTV-L2}]
	\label{theorem:discrete_dual_problem_TV-L2}
	Let $0 \le r \le 4$.
	Then the dual problem of \eqref{eq:DTV-L2} is
	\begin{equation}
		\label{eq:dualDTV-L2}
		\tag{DTV-L2-D}
		\text{Minimize} \quad \frac{1}{2} \norm{\div \bp + f}_{L^2(\added{\Omega_0})}^2 \text{ s.t.\ } \bp \in \beta \bP.
	\end{equation}
\end{theorem}
\added{Here $\bp \in \beta \bP$ means that $\bp$ satisfies constraints as in \eqref{eq:constraint_set} but with $c_{T,i}$ and $c_{E,j}$ replaced by $\beta \, c_{T,i}$ and $\beta \, c_{E,j}$, respectively.}
\begin{proof}
	We cast \eqref{eq:DTV-L2} in the common form $F(u) + \beta \, G(\Lambda u)$.
	Let us define $U \coloneqq \DG{r}(\Omega)$ and $F(u) \coloneqq \frac{1}{2} \norm{u-f}_{L^2(\added{\Omega_0})}^2$.
	The operator $\Lambda$ represents the gradient of $u$, which consists of the triangle-wise contributions plus measure-valued contributions due to (normal) edge jumps.
	We therefore define
	\begin{subequations}
		\label{eq:definition_of_Lambda}
		\begin{equation}
			\label{eq:definition_of_Lambda_1}
			\Lambda: U \to Y \coloneqq \prod_T \PP_{r-1}(T)^2 \times \prod_E \PP_r(E).
		\end{equation}
		The components of $\Lambda u$ will be addressed by $(\Lambda u)_T$ and $(\Lambda u)_E$ respectively, and they are defined by
		\begin{equation}
			\label{eq:definition_of_Lambda_2}
			(\Lambda u)_T \coloneqq \restr{\nabla u}{T}
			\quad \text{and} \quad
			(\Lambda u)_E \coloneqq \jump{u}_E.
		\end{equation}
	\end{subequations}
	Finally, the function $G: Y \to \R$ is defined by
	\begin{align}
		G(\bd) 
		\coloneqq
		\sum_T \int_T &\II_T \big\{ \abs{\bd_T}_s \big\} \, \dx 
		\nonumber
		\ifthenelse{\boolean{ispreprint}}{}{\\}
		+
		\ifthenelse{\boolean{ispreprint}}{}{&}
		\sum_E \abs{\bn_E}_s \int_E \II_E \big\{ \abs{d_E} \big\} \, \ds.
		\label{eq:definition_of_G}
	\end{align}
	A crucial observation now is that the dual space $Y^*$ of $Y$ can be identified with $\RT{r+1}^0(\Omega)$ when the duality product is defined as
	\begin{equation}
		\label{eq:duality_product}
		\dual{\bp}{\bd}
		\coloneqq
		\sum_T \int_T \bp \cdot \bd_T \, \dx 
		+
		\sum_E \int_E (\bp \cdot \bn_E) \, d_E \, \ds.
	\end{equation}
	In fact, $\RT{r+1}^0(\Omega)$ has the same dimension as $Y$ and, for any $\bp \in \RT{r+1}^0(\Omega)$, \eqref{eq:duality_product} clearly defines a linear functional on $Y$.
	Moreover, the mapping $\bp \mapsto \dual{\bp}{\cdot}$ is injective since $\dual{\bp}{\bd} = 0$ for all $\bd \in Y$ implies $\bp = \bnull$; see \eqref{eq:RTr+1_dofs}.
	With this representation of $Y^*$ available, we can evaluate $\Lambda^*: \RT{r+1}^0(\Omega) \to U$, where we identify $U$ with its dual space using the Riesz isomorphism induced by the $L^2(\Omega)$ inner product.
	Consequently, $\Lambda^*$ is defined by the condition $\dual{\bp}{\Lambda u} = (u, \Lambda^* \bp)_{L^2(\Omega)}$ for all $\bp \in \RT{r+1}^0(\Omega)$ and all $u \in \DG{r}(\Omega)$.
	The left hand side is
	\begin{multline}
		\label{eq:calculation_of_Lambdastar}
		\dual{\bp}{\Lambda u} 
		=
		\sum_T \int_T \bp \cdot \nabla u \, \dx 
		+
		\sum_E \int_E (\bp \cdot \bn_E) \, \jump{u} \, \ds
		\\
		=
		\sum_T - \int_T (\div \bp) \, u \, \dx 
		+
		\sum_T \int_{\partial T} (\bp \cdot \bn_T) \, u \, \ds
		\ifthenelse{\boolean{ispreprint}}{}{\\}
		+
		\sum_E \int_E (\bp \cdot \bn_E) \, \jump{u} \, \ds
		=
		- \int_\Omega (\div \bp) \, u \, \dx,
	\end{multline}
	hence $\Lambda^* = - \div$ holds.
	Here $\bn_T$ denotes the outward unit normal along the triangle boundary $\partial T$.

	The dual problem can be cast as
	\begin{equation}
		\label{eq:abstract_dual_problem}
		\text{Minimize} \quad F^*(-\Lambda^* \bp) + \beta \, G^*(\bp/\beta).
	\end{equation}
	It is well known that the convex conjugate of $F(u) = \frac{1}{2} \norm{u-f}_{L^2(\added{\Omega_0})}^2$ is $F^*(u) = \frac{1}{2} \norm{u+f}_{L^2(\added{\Omega_0})}^2 - \frac{1}{2} \norm{f}_{L^2(\added{\Omega_0})}^2$.
	It remains to evaluate 
	\begin{multline*}
		G^*(\bp)
		=
		\sup_{\bd \in Y} \dual{\bp}{\bd} - G(\bd)
		\\
		=
		\sup_{\bd \in Y}
		\sum_T \int_T \left[ \bp \cdot \bd_T - \II_T \big\{ \abs{\bd_T}_s \big\} \right] \, \dx 
		\ifthenelse{\boolean{ispreprint}}{}{\\}
		+
		\sum_E \int_E \left[ (\bp \cdot \bn_E) \, d_E - \II_E \big\{ \abs{d_E} \big\} \abs{\bn_E}_s \right] \, \ds.
	\end{multline*}
	Let us consider the contribution from $d_E = \alpha \, \LagrangeBasisEr{j}$ for some $\alpha \in \R$ on a single interior edge $E$, and $\bd \equiv 0$ otherwise.
	By \eqref{eq:RTr+1_dofs_E} and \eqref{eq:dual_representation_of_DTV_upper_bounds}, this contribution is $\alpha \, \RTDofsE{j}(\bp) - \abs{\alpha} \, \abs{\bn_E}_s \, c_{E,j}$, which is bounded above if and only if $\abs{\RTDofsE{j}(\bp)} \le \abs{\bn_E}_s \, c_{E,j}$.
	In this case, the maximum is zero.
	Similarly, it can be shown that the contribution from $\bd_T = \begin{psmallmatrix} \alpha_1 \\ \alpha_2 \end{psmallmatrix} \, \LagrangeBasisTrmone{i}$ remains bounded above if and only if $\abs{\RTDofsT{i}(\bp)}_{s^*} \le c_{T,i}$, in which case the maximum is zero as well.
	This shows that $G^* = I_{\bP}$ is the indicator function of the constraint set $\bP$ defined in \eqref{eq:constraint_set}, which concludes the proof.
	\qed
\end{proof}

Notice that the discrete dual problem \eqref{eq:dualDTV-L2} features the same, very simple set of constraints which already appeared in \eqref{eq:dual_representation_of_DTV}.
As is the case for \eqref{eq:dualTV-L2}, the solution of the discrete dual problem \eqref{eq:dualDTV-L2} is not necessarily unique.
However its divergence is unique \added{on $\Omega_0$} due to the strong convexity of the objective in terms of $\div \bp$. 

\added{Although not needed for \cref{alg:split_Bregman_s_arbitrary,alg:Chambolle-Pock_s_arbitrary}, we state the following relation between the primal and the dual solutions for completeness.}
\begin{lemma}[Recovery of the Primal Solution in \eqref{eq:DTV-L2}]
	\label{lemma:recovery_of_primal_solution_DTV-L2}
	Suppose that $\bp \in \RT{r+1}^0(\Omega)$ is a solution of \eqref{eq:dualDTV-L2}\added{ in case $\Omega_0 = \Omega$}.
	Then the unique solution of \eqref{eq:DTV-L2} is given by 
	\begin{equation}
		\label{eq:Recovery_u_from_p_DTVL2}
		u = \div \bp + f \in \DG{r}(\Omega).
	\end{equation}
\end{lemma}
\begin{proof}
	From \eqref{eq:abstract_dual_problem}, the pair of optimality conditions \added{to analyze} is 
	\begin{equation}
		\label{eq:abstract_optimality_conditions}
		-\Lambda^* \bp \in \partial F(u)
		\quad \text{and} \quad
		\added{\bp \in \partial (\beta \, G)(\Lambda u)},
	\end{equation}
	\added{see \cite[Ch.~III, Sect.~4]{EkelandTemam1999}.}
	\added{Here it suffices to consider the first condition, which by \cite[Prop.~I.5.1]{EkelandTemam1999} is}
	equivalent to $F(u) + F^*(-\Lambda^* \bp) - \scalarprod{u}{-\Lambda^* \bp}_{L^2(\Omega)} = 0$. 
	\added{This equality} can be rewritten as
	\begin{multline*}
		\ifthenelse{\boolean{ispreprint}}{\hfill}{} 
		\norm{u-f}^2_{L^2(\Omega)} 
		+ 
		\norm{\div \bp + f}^2_{L^2(\Omega)} 
		- 
		\norm{f}^2_{L^2(\Omega)} 
		\ifthenelse{\boolean{ispreprint}}{}{\\}
		- 2 \, \scalarprod{u}{\div \bp}_{L^2(\Omega)} 
		= 0.
		\ifthenelse{\boolean{ispreprint}}{\hfill}{}
	\end{multline*}
	Developing each summand in terms of the inner product $\scalarprod{\cdot}{\cdot}_{L^2(\Omega)}$ and rearranging appropriately, we obtain
	\begin{multline*}
		\ifthenelse{\boolean{ispreprint}}{\hfill}{}
		\scalarprod{u - f - \div \bp}{u}_{L^2(\Omega)} 
		+ 
		\scalarprod{- u + f + \div \bp}{f}_{L^2(\Omega)} 
		\ifthenelse{\boolean{ispreprint}}{}{\\}
		+ 
		\scalarprod{\div \bp + f - u}{\div \bp}_{L^2(\Omega)} = 0,
		\ifthenelse{\boolean{ispreprint}}{\hfill}{}
	\end{multline*}
	which amounts to $\norm{u-f-\div \bp}^2_{L^2(\Omega)} = 0$, and \eqref{eq:Recovery_u_from_p_DTVL2} is proved.
	\qed
\end{proof}

\begin{remark}
	\label{remark:recovery_of_primal_solution_DTV-L2}
	\added{%
		In case $\Omega_0 \subsetneq \Omega$, the solution of the primal problem will not be unique in general.
		An inspection of the proof of \cref{lemma:recovery_of_primal_solution_DTV-L2} shows that in this case, one can derive the relation
	}
	\begin{equation*}
		\norm{u-f-\div \bp}^2_{L^2(\Omega_0)} 
		= 
		2 \int_{\Omega \setminus \Omega_0} u \, \div \bp \, \dx
		.
	\end{equation*}
\end{remark}

\subsection{The TV-$L^1$ Problem}
\label{subsec:Discrete_dual_problem_TV-L1}

The continuous (pre-)dual problem associated with 
\begin{equation*}
	\tag{TV-L1}
	\text{Minimize} \quad \norm{u-f}_{L^1(\added{\Omega_0})} + \beta \, \abs{u}_{TV(\Omega)}
\end{equation*}
can be shown along the lines of \cite[Thm.~2.2]{HintermuellerKunisch2004:2} to be
\begin{equation}
	\label{eq:dualTV-L1}
	\tag{TV-L1-D}
	\begin{aligned}
		\text{Minimize} \quad & \int_{\Omega\added{_0}} (\div \bp) \, f \, \dx 
		\\
		\text{s.t.} \quad & \abs{\div \bp} \le \added{\chi_{\Omega_0}} \text{ and } \abs{\bp}_{s^*} \le \beta
	\end{aligned}
\end{equation}
with $\bp \in \bH_0(\div;\Omega)$\added{, where $\chi_{\Omega_0}$ is the characteristic function of $\Omega_0$}.

The definition of an appropriate discrete counterpart of \eqref{eq:TV-L1} deserves some attention.
Simply replacing $\abs{u}_{TV(\Omega)}$ by $\abs{u}_{DTV(\Omega)}$ would yield a discrete dual problem with an infinite number of pointwise constraints $\abs{\div \bp} \le \added{\chi_{\Omega_0}}$ as in \eqref{eq:dualTV-L1}, which would render the problem intractable.
We therefore advocate to consider
\begin{equation}
	\label{eq:DTV-L1}
	\tag{DTV-L1}
	\text{Minimize} \quad \sum_{\added{T \subset \Omega_0}} \int_T \JJ_T \big\{ \abs{u-f} \big\} \, \dx + \beta \, \abs{u}_{DTV(\Omega)}
\end{equation}
as an appropriate discrete version of \eqref{eq:TV-L1} with $u \in \DG{r}(\Omega)$.
Here $\JJ_T$ denotes the interpolation operator into $\PP_r(T)$, i.e.,
\begin{equation*}
	\JJ_T \big\{ \abs{u-f} \big\}
	=
	\sum_k \abs{u-f}(\LagrangeNodesTr{k}) \, \LagrangeBasisTr{k}
	.
\end{equation*}
This choice of applying an interpolatory quadrature formula to the data fidelity (loss) term as well is a decisive advantage, yielding a favorable dual problem.

\begin{theorem}[Discrete dual problem for \eqref{eq:DTV-L1}]
	\label{theorem:discrete_dual_problem_TV-L1}
	Let $0 \le r \le \added{3}$.
	Then the dual problem of \eqref{eq:DTV-L1} is
	\begin{equation}
		\label{eq:dualDTV-L1}
		\tag{DTV-L1-D}
		\begin{aligned}
			\text{Minimize} \quad & \int_{\Omega\added{_0}} (\div \bp) \, f \, \dx 
			\\
			\text{s.t.\ } \quad & \vabs{\int_T (\div \bp) \, \LagrangeBasisTr{k} \, \dx} \le C_{T,k} \added{\text{ when } T \subset \Omega_0} 
			\\
			\text{and} \quad & \vabs{\int_T (\div \bp) \, \LagrangeBasisTr{k} \, \dx} \added{{}= 0} \added{\text{ when } T \subset \Omega \setminus \Omega_0} 
			\\
			\text{and} \quad & \bp \in \beta \bP.
		\end{aligned}
	\end{equation}
\end{theorem}
\begin{proof}
	We proceed similarly as in the proof of \cref{theorem:discrete_dual_problem_TV-L2}.
	The functions $G$, $G^*$ and $\Lambda$ remain unchanged, and we replace $F$ by 
	\begin{multline}
		\label{eq:definition_of_F_for_DTV-L1}
		\ifthenelse{\boolean{ispreprint}}{\hfill}{}
		F(u) 
		= 
		\sum_{\added{T \subset \Omega_0}} \int_T \JJ_T \big\{ \abs{u-f} \big\} \, \dx
		\ifthenelse{\boolean{ispreprint}}{}{\\}
		=
		\sum_{\added{T \subset \Omega_0},k} \abs{u-f}(\LagrangeNodesTr{k}) \, C_{T,k},
		\ifthenelse{\boolean{ispreprint}}{\hfill}{}
	\end{multline}
	where $C_{T,k} \coloneqq \int_T \LagrangeBasisTr{k} \, \dx$ is non-negative due to \cref{lemma:basis_functions_positive_integrals}.
	We identify again $U = \DG{r}(\Omega)$ with its dual but this time not via the regular $L^2(\Omega)$ inner product but via its lumped approximation, i.e.,
	\begin{equation}
		\label{eq:scalar_product_DGr_lumped}
		\scalarprod{u}{v}_{\text{lumped}}
		\coloneqq
		\sum_{T,k} u(\LagrangeNodesTr{k}) \, v(\LagrangeNodesTr{k}) \, C_{T,k}
	\end{equation}
	for $u, v \in \DG{r}(\Omega)$.
	Notice that this choice first of all affects the representation of $\Lambda^*:\RT{r+1}^0(\Omega) \to U$.
	Indeed, using \eqref{eq:calculation_of_Lambdastar} it follows that $v = \Lambda^* \bp$ is now defined by
	\begin{equation}
		\label{eq:calculation_of_Lambdastar_lumped_1}
		\scalarprod{u}{v}_{\text{lumped}}
		=
		- \int_\Omega (\div \bp) \, u \, \dx
		\quad \text{for all $u \in \DG{r}(T)$}
		.
	\end{equation}
	For the particular choice $u = \LagrangeBasisTr{k}$, this yields
	\begin{equation}
		\label{eq:calculation_of_Lambdastar_lumped_2}
		v(\LagrangeNodesTr{k}) 
		=
		(\Lambda^* \bp)(\LagrangeNodesTr{k})
		=
		- \frac{1}{C_{T,k}} \int_T (\div \bp) \, \LagrangeBasisTr{k} \, \dx
	\end{equation}
	when $C_{T,k} > 0$.
	As a side remark, we mention that \eqref{eq:calculation_of_Lambdastar_lumped_2} means that $\Lambda^* \bp$ is given locally by Carstensen's quasi-interpolant of $-\div \bp$ into $\PP_r(T)$; see \cite{Carstensen1999:2}.
	When $C_{T,k} = 0$, then \eqref{eq:calculation_of_Lambdastar_lumped_1} can only be satisfied when $$\int_T (\div \bp) \, \LagrangeBasisTr{k} \, \dx = 0$$ holds, in which case $v(\LagrangeNodesTr{k})$ is arbitrary. 

	Next, since $F$ from \eqref{eq:definition_of_F_for_DTV-L1} is a weighted $\ell_1$-norm, its convex conjugate can be easily seen to be
	\begin{equation*}
		F^*(u)
		=
		\sum_{\added{T \subset \Omega_0},k} u(\LagrangeNodesTr{k}) \, f(\LagrangeNodesTr{k}) \, C_{T,k} 
	\end{equation*}
	if $\abs{u(\LagrangeNodesTr{k})} \le \added{\chi_{\Omega_0}(\LagrangeNodesTr{k})}$ for all triangles $T$ and $k$ s.t.\ $C_{T,k} > 0$; and $F^*(u) = \infty$ otherwise.
	Consequently, by \eqref{eq:calculation_of_Lambdastar_lumped_2}, 
	\begin{multline*}
		\ifthenelse{\boolean{ispreprint}}{\hfill}{}
		F^*(-\Lambda^* \bp)
		=
		\sum_{\added{T \subset \Omega_0},k} \int_T (\div \bp) \, \LagrangeBasisTr{k} \, \dx \, f(\LagrangeNodesTr{k}) 
		\ifthenelse{\boolean{ispreprint}}{}{\\}
		=
		\sum_{\added{T \subset \Omega_0}} \int_T (\div \bp) \, f \, \dx
		=
		\int_{\added{\Omega_0}} (\div \bp) \, f \, \dx
		\ifthenelse{\boolean{ispreprint}}{\hfill}{}
	\end{multline*}
	holds when $\vabs{\int_T (\div \bp) \, \LagrangeBasisTr{k} \, \dx} \le C_{T,k}\added{\, \chi_{\Omega_0}(\LagrangeNodesTr{k})}$ is satisfied, and $F^*(-\Lambda^* \bp) = \infty$ otherwise.
	Plugging this into \eqref{eq:abstract_dual_problem} concludes the proof.
	\added{Notice that in case $T \subset \Omega \setminus \Omega_0$, the constraints $\int_T (\div \bp) \, \LagrangeBasisTr{k} \, \dx = 0$ for all $k$ imply that $\div \bp \equiv 0$ on $T$ since $\div \bp \in \PP_r(T)$; see \eqref{eq:RTr+1}.}
	\qed
\end{proof}

\begin{remark}[Discrete dual problem \eqref{eq:dualDTV-L1}] 
	\label{remark:discrete_dual_problem_DTV-L1}
	\begin{enumerate}[label=$(\alph*)$]
		\item 
			The replacement of $\norm{\,\cdot\,}_{L^1(\Omega)}$ in the objective as well as of the $L^2(\Omega)$ inner product in $U$ by lumped versions obtained by interpolatory quadrature has been successful in other contexts before; see for instance \cite{CasasHerzogWachsmuth2011:1}.
			Here, it is essential in converting the otherwise infinitely many \emph{pointwise} constraints $\abs{\div \bp} \le \added{\chi_{\Omega_0}}$ into just finitely many constraints on $\div \bp$.

		\item 
			Notice that when $s^* \in \{1, \infty\}$ holds, then the dual \eqref{eq:dualDTV-L1} is a linear program.

		\item 
			One may ask what would have happened if we had applied the same quadrature formula to the $L^2(\Omega)$ inner product already in \eqref{eq:DTV-L2}.
			It can be seen by straightforward calculations that the objective in \eqref{eq:dualDTV-L2} would have been replaced by 
			\begin{equation*}
				\frac{1}{2} \sum_{\added{T \subset \Omega_0},k} \left( \frac{1}{C_{T,k}} \int_T (\div \bp) \, \LagrangeBasisTr{k} \, \dx + f(\LagrangeNodesTr{k}) \right)^2 C_{T,k}
			\end{equation*}
			with summands involving $C_{T,k} = 0$ omitted.
			There is, however, no structural advantage compared to \eqref{eq:dualDTV-L2}.
	\end{enumerate}
\end{remark}

\section{Algorithms for \eqref{eq:DTV-L2}}
\label{sec:Algorithms}

Our goal in this section is to show that \ifthenelse{\boolean{ispreprint}}{a variety of}{\added{two}} standard algorithms developed for images on Cartesian grids, with finite difference approximations of gradient and divergence operations, are implementable with the same efficiency in our framework of higher-order finite elements on triangular meshes.
We focus in this section on \eqref{eq:DTV-L2} and come back to \eqref{eq:DTV-L1} in \cref{sec:DTV-L1}.
\ifthenelse{\boolean{ispreprint}}%
{Specifically, we consider in the following the split Bregman iteration \cite{GoldsteinOsher2009}, the primal-dual method of \cite{ChambollePock2011}, Chambolle's projection method \cite{Chambolle2004}, and a primal-dual active set method similar to \cite{HintermuellerKunisch2004:2}.}%
{Specifically, we consider in the following the split Bregman iteration \cite{GoldsteinOsher2009} \added{and} the primal-dual method of \added{Chambolle and Pock} \cite{ChambollePock2011}.
		\added{We refer the reader to the extended preprint \cite{HerrmannHerzogSchmidtVidalNunezWachsmuth2018:1_preprint} for a additional discussion of Chambolle's projection method \cite{Chambolle2004} and a primal-dual active set method similar to \cite{HintermuellerKunisch2004:2}.}}
	Since \added{these} algorithms are well known, we only focus on the main steps in each case.
Let us recall that we are seeking a solution $u \in \DG{r}$.
For simplicity, we exclude the case $r = 3$ in this section, i.e., we restrict the discussion to the polynomial degrees $r \in \{0,1,2,4\}$ so that all weights $c_{T,i}$ and $c_{E,j}$ are strictly positive.
The case $r = 3$ can be included provided that zero weights are properly treated and we come back to this in \cref{subsec:Missing_polynomial_degrees}.

\subsection{Split Bregman Method}
\label{subsec:split_Bregman}

The split Bregman method (also known as alternating direction method of multipliers (ADMM)) considers the primal problem \eqref{eq:DTV-L2}. 
It introduces an additional variable $\bd$ so that \eqref{eq:DTV-L2} becomes
\ifthenelse{\boolean{ispreprint}}
	{
		\begin{equation}
			\label{eq:DTV-L2_with_splitting}
			\begin{aligned}
				\text{Minimize} \quad & \frac{1}{2} \norm{u-f}_{L^2(\added{\Omega_0})}^2 
				+
				\beta \sum_{T,i} c_{T,i} \, \bigabs{\bd_{T,i}}_s 
				+ 
				\beta \sum_{E,j} \abs{\bn_E}_s \, c_{E,j} \, \abs{d_{E,j}}
				\\
				\text{s.t.} \quad & \bd = \Lambda u,
				\quad (u,\bd) \in \DG{r}(\Omega) \times Y
			\end{aligned}
		\end{equation}
	}
	{
		\begin{multline}
			\label{eq:DTV-L2_with_splitting}
			\text{Minimize} \quad \frac{1}{2} \norm{u-f}_{L^2(\added{\Omega_0})}^2 
			+
			\beta \sum_{T,i} c_{T,i} \, \bigabs{\bd_{T,i}}_s 
			\\
			+ 
			\beta \sum_{E,j} \abs{\bn_E}_s \, c_{E,j} \, \abs{d_{E,j}}
			\quad \text{s.t.} \quad \bd = \Lambda u
		\end{multline}
	}
and enforces the constraint $\bd = \Lambda u = \nabla u$ by an augmented Lagrangian approach.
As detailed in \eqref{eq:definition_of_Lambda}, $\bd$ has contributions $\restr{\nabla u}{T}$ per triangle, as well as contributions $\jump{u}_E$ per interior edge.
We can thus express $\bd$ through its coefficients $\{\bd_{T,i}\}$ and $\{d_{E,j}\}$ w.r.t.\ the standard Lagrangian bases of $\PP_{r-1}(T)^2$ and $\PP_r(E)$,
\begin{equation}
	\label{eq:expansion_of_d}
	\bd
	=
	\sum_i \bd_{T,i} \, \LagrangeBasisTrmone{i} 
	+
	\sum_j d_{E,j} \, \LagrangeBasisEr{j}.
\end{equation}
Using \eqref{eq:interpolation_terms_in_DTV} and \eqref{eq:dual_representation_of_DTV_upper_bounds}, we rewrite the discrete total variation \eqref{eq:discrete_TV_for_DG} in terms of $\bd$ and adjoin the constraint $\bd = \nabla u$ by way of an augmented Lagrangian functional,
	\begin{multline}
		\label{eq:split_Bregman_ALM_functional_with_inner_product_in_Y}
		\ifthenelse{\boolean{ispreprint}}{\hfill}{}
			\frac{1}{2} \norm{u-f}_{L^2(\added{\Omega_0})}^2 
		+
		\beta \sum_{T,i} c_{T,i} \, \bigabs{\bd_{T,i}}_s 
		\ifthenelse{\boolean{ispreprint}}{}{\\}
		+
		\beta \sum_{E,j} \abs{\bn_E}_s \, c_{E,j} \, \abs{d_{E,j}}
		+ 
		\frac{\lambda}{2} \norm{\bd - \Lambda u - \bb}_Y^2.
		\ifthenelse{\boolean{ispreprint}}{\hfill}{}
	\end{multline}
Here $\bb$ is an estimate of the Lagrange multiplier associated with the constraint $\bd = \nabla u \in Y$, and $\bb$ is naturally discretized in the same way as $\bd$.

\begin{remark}[Inner product on $Y$] \hfill \\
	\label{remark:inner_product_on_Y}
	So far we have not endowed the space $$Y = \prod_T \PP_{r-1}(T)^2 \times \prod_E \PP_r(E)$$ with an inner product.
	Since elements of $Y$ represent (mea\deleted{u}sure-valued) gradients of $\DG{r}(\Omega)$ functions, the natural choice would be to endow $Y$ with a total variation norm of vector measures, which would amount to 
	\begin{equation*}
		\sum_T \int_T \abs{\bd_T}_s \, \dx 
		+
		\sum_E \abs{\bn_E}_s \int_E \abs{d_E} \, \ds
	\end{equation*}
	for $\bd \in Y$.
	Clearly, this $L^1$-type norm is not induced by an inner product.
	Therefore we are using\added{ the $L^2$ inner product instead.}
	For computational efficiency, it is crucial to consider its lumped version, which amounts to
	\begin{equation}
		\label{eq:inner_product_in_Y}
		\scalarprod{\bd}{\be}_Y
		\coloneqq
		\added{\scaling}\sum_{T,i} c_{T,i} \, \bd_{T,i} \, \be_{T,i}
		+
		\sum_{E,j} c_{E,j} \, d_{E,j} \, e_{E,j} 
	\end{equation}
	for $\bd, \be \in Y$.
	The associated norm is denoted as $\norm{\bd}_Y^2 = \scalarprod{\bd}{\bd}_Y$.
	\added{Notice that $\scaling > 0$ is a scaling parameter which can be used to improve the convergence of the split Bregman and other iterative methods.}
\end{remark}

The efficiency of the split Bregman iteration depends on the ability to efficiently minimize \eqref{eq:split_Bregman_ALM_functional_with_inner_product_in_Y} independently for $u$, $\bd$ and $\bb$, respectively.
Let us show that this is the case.

\subsubsection*{The Gradient Operator $\Lambda$}

The gradient operator $\Lambda$ evaluates the cell-wise gradient of $u \in \DG{r}(\Omega)$ as well as the edge jump contributions, see \eqref{eq:definition_of_Lambda}.
These are standard operations in any finite element toolbox.
For computational efficiency, the matrix realizing $u(\LagrangeNodesTrmone{i})$ and $u(\LagrangeNodesEr{j})$ in terms of the coefficients of $u$ can be stored once and for all.

\subsubsection*{Solving the $u$-problem}

We consider the minimization of \eqref{eq:split_Bregman_ALM_functional_with_inner_product_in_Y}, or equivalently, of 
\begin{multline}
	\label{eq:split_Bregman_ALM_functional_u}
	\frac{1}{2} \norm{\alert{u}-f}_{L^2(\added{\Omega}\added{_0})}^2 
	+ 
	\frac{\lambda\added{\scaling}}{2} \sum_{T,i} c_{T,i} \, \bigabs{\bd_{T,i} - \nabla \alert{u}(\LagrangeNodesTrmone{i}) - \bb_{T,i}}_2^2 
	\\
	+
	\frac{\lambda}{2} \sum_{E,j} c_{E,j} \, \bigabs{d_{E,j} - \jump{\alert{u}}(\LagrangeNodesEr{j}) - b_{E,j}}^2
\end{multline}
w.r.t.\ $u \in \DG{r}(\Omega)$.
This problem can be interpreted as a DG finite element formulation of the elliptic partial differential equation $- \lambda \, \laplace u + \added{\chi_{\Omega_0}} u = \added{\chi_{\Omega_0}} f + \lambda \div(\bb - \bd)$ in $\Omega$.
More precisely, it constitutes a nonsymmetric interior penalty Galerkin (NIPG) method; compare for instance \cite{RiviereWheelerGirault1999} or \cite[Ch.~2.4, 2.6]{Riviere2008}.
Specialized preconditioned solvers for such systems are available, see for instance \cite{AntoniettiAyuso2007}.
However, as proposed in \cite{GoldsteinOsher2009}, a (block) Gauss--Seidel method may be sufficient.
It is convenient to group the unknowns of the same triangle together, which leads to local systems of size $(r+1)(r+2)/2$.

\subsubsection*{Solving the $\bd$-problem}

The minimization of \eqref{eq:split_Bregman_ALM_functional_with_inner_product_in_Y}, or equivalently, of 
\begin{multline}
	\label{eq:split_Bregman_ALM_functional_d}
	\beta \sum_{T,i} c_{T,i} \, \bigabs{\alert{\bd_{T,i}}}_s 
	+
	\beta \sum_{E,j} \abs{\bn_E}_s \, c_{E,j} \, \abs{\alert{d_{E,j}}}
	\\
	+ 
	\frac{\lambda\added{\scaling}}{2} \sum_{T,i} c_{T,i} \, \bigabs{\alert{\bd_{T,i}} - \nabla u(\LagrangeNodesTrmone{i}) - \bb_{T,i}}_2^2 
	\ifthenelse{\boolean{ispreprint}}{}{\\}
	+
	\frac{\lambda}{2} \sum_{E,j} c_{E,j} \, \bigabs{\alert{d_{E,j}} - \jump{u}(\LagrangeNodesEr{j}) - b_{E,j}}^2
\end{multline}
decouples into the minimization of
\begin{subequations}
	\label{eq:split_Bregman_ALM_functional_dTE}
	\begin{align}
		\label{eq:split_Bregman_ALM_functional_dT}
		\beta \, \bigabs{\alert{\bd_{T,i}}}_s 
		+
		\frac{\lambda\added{\scaling}}{2} \bigabs{\alert{\bd_{T,i}} - \nabla u(\LagrangeNodesTrmone{i}) - \bb_{T,i}}_2^2
		\\
		\label{eq:split_Bregman_ALM_functional_dE}
		\text{and }
		\beta \, \abs{\bn_E}_s \, \abs{\alert{d_{E,j}}}
		+
		\frac{\lambda}{2} \bigabs{\alert{d_{E,j}} - \jump{u}(\LagrangeNodesEr{j}) - b_{E,j}}^2
	\end{align}
\end{subequations}
w.r.t.\  $\alert{\bd_{T,i}} \in \R^2$ and $\alert{d_{E,j}} \in \R$, respectively.

It is well known that the scalar problem \eqref{eq:split_Bregman_ALM_functional_dE} is solved via 
\begin{equation*}
	\alert{d_{E,j}} 
	= 
	\shrink \left( \jump{u}(\LagrangeNodesEr{j}) + b_{E,j}, \; \frac{\beta \, \abs{\bn_E}_s }{\lambda} \right),
\end{equation*}
where $\shrink(\xi,\gamma) \coloneqq \max \left\{ \abs{\xi} - \gamma, \; 0 \right\} \sgn{\xi}$,
while the minimization of \eqref{eq:split_Bregman_ALM_functional_dT} defines the (Euclidean) $\prox$ mapping of $\abs{\,\cdot\,}_s$ and thus we have
\begin{equation*}
	\alert{\bd_{T,i}} 
	= 
	\prox_{\beta/\added{(}\lambda\added{\scaling}\added{)} \abs{\,\cdot\,}_s} \big( \nabla u(\LagrangeNodesTrmone{i}) + \bb_{T,i} \big),
\end{equation*}
where
\begin{equation*}
	\prox_{\beta/\added{(}\lambda\added{\scaling}\added{)} \abs{\,\cdot\,}_s}(\bxi) 
	= 
	\bxi - \frac{\beta}{\lambda\added{\scaling}} \proj_{B_{\abs{\,\cdot\,}_{s^*}}} \left( \frac{\lambda\added{\scaling}}{\beta} \bxi \right).
\end{equation*}
Here $\proj_{B_{\abs{\,\cdot\,}_{s^*}}}$ is the Euclidean orthogonal projection onto the closed $\abs{\,\cdot\,}_{s^*}$-norm unit ball; see for instance \cite[Ex.~6.47]{Beck2017}.
When $s \in \{1, 2\}$, then we have closed-form solutions of \eqref{eq:split_Bregman_ALM_functional_dT}:
\begin{equation*}
	[\alert{\bd_{T,i}}]_\ell
	= 
	\shrink \left( \big[ \nabla u(\LagrangeNodesTrmone{i}) + \bb_{T,i} \big]_\ell, \; \frac{\beta}{\lambda\added{\scaling}} \right)
	\text{ for } \ell = 1,2
\end{equation*}
when $s = 1$
and
\begin{multline*}
	\ifthenelse{\boolean{ispreprint}}{\hfill}{}
	\alert{\bd_{T,i}} 
	= 
	\max \left\{ \bigabs{\nabla u(\LagrangeNodesTrmone{i}) + \bb_{T,i}}_2 - \frac{\beta}{\lambda\added{\scaling}}, \; 0 \right\}
	\ifthenelse{\boolean{ispreprint}}{}{\\}
	\cdot
	\frac{\nabla u(\LagrangeNodesTrmone{i}) + \bb_{T,i}}{\bigabs{\nabla u(\LagrangeNodesTrmone{i}) + \bb_{T,i}}_2}
	\ifthenelse{\boolean{ispreprint}}{\hfill}{}
\end{multline*}
when $s = 2$.
When $\nabla u(\LagrangeNodesTrmone{i}) + \bb_{T,i} = 0$, the second formula is understood as $\alert{\bd_{T,i}} = 0$. 
Efficient approaches for $s = \infty$ are also available; see \cite{DuchiShalev-ShwartzSingerChandra2008}.

\subsubsection*{Updating $\bb$}

This is simply achieved by replacing the current values for $\alert{\bb_{T,i}}$ and $\alert{b_{E,j}}$ by $\bb_{T,i} + \nabla u(\LagrangeNodesTrmone{i}) - \bd_{T,i}$ and $b_{E,j} + \jump{u}(\LagrangeNodesEr{j}) - d_{E,j}$, respectively.

The quantities $\bb_{T,i}$ and $b_{E,j}$ represent discrete multipliers associated with the components of the constraint $\bd = \Lambda u$.
Here we clarify how these multipliers relate to the dual variable $\bp \in \RT{r+1}^0(\Omega)$ in \eqref{eq:dualDTV-L2}.
In fact, let us interpret $\bb_{T,i}$ as the coefficients of a function $\bb_T \in \PP_{r-1}(T)$ and $b_{E,j}$ as the coefficients of a function $b_E \in \PP_r(E)$ w.r.t.\ the standard nodal bases, just as in \eqref{eq:expansion_of_d}.
Moreover, let us define a function $\bar \bp \in \RT{r+1}^0(\Omega)$ by specifying its coefficients as follows,
\begin{equation}
	\label{eq:recover_p_from_discrete_multipliers}
	\RTDofsT{i}(\bar \bp) \coloneqq \lambda\added{\scaling} \, \bb_{T,i} \, c_{T,i}
	\quad \text{and} \quad
	\RTDofsE{j}(\bar \bp) \coloneqq \lambda \, b_{E,j} \, c_{E,j}.
\end{equation}
Then 
\begin{align*}
	\MoveEqLeft
	\int_T \! \bar \bp \cdot (\nabla u - \bd_T) \, \dx
	\ifthenelse{\boolean{ispreprint}}{}{\\ &}
	=
	\added{\sum_i \int_T \! \bar \bp \, \LagrangeBasisTrmone{i} \cdot \big( \nabla u(\LagrangeNodesTrmone{i}) - \bd_{T,i} \big) \, \dx}
	\\
	&
	=
	\ifthenelse{\boolean{ispreprint}}{\sum_i \RTDofsT{i}(\bar \bp) \cdot \big( \nabla u(\LagrangeNodesTrmone{i}) - \bd_{T,i} \big) = }{}
	\lambda\added{\scaling} \sum_i c_{T,i} \, \bb_{T,i} \cdot \big( \nabla u(\LagrangeNodesTrmone{i}) - \bd_{T,i} \big) 
\end{align*}
and
\begin{align*}
	\MoveEqLeft
	\int_E \! \bar \bp \, (\jump{u} - d_E) \, \bn_E \, \ds
	\ifthenelse{\boolean{ispreprint}}{}{\\ &}
	=
	\added{\sum_j \int_E \! \bar \bp \, \LagrangeBasisEr{j} \big( \jump{u}(\LagrangeNodesEr{j}) - d_{E,j} \big) \, \ds }
	\\
	&
	=
	\ifthenelse{\boolean{ispreprint}}{\sum_j \RTDofsE{j}(\bar \bp) \, (\jump{u}(\LagrangeNodesEr{j}) - d_{E,j}) = }{}
	\lambda \sum_j c_{E,j} \, b_{E,j} \, (\jump{u}(\LagrangeNodesEr{j}) - d_{E,j})
	,
\end{align*}
and these are precisely the terms appearing in the discrete augmented Lagrangian functional \eqref{eq:split_Bregman_ALM_functional_with_inner_product_in_Y}.
Con\-se\-quent\-ly, $\bar \bp$ can be interpreted as the Lagrange multiplier associated with the components of the constraint $\bd = \Lambda u$, when the latter are adjoined using the lumped $\bL^2(T)$ and $L^2(E)$ inner products.
It can be shown using 
\added{the KKT conditions for \eqref{eq:DTV-L2_with_splitting} and the optimality conditions \eqref{eq:abstract_optimality_conditions}} that $\bar \bp$ defined by \eqref{eq:recover_p_from_discrete_multipliers} solves the dual problem \eqref{eq:dualDTV-L2}.
\added{%
	To prove this assertion, suppose that $(u,\bd)$ is optimal for \eqref{eq:DTV-L2_with_splitting}.
	We will show that $(u,\bar \bp)$ satisfy the necessary and sufficient optimality conditions \eqref{eq:abstract_optimality_conditions}.
	The Lagrangian for \eqref{eq:DTV-L2_with_splitting} can be written as $F(u) + \beta \, G(\bd) + \dual{\bar \bp}{\Lambda u - \bd}$ and the optimality of $(u,\bd)$ implies $\bar \bp \in \partial (\beta \, G)(\bd) = \partial (\beta \, G)(\Lambda u)$.
	On the other hand, $u$ is optimal for \eqref{eq:DTV-L2}, which implies $0 \in \partial F(u) + \Lambda^* \partial (\beta \, G)(\Lambda u)$ and thus $- \Lambda^* \bar \bp \in \partial F(u)$.
	Altogether, we have verified \eqref{eq:abstract_optimality_conditions}, which is necessary and sufficient for $\bar \bp$ to be optimal for \eqref{eq:dualDTV-L2}.
}

For convenience, we specify the split Bregman iteration in \cref{alg:split_Bregman_s_arbitrary}.
\begin{algorithm}
	\caption{Split Bregman algorithm for \eqref{eq:DTV-L2} with $s \in [1,\infty]$}
	\label{alg:split_Bregman_s_arbitrary}
	\begin{algorithmic}[1]
		\STATE Set $u^{(0)} \coloneqq f \in \DG{r}(\Omega)$, $\bb^{(0)} \coloneqq \bnull \in Y$ and $\bd^{(0)} \coloneqq \bnull \in Y$
		\STATE Set $n \coloneqq 0$
		\WHILE{not converged}
		\STATE Minimize \eqref{eq:split_Bregman_ALM_functional_u} for $u^{(n+1)}$ with data $\bb^{(n)}$ and $\bd^{(n)}$
		\STATE Minimize \eqref{eq:split_Bregman_ALM_functional_dTE} for $\bd^{(n+1)}$ with data $u^{(n+1)}$ and $\bb^{(n)}$
		\STATE Set $\bb_{T,i}^{(n+1)} \coloneqq \bb_{T,i}^{(n)} + \nabla u^{(n+1)}(\LagrangeNodesTrmone{i}) - \bd_{T,i}^{(n+1)}$ 
		\STATE Set $b_{E,j}^{(n+1)} \coloneqq b_{E,j}^{(n)} + \jump{u^{(n+1)}}(\LagrangeNodesEr{j}) - d_{E,j}^{(n+1)}$
		\STATE Set $n \coloneqq n+1$
		\ENDWHILE
		\STATE Set $\bp^{(n)}$ by \eqref{eq:recover_p_from_discrete_multipliers} with data $\bb^{(n)}$ 
	\end{algorithmic}
\end{algorithm}

\subsection{Chambolle--Pock Method}
\label{subsec:Chambolle-Pock}

The method by \cite{ChambollePock2011}, also known as primal-dual extragradient method, see \cite{HeYuan2012}, is based on a reformulation of the optimality conditions in terms of the $\prox$ operators pertaining to $F$ and $G^*$.
We recall that $F$ is defined by $F(u) = \frac{1}{2} \norm{u-f}_{L^2(\added{\Omega_0})}^2$ on $U = \DG{r}(\Omega)$.
Moreover, $G^*$ is defined on $Y^* \cong \RT{r+1}^0(\Omega)$ by $G^* = I_{\bP}$, the indicator function of $\bP$, see \eqref{eq:constraint_set}.

Notice that $\prox$ operators depend on the inner product in the respective space.
We recall that $U$ has been endowed with the (regular, non-lumped) $L^2(\Omega)$ inner product, see the proof of \cref{theorem:discrete_dual_problem_TV-L2}.
For the space $Y$ we are using again the inner product defined in \eqref{eq:inner_product_in_Y}.
Exploiting the duality product \eqref{eq:duality_product} between $Y$ and $Y^* \cong \RT{r+1}^0(\Omega)$ it is then straightforward to derive the Riesz map $R: Y \ni \bd \mapsto \bp \in Y^*$.
In terms of the coefficients of $\bp$, we have
\begin{equation}
	\label{eq:Riesz_map_Y_to_Ystar}
	\RTDofsT{i}(\bp) = c_{T,i}\added{\scaling} \, \bd_{T,i}
	\quad \text{and} \quad
	\RTDofsE{j}(\bp) = c_{E,j} \, d_{E,j}.
\end{equation}
Consequently, the induced inner product in $\RT{r+1}^0(\Omega)$ becomes
\begin{multline}
	\label{eq:inner_product_in_Ystar}
	\ifthenelse{\boolean{ispreprint}}{\hfill}{}
	\scalarprod{\bp}{\bq}_{Y^*}
	\coloneqq
	\sum_{T,i} \frac{1}{c_{T,i}\added{\scaling}} \, \RTDofsT{i}(\bp) \cdot \RTDofsT{i}(\bq) 
	\ifthenelse{\boolean{ispreprint}}{}{\\}
	+
	\sum_{E,j} \frac{1}{c_{E,j}} \, \RTDofsE{j}(\bp) \, \RTDofsE{j}(\bq).
	\ifthenelse{\boolean{ispreprint}}{\hfill}{}
\end{multline}
To summarize, the inner products in $Y$, $Y^*$ as well as the Riesz map are realized efficiently by simple, diagonal operations on the coefficients.

\subsubsection*{Solving the $F$-prox}

Let $\sigma > 0$.
The $\prox$-operator of $\sigma F$, denoted by $$\prox_{\sigma F}(\bar u): U \to U,$$
is defined as $\alert{u} = \prox_{\sigma F}(\bar u)$ if and only if
\begin{equation*}
	\alert{u} = \argmin_{v \in \DG{r}(\Omega)} \frac{1}{2} \norm{v - \bar u}_{L^2(\Omega)}^2 + \frac{\sigma}{2} \norm{v-f}_{L^2(\added{\Omega_0})}^2.
\end{equation*}
For given data $\bar u \added{{}\in \DG{r}(\Omega)}$ and $f \in \added{\DG{r}(\Omega_0)}$, it is easy to see that a necessary and sufficient condition is $\alert{u} - \bar u + \sigma \, (\alert{u} - f) = 0$, which amounts to the coefficient-wise formula
\begin{equation}
	\label{eq:definition_of_prox_sigmaF}
	\alert{u_{T,k}} 
	=
	\frac{1}{1 + \sigma_{\added{T,k}}} \big( \bar u_{T,k} + \sigma_{\added{T,k}} f_{T,k} \big),
\end{equation}
\added{where $\sigma_{T,k} = \sigma$ if $T \subset \Omega_0$ and $\sigma_{T,k} = 0$ otherwise.}

\subsubsection*{Solving the $G^*$-prox}

Let $\tau > 0$.
The $\prox$-operator $$\prox_{\tau G^*}: Y^* \cong \RT{r+1}^0(\Omega) \to Y^*$$ is defined as $\alert{\bp} = \prox_{\tau G^*}(\bar \bp)$ if and only if
\begin{equation}
	\label{eq:prox_tauGstar}
	\alert{\bp} = \argmin_{\bq \in \RT{r+1}^0(\Omega)} \frac{1}{2} \norm{\bq - \bar \bp}_{Y^*}^2 
	\text{ s.t.\ } \bq \in \bP.
\end{equation}
Similarly, the $\prox$ operator for $(\beta \, G)^*$ is obtained by replacing $\bP$ by $\beta \bP$, for any $\tau > 0$.
Due to the diagonal structure of the inner product in $Y^*$, this is efficiently implementable.
When $\bar \bp \in \RT{r+1}^0(\Omega)$, then we obtain the solution in terms of the coefficients, similar to \eqref{eq:split_Bregman_ALM_functional_dTE}, as
\begin{equation}
	\label{eq:definition_of_prox_tauGstar}
	\begin{aligned}
		\RTDofsT{i}(\alert{\bp})
		&
		=
		\proj_{\beta \, c_{T,i} B_{\abs{\,\cdot\,}_{s^*}}} \left( \RTDofsT{i}(\bar \bp) \right) 
		\\
		\RTDofsE{j}(\alert{\bp})
		&
		=
		\min \big\{ \abs{\RTDofsE{j}(\bar \bp)} , \; \beta \, \abs{\bn_E}_s \, c_{E,j} \big\} \frac{\RTDofsE{j}(\bar \bp)}{\abs{\RTDofsE{j}(\bar \bp)}}
		.
	\end{aligned}
\end{equation}
In particular we have
\begin{equation*}
	\big[ \RTDofsT{i}(\alert{\bp}) \big]_{\added{\ell}}
	=
	\min \left\{ \bigabs{[ \RTDofsT{i}(\bar \bp) ]_{\added{\ell}}} , \; \beta \, c_{T,i} \right\} \sgn [ \RTDofsT{i}(\bar \bp) ]_{\added{\ell}}
\end{equation*}
\added{for $\ell = 1,2$} when $s = 1$ and
\begin{equation*}
	\RTDofsT{i}(\alert{\bp})
	=
	\min \left\{ \abs{\RTDofsT{i}(\bar \bp)}_2, \; \beta \, c_{T,i} \right\} \frac{\RTDofsT{i}(\bar \bp)}{\abs{\RTDofsT{i}(\bar \bp)}_2}
\end{equation*}
when $s = 2$.
The second formula is understood as $$\RTDofsT{i}(\alert{\bp}) = 0$$ when $\abs{\RTDofsT{i}(\bar \bp)}_2 = 0$.
An implementation of the Cham\-bolle--Pock method is given in \cref{alg:Chambolle-Pock_s_arbitrary}.
\added{Notice that the solution of the $\prox_{\tau G^*}$ problem is independent of the scaling parameter $S > 0$. 
However $S$ enters through the Riesz isomorphism \eqref{eq:Riesz_map_Y_to_Ystar}.}

\begin{algorithm}
	\caption{Chambolle--Pock algorithm for \eqref{eq:DTV-L2} with $s \in [1,\infty]$}
	\label{alg:Chambolle-Pock_s_arbitrary}
	\begin{algorithmic}[1]
		\STATE Set $u^{(0)} \coloneqq f \in \DG{r}(\Omega)$, $\bp^{(0)} \coloneqq \bnull \in \RT{r+1}^0(\Omega)$ and $\bar \bp^{(0)} \coloneqq \bnull \in \RT{r+1}^0(\Omega)$
		\STATE Set $n \coloneqq 0$
		\WHILE{not converged}
		\STATE Set $v^{(n+1)} \coloneqq \added{\div \bar \bp^{(n)}} \in \DG{r}(\Omega)$
			\COMMENT{$v^{(n+1)} = \added{{}-}\Lambda^* \bar \bp^{(n)}$}
		\STATE Set $u^{(n+1)} \coloneqq \prox_{\sigma F}(u^{(n)} + \sigma \, v^{(n+1)})$, see \eqref{eq:definition_of_prox_sigmaF}
			\COMMENT{$u^{(n+1)} = \prox_{\sigma F}(u^{(n)} \added{{}-{}} \sigma \, \Lambda^* \bar \bp^{(n)})$}
		\STATE Set $\bd^{(n+1)} \coloneqq \Lambda u^{(n+1)} \in Y$
		\STATE Set $\bq^{(n+1)} \coloneqq R \, \bd^{(n+1)} \in \RT{r+1}^0(\Omega)$, where $R$ is the Riesz map \eqref{eq:Riesz_map_Y_to_Ystar}
		\STATE Set $\bp^{(n+1)} \coloneqq \prox_{\tau (\beta G)^*}(\bp^{(n)} \added{{}+\tau \, \bq^{(n+1)}})$, see \eqref{eq:definition_of_prox_tauGstar}
			\\
			\COMMENT{$\bp^{(n+1)} = \prox_{\tau (\beta G)^*}(\bp^{(n)} \added{{}+{}} \tau \, R \, \Lambda u^{(n+1)})$}
		\STATE Set $\bar \bp^{(n+1)} \coloneqq \bp^{(n+1)} + \theta \, (\bp^{(n+1)} - \bp^{(n)})$
		\STATE Set $n \coloneqq n+1$
		\ENDWHILE
	\end{algorithmic}
\end{algorithm}

\ifthenelse{\boolean{ispreprint}}%
{
\subsection{Chambolle's Projection Method}
\label{subsec:Chambolle_Projection}

Chambolle's method was introduced in \cite{Chambolle2004} and it solves \eqref{eq:DTV-L2} via its dual \eqref{eq:dualDTV-L2}, specifically in the case $s = s^* = 2$.
\added{We also require $\Omega_0 = \Omega$ here.}
Squaring the constraints pertaining to $\bp \in \beta \bP$, we obtain the Lagrangian
\begin{multline}
	\label{eq:Chambolle_Projection_Lagrangian}
	\frac{1}{2} \norm{\div \bp + f}_{L^2(\Omega)}^2
	+
	\sum_{T,i} \frac{\alpha_{T,i}}{2} \left( \abs{\RTDofsT{i}(\bp)}_2^2 - \beta^2 c_{T,i}^2 \right) 
	\ifthenelse{\boolean{ispreprint}}{}{\\}
	+ 
	\sum_{E,j} \frac{\alpha_{E,j}}{2} \left( \abs{\RTDofsE{j}(\bp)}^2 - \beta^2 c_{E,j}^2 \right)
	,
\end{multline}
where $\alpha_{T,i}$ and $\alpha_{E,j}$ are Lagrange multipliers.
Consequently, the KKT conditions associated with this formulation of \eqref{eq:dualDTV-L2} are
\begin{multline}
	\label{eq:Chambolle_Projection_KKT_1}
	\scalarprod{\div \bp + f}{\div \delta \bp}_{L^2(\Omega)}
	+
	\sum_{T,i} \alpha_{T,i} \, \RTDofsT{i}(\bp) \cdot \RTDofsT{i}(\delta \bp)
	\ifthenelse{\boolean{ispreprint}}{}{\\}
	+ 
	\sum_{E,j} \alpha_{E,j} \, \RTDofsE{j}(\bp) \, \RTDofsE{j}(\delta \bp)
	= 
	0
\end{multline}
for all $\delta \bp \in \RT{r+1}^0(\Omega)$, together with the complementarity conditions
\begin{subequations}
	\label{eq:Chambolle_Projection_KKT_2}
	\begin{align}
		0 
		\le 
		\alpha_{T,i} 
		&
		\quad \perp \quad
		\abs{\RTDofsT{i}(\bp)}_2 - \beta \, c_{T,i} 
		\le 
		0
		\label{eq:Chambolle_Projection_KKT_2a}
		\ifthenelse{\boolean{ispreprint}}{}{\\}
		& &
		\text{for all $T$ and $i = 1, \ldots, r \, (r+1)/2$ and}
		\ifthenelse{\boolean{ispreprint}}{}{\notag}
		\\
		0 
		\le 
		\alpha_{E,j} 
		&
		\quad \perp \quad
		\abs{\RTDofsE{j}(\bp)} - \beta \, c_{E,j} 
		\le 
		0
		\label{eq:Chambolle_Projection_KKT_2b}
		\ifthenelse{\boolean{ispreprint}}{}{\\}
		& &
		\text{for all $E$ and $j = 1, \ldots, r+1$.}
		\ifthenelse{\boolean{ispreprint}}{}{\notag}
	\end{align}
\end{subequations}
Let us observe that the first term in \eqref{eq:Chambolle_Projection_KKT_1} can be written as $- \dual{\Lambda (\div \bp + f)}{\delta \bp}_{Y,Y^*}$, and hence as
\begin{equation*}
	- \sum_T \int_T \restr{\nabla u}{T} \cdot \delta \bp \, \dx
	- \sum_E \int_E \jump{u} \, (\delta \bp \cdot \bn_E) \, \ds,
\end{equation*}
where we set $u \coloneqq \div \bp + f$ as an abbreviation in accordance with \eqref{eq:Recovery_u_from_p_DTVL2}.
By selecting directions $\delta \bp$ from the collections $\{ \RTBasisT{i} \}$ and $\{ \RTBasisE{j} \}$ of $\RT{r+1}^0(\Omega)$ basis functions, see \cref{sec:Finite_Element_Spaces}, we infer that \eqref{eq:Chambolle_Projection_KKT_1} is equivalent to 
\begin{subequations}
	\label{eq:Chambolle_Projection_KKT_3}
	\begin{align}
		- \nabla u(\LagrangeNodesTrmone{i}) + \alpha_{T,i} \, \RTDofsT{i}(\bp) 
		&
		= 
		0
		\label{eq:Chambolle_Projection_KKT_3a}
		\ifthenelse{\boolean{ispreprint}}{}{\\}
		& &
		\text{for all } T \text{ and } i = 1, \ldots, r \, (r+1)/2
		\ifthenelse{\boolean{ispreprint}}{}{\notag}
		\\
		- \jump{u}(\LagrangeNodesEr{j}) + \alpha_{E,j} \, \RTDofsE{j}(\bp)
		&
		=
		0
		\label{eq:Chambolle_Projection_KKT_3b}
		\ifthenelse{\boolean{ispreprint}}{}{\\}
		& &
		\text{for all } E \text{ and } j = 1, \ldots, r+1
		\ifthenelse{\boolean{ispreprint}}{}{\notag}
		.
	\end{align}
\end{subequations}
A simple calculation similar as in \cite{Chambolle2004} then shows that \eqref{eq:Chambolle_Projection_KKT_2} and \eqref{eq:Chambolle_Projection_KKT_3} imply
\begin{equation}
	\begin{aligned}
		\beta \, \alpha_{T,i} \, c_{T,i} 
		& 
		= 
		\abs{\nabla u(\LagrangeNodesTrmone{i})}_2
		,
		\\
		\beta \, \alpha_{E,j} \, c_{E,j} 
		& 
		= 
		\bigabs{\jump{u}(\LagrangeNodesEr{j})}
		.
		\label{eq:Chambolle_Projection_KKT_4}
	\end{aligned}
\end{equation}
In order to re-derive Chambolle's algorithm for the setting at hand, it remains to rewrite the directional derivative \eqref{eq:Chambolle_Projection_KKT_1} in terms of the gradient $\bg \in Y^*$ w.r.t.\ the $Y^*$ inner product \eqref{eq:inner_product_in_Ystar}.
We obtain that $\bg$ is given by its coefficients
\begin{subequations}
	\label{eq:Chambolle_Projection_gradient}
	\begin{align}
		\RTDofsT{i}(\bg) 
		&
		= 
		c_{T,i} \, \big( \alpha_{T,i} \, \RTDofsT{i}(\bp) - \nabla u(\LagrangeNodesTrmone{i}) \big)
		,
		\label{eq:Chambolle_Projection_gradienta}
		\\
		\RTDofsE{j}(\bg) 
		&
		=
		c_{E,j} \, \big( \alpha_{E,j} \, \RTDofsE{j}(\bp) - \jump{u}(\LagrangeNodesEr{j}) \big)
		.
		\label{eq:Chambolle_Projection_gradientb}
	\end{align}
\end{subequations}
Given an iterate for $\bp$, the main steps of the algorithm are then to update the auxiliary quantity $u = \div \bp + f$ as well as the multipliers $\alpha_{T,i}$ and $\alpha_{E,j}$ according to \eqref{eq:Chambolle_Projection_KKT_4}, and take a semi-implicit gradient step with a suitable step length to update $\bp$.
Since all of these steps are inexpensive, Chambolle's method can be implemented just as efficiently as its finite difference version originally given in \cite{Chambolle2004}. 
For the purpose of comparison, we point out that one step of the method can be written compactly as
\begin{equation*}
	\begin{aligned}
		\RTDofsT{i}(\bp^{(n+1)}) 
		&
		\coloneqq 
		\frac{\RTDofsT{i}(\bp^{(n)}) + \tau \, c_{T,i} \nabla (\div \bp^{(n)} + f)(\LagrangeNodesTrmone{i})}{1 + \tau \, \beta^{-1} \bigabs{\nabla (\div \bp^{(n)} + f)(\LagrangeNodesTrmone{i})}_2}
		,
		\\
		\RTDofsE{j}(\bp^{(n+1)}) 
		&
		\coloneqq 
		\frac{\RTDofsE{j}(\bp^{(n)}) + \tau \, c_{E,j} \jump{\div \bp^{(n)} + f}(\LagrangeNodesEr{j})}{1 + \tau \, \beta^{-1} \bigabs{\jump{\div \bp^{(n)} + f}(\LagrangeNodesEr{j})}}
		.
	\end{aligned}
\end{equation*}
for all $T$ and $i$, and for all $E$ and $j$, respectively.
Let us mention that our variable $\bp$ differs by a factor of $\beta$ from the one used in \cite{Chambolle2004}.
Moreover, in the implemention given as \cref{alg:Chambolle_Projection_s=2}, we found it convenient to rename $\alpha_{T,i} \, c_{T,i}$ as $\gamma_{T,i}$, and similarly for the edge based quantities.
Notice that $\gamma_{T,i}$ and $\gamma_{E,j}$ can be conveniently stored, for instance, as the coefficients of a $\DG{r-1}(\Omega)$ function, and another $\DG{r}$ function on the skeleton of the mesh, i.e., the union of all interior edges.

\begin{algorithm}
	\caption{Chambolle's algorithm for \eqref{eq:DTV-L2} with $s = 2$}
	\label{alg:Chambolle_Projection_s=2}
	\begin{algorithmic}[1]
		\STATE Set $\bp^{(0)} \coloneqq \bnull \in \RT{r+1}^0(\Omega)$
		\STATE Set $n \coloneqq 0$
		\WHILE{not converged}
		\STATE Set $u^{(n)} \coloneqq \div \bp^{(n)} + f \in \DG{r}(\Omega)$
		\STATE Set $\gamma_{T,i} \coloneqq \beta^{-1} \abs{\nabla u^{(n)}(\LagrangeNodesTrmone{i})}_2$ 
			\COMMENT{$\gamma_{T,i} = \alpha_{T,i} \, c_{T,i}$, see \eqref{eq:Chambolle_Projection_KKT_4}}
		\STATE Set $\gamma_{E,j} \coloneqq \beta^{-1} \abs{\jump{u^{(n)}}(\LagrangeNodesEr{j})}$
			\COMMENT{$\gamma_{E,j} = \alpha_{E,j} \, c_{E,j}$, see \eqref{eq:Chambolle_Projection_KKT_4}}
		\STATE Set $\RTDofsT{i}(\bp^{(n+1)}) \coloneqq \displaystyle \frac{\RTDofsT{i}(\bp^{(n)}) + \tau \, c_{T,i} \nabla u^{(n)}(\LagrangeNodesTrmone{i})}{1 + \tau \, \gamma_{T,i}}$
		\STATE Set $\RTDofsE{j}(\bp^{(n+1)}) \coloneqq \displaystyle \frac{\RTDofsE{j}(\bp^{(n)}) + \tau \, c_{E,j} \jump{u^{(n)}}(\LagrangeNodesEr{j})}{1 + \tau \, \gamma_{E,j}}$
		\STATE Set $n \coloneqq n+1$
		\ENDWHILE
	\end{algorithmic}
\end{algorithm}

\subsection{Primal-Dual Active Set Method}
\label{subsec:PDAS}

We consider a primal-dual active set (PDAS) strategy for the dual problem \eqref{eq:dualDTV-L2}.
A similar approach was proposed in \cite{HintermuellerKunisch2004:2}, however in the context of finite difference approximation and an additional regularization of the dual problem.
The PDAS method is closely related to a semi-smooth Newton approach, see \cite{HintermuellerItoKunisch2002}, and it is based on the associated KKT conditions and a semi-smooth reformulation of the complementarity conditions associated with the constraints $\bp \in \beta \bP$.
The approach is particularly suitable when $s = 1$ and thus the constraints describing $\bP$ are simple bounds.
We thus focus on the case $s = 1$.
\added{Moreover, we assume again $\Omega_0 = \Omega$.}
Then the KKT conditions associated with \eqref{eq:dualDTV-L2} can be written as follows:
\begin{multline}
	\label{eq:PDAS_KKT_1}
	\ifthenelse{\boolean{ispreprint}}{\hfill}{}
		\scalarprod{\div \bp + f}{\div \delta \bp}_{L^2(\Omega)}
	+ 
	\sum_{T,i} \bmu_{T,i} \cdot \RTDofsT{i}(\delta \bp)
	\ifthenelse{\boolean{ispreprint}}{}{\\}
	+ 
	\sum_{E,j} \mu_{E,j} \, \RTDofsE{j}(\delta \bp)
	=
	0
	\ifthenelse{\boolean{ispreprint}}{\hfill}{}
\end{multline}
for all $\delta \bp \in \RT{r+1}^0(\Omega)$, together with the complementarity conditions
\begin{subequations}
	\label{eq:PDAS_KKT_2}
	\begin{align}
		& 
		\bmu_{T,i} 
		= 
		\max \left\{ 0, \bmu_{T,i} + c \, \big( \RTDofsT{i}(\bp) \! - \! \beta \, c_{T,i} \, \bone \big) \right\}
		\ifthenelse{\boolean{ispreprint}}{}{\notag \\ & \quad} 
		+
		\min \left\{ 0, \bmu_{T,i} + c \, \big( \RTDofsT{i}(\bp) \! + \! \beta \, c_{T,i} \, \bone \big) \right\}
		,
		\label{eq:PDAS_KKT_2a}
		\\
		& 
		\mu_{E,j} 
		= 
		\max \left\{ 0, \mu_{E,j} + c \, \big( \RTDofsE{j}(\bp) \! - \! \beta \, \abs{\bn_E}_1 \, c_{E,j} \big) \right\}
		\ifthenelse{\boolean{ispreprint}}{}{\notag \\ & \quad} 
		+
		\min \left\{ 0, \mu_{E,j} + c \, \big( \RTDofsE{j}(\bp) \! + \! \beta \, \abs{\bn_E}_1 \, c_{E,j} \big) \right\}
		,
		\label{eq:PDAS_KKT_2b}
	\end{align}
\end{subequations}
where $c > 0$ is arbitrary.
Notice that, as is customary for bound constrained problems, we are using signed multipliers $\bmu_{T,i}$ and $\mu_{E,j}$.
Moreover, \eqref{eq:PDAS_KKT_2a} is understood componentwise in $\R^2$.
The semi-smooth linearization of \eqref{eq:PDAS_KKT_2} agrees with a piecewise linearization on the three branches possible per expression.
When we write the (non-globalized) semi-smooth Newton method in terms of the subsequent iterate, we arrive at \cref{alg:PDAS_s=1}.

Notice that the solution of \eqref{eq:PDAS_reduced_problem_s=1} in \cref{alg:PDAS_s=1} is not necessarily unique.
This is not an obstacle when \eqref{eq:PDAS_reduced_problem_s=1} is solved iteratively, e.g., by the conjugate gradient method.
Alternatively, we might add the regularizing term $(\varepsilon/2) \norm{\bp}_{Y^*}^2$ to the objective.
In this case, also the multplier update on the active sets must be replaced by
\begin{equation*}
	\begin{aligned}
		\big[\bmu_{T,i}^{(n+1)}\big]_{1,2}
		&
		\coloneqq 
		\big[\nabla u^{(n+1)}(\LagrangeNodesTrmone{i})\big]_{1,2}
		-
		\frac{\varepsilon}{c_{T,i}\added{\scaling}} \big[\RTDofsT{i}(\bp^{(n+1)})\big]_{1,2}
		,
		\\
		\mu_{E,j}^{(n+1)}
		&
		\coloneqq 
		\jump{u^{(n+1)}}(\LagrangeNodesEr{j})
		-
		\frac{\varepsilon}{c_{E,j}} \RTDofsE{j}(\bp^{(n+1)})
		.
	\end{aligned}
\end{equation*}
This modification amounts to employing a Huber regularization to $\abs{u}_{DTV(\Omega)}$, see \cref{subsec:Huber_TV-Seminorm}.

\begin{algorithm}
	\caption{Primal-dual active set method for \eqref{eq:dualDTV-L2} with $s = 1$}
	\label{alg:PDAS_s=1}
	\begin{algorithmic}[1]
		\STATE Set $\bp^{(0)} \coloneqq \bnull \in \RT{r+1}^0(\Omega)$ and $\bmu^{(0)} \coloneqq \bnull$
		\STATE Set $n \coloneqq 0$
		\WHILE{not converged}
		\STATE Determine the active sets
			\begin{equation*}
				\begin{aligned}
					\AA_T^{\pm,1}
					&
					\coloneqq
					\left\{ (T,i): \pm [\bmu_{T,i}^{(n)} + c \, \RTDofsT{i}(\bp^{(n)})]_1 > c \, \beta \, c_{T,i} \right\}
					,
					\\
					\AA_T^{\pm,2}
					&
					\coloneqq
					\left\{ (T,i): \pm [\bmu_{T,i}^{(n)} + c \, \RTDofsT{i}(\bp^{(n)})]_2 > c \, \beta \, c_{T,i} \right\}
					,
					\\
					\AA_E^\pm
					&
					\coloneqq
					\left\{ (E,j): \pm [\mu_{E,j}^{(n)} + c \, \RTDofsE{j}(\bp^{(n)})] > c \, \beta \, \abs{\bn_E}_1 \, c_{E,j} \right\}
				\end{aligned}
			\end{equation*}
		\STATE Solve for $\alert{\bp} \in \RT{r+1}^0(\Omega)$ and assign the solution to $\bp^{(n+1)}$
			\begin{equation}
				\label{eq:PDAS_reduced_problem_s=1}
					\quad
					\begin{aligned}
						\text{Minimize} \quad 
						& 
						\frac{1}{2} \norm{\div \alert{\bp} + f}_{L^2(\Omega)}^2,
						\\
						\text{s.t.} 
						\quad
						& 
						\left\{
							\quad
							\begin{aligned}
								&
								[\RTDofsT{i}(\alert{\bp})]_1 = \pm \beta \, c_{T,i}
								& &
								\text{where } (T,i) \in \AA_T^{\pm,1} 
								\\
								& 
								[\RTDofsT{i}(\alert{\bp})]_2 = \pm \beta \, c_{T,i}
								& &
								\text{where } (T,i) \in \AA_T^{\pm,2} 
								\\
								&
								\RTDofsE{j}(\alert{\bp}) = \pm \beta \, \abs{\bn_E}_1 \, c_{E,j}
								& &
								\text{where } (E,j) \in \AA_E^\pm 
							\end{aligned}
						\right.
					\end{aligned}
			\end{equation}
		\STATE Set $u^{(n+1)} \coloneqq \div \bp^{(n+1)} + f \in \DG{r}(\Omega)$
		\STATE Set
			\begin{equation*}
				\begin{aligned}
					\big[\bmu_{T,i}^{(n+1)}\big]_1 
					&
					\coloneqq 
					\big[\nabla u^{(n+1)}(\LagrangeNodesTrmone{i})\big]_1
					& &
					\text{where } (T,i) \in \AA_T^{\pm,1} 
					\\
					\big[\bmu_{T,i}^{(n+1)}\big]_2 
					&
					\coloneqq 
					\big[\nabla u^{(n+1)}(\LagrangeNodesTrmone{i})\big]_2
					& &
					\text{where } (T,i) \in \AA_T^{\pm,2} 
					\\
					\mu_{E,j}^{(n+1)}
					&
					\coloneqq 
					\jump{u^{(n+1)}}(\LagrangeNodesEr{j})
					& &
					\text{where } (E,j) \in \AA_E^\pm
				\end{aligned}
			\end{equation*}
			and zero elsewhere
		\STATE Set $n \coloneqq n+1$
		\ENDWHILE
	\end{algorithmic}
\end{algorithm}
}{}

\section{Implementation Details}
\label{sec:Implementation_Details}

Our implementation was carried out in the finite element framework \fenics\ (version~2017.2).
We refer the reader to \cite{LoggMardalWells2012:1,AlnaesBlechtaHakeJohanssonKehletLoggRichardsonRingRognesWells2015} for background reading.
\fenics\ supports finite elements of various types\added{ on simplicial meshes}, including $\CG{r}$, $\DG{r}$ and $\RT{r+1}$ elements of arbitrary order.
Although we focus on this piece of software, the content of this section will apply to other finite element frameworks as well.

While the bases for the spaces $\CG{r}$ and $\DG{r}$ in \fenics\ are given by the standard nodal basis functions as described in \cref{sec:Finite_Element_Spaces}, the implementation of $\RT{r+1}$ elements in \fenics\ uses degrees of freedom based on point evaluations of $\bp$ and $\bp \cdot \bn_E$, rather than the integral-type dofs in \eqref{eq:RTr+1_dofs}.
Since we wish to take advantage of the simple structure of the constraints in the dual representation \eqref{eq:dual_representation_of_DTV} of $\abs{u}_{DTV(\Omega)}$ however, we rely on the choice of dofs described in \eqref{eq:RTr+1_dofs}.
In order to avoid a global basis transformation, we implemented our own version of the $\RT{r+1}$ finite element in \fenics.

Our implementation uses the dofs in \eqref{eq:RTr+1_dofs} on the reference cell $\widehat T$.
As usual in finite element methods, an arbitrary cell $T$ is then obtained via an affine geometry transformation, i.e.,
\begin{equation*}
	G_T: \widehat T \to T, 
	\qquad
	G_T(\widehat x) = B_T \, \widehat x + b_T,
\end{equation*}
where $B_T \in \R^{2 \times 2}$ is a non-singular matrix and $b_T \in \R^2$.
We mention that $B_T$ need not necessarily have a positive determinant, i.e., the transformation $G_T$ may not necessarily be orientation preserving.
In contrast to $\CG{}$ and $\DG{}$ elements, a second transformation is required to define the dofs and basis functions on the world cell $T$ from the dofs and basis functions on $\widehat T$.
For the ($\bH(\div;\Omega)$-conforming) $\RT{}$ spaces, this is achieved via the (contravariant) Piola transform; see for instance \cite[Ch.~1.4.7]{ErnGuermond2004} or \cite{RognesKirbyLogg2009:1}.
In terms of functions $\widehat \bp$ from the local polynomial space, we have
\begin{equation*}
	\begin{aligned}
		&
		P_T: \PP_r(\widehat T)^2 + \widehat \bx \, \PP_r(\widehat T) \to \PP_r(T)^2 + \bx \, \PP_r(T),
		\\
		&
		P_T(\widehat \bp) = (\det B_T^{-1}) \, B_T \, [ \widehat \bp \circ G_T^{-1} ].
	\end{aligned}
\end{equation*}
The Piola transform preserves tangent directions on edges, as well as normal traces of vector fields, up to edge lengths.
It satisfies
\begin{equation}
	\label{eq:properties_of_Piola_transform}
	\abs{\widehat E} \, \widehat \bp \cdot \widehat \bn_{\widehat E}
	= 
	\pm \abs{E} \, \bp \cdot \bn_E 
	\quad
	\text{and}
	\quad
	\abs{\widehat T} \, B_T \, \widehat \bp
	= 
	\pm \abs{T} \, \bp,
\end{equation}
where $\widehat E$ is an edge of $\widehat T$, $\widehat \bn_{\widehat E}$ is the corresponding unit outer normal, $E = G_T(\widehat E)$, $\bn_E$ is a unit normal vector on $E$ with arbitrary orientation, $\bp = P_T(\widehat \bp)$, and $\abs{T}$ is the area of $T$; see for instance \cite[Lem.~1.84]{ErnGuermond2004}.

We denote by $\wRTDofsT{i}$ and $\wRTDofsE{j}$ the degrees of freedom as in \eqref{eq:RTr+1_dofs}, defined in terms of the nodal basis functions $\wLagrangeBasisTrmone{i} \in \PP_{r-1}(\widehat T)$ and $\wLagrangeBasisEr{j} \in \PP_r(\widehat E)$ on the reference cell. 
Let us consider how these degrees of freedom act on the world cell.
Indeed, the relations above imply 
\begin{subequations}
	\label{eq:transformation_of_RT+1_dofs}
	\begin{align}
		\wRTDofsT{i}(\widehat \bp)
		&
		\coloneqq
		\int_{\widehat T} \wLagrangeBasisTrmone{i} \, \widehat \bp \, \d\widehat x
		\notag
		\\
		&
		= 
		\pm \int_T \LagrangeBasisTrmone{i} \, B_T^{-1} \, \bp \, \dx
		=: 
		\pm \tRTDofsT{i}(\bp)
		,
		\label{eq:transformation_of_RT+1_dofs_T}
		\\
		\wRTDofsE{j}(\widehat \bp) 
		&
		\coloneqq
		\int_{\widehat E} \, \wLagrangeBasisEr{j} \, (\widehat \bp \cdot \widehat \bn_{\widehat E}) \, \d\widehat s
		\notag
		\\
		&
		= 
		\pm \int_E \, \LagrangeBasisEr{j} \, (\bp \cdot \bn_E) \, \ds
		= 
		\pm \RTDofsE{j}(\bp)
		,
		\label{eq:transformation_of_RT+1_dofs_E}
	\end{align}
\end{subequations}
where we used that Lagrangian basis functions are transformed according to $\LagrangeBasisTrmone{i} = \wLagrangeBasisTrmone{i} \circ G_T^{-1}$, and similarly for the edge-based quantities.
The correct choice of the sign in \eqref{eq:properties_of_Piola_transform} and \eqref{eq:transformation_of_RT+1_dofs} depends on the sign of $\det B_T$ and on the relative orientations of $P_T(\widehat \bn_{\widehat E})$ and $\bn_E$.
However the sign is not important since all operations depending on the dofs or coefficients, such as $\RTDofsT{i}(\bp)$, are sign invariant, notably the constraint set in \eqref{eq:constraint_set}.

Notice that while \eqref{eq:transformation_of_RT+1_dofs_E} agrees (possibly up to the sign) with our preferred set of edge-based dofs \eqref{eq:RTr+1_dofs_E}, the interior dofs $\tRTDofsT{i}$ available through the transformation \eqref{eq:transformation_of_RT+1_dofs_T} are related to the desired dofs $\RTDofsT{i}$ from \eqref{eq:RTr+1_dofs_T} via
\begin{equation}
	\label{eq:RTr+1_dofs_T_available_vs_desired}
	\RTDofsT{i}(\bp)
	=
	\sgn (\det B_T) \, B_T^\top \, \tRTDofsT{i}(\bp).
\end{equation}
Notice that this transformation is impossible to avoid since the dofs \eqref{eq:RTr+1_dofs_T} are not invariant under the Piola transform.
However, \eqref{eq:RTr+1_dofs_T_available_vs_desired} is completely local to the triangle and inexpensive to evaluate.
Although not required for our numerical computations, we mention for completeness that the corresponding dual basis functions are related via
\begin{equation}
	\label{eq:RTr+1_basis_functions_T_available_vs_desired}
	\RTBasisT{i}
	=
	\sgn (\det B_T) \, \tRTBasisT{i} B_T^{-\top}.
\end{equation}
To summarize this discussion, functions $\bp \in \RT{r+1}(\Omega)$ will be represented in terms of coefficients w.r.t.\ the dofs $\{\RTDofsE{j}\}$ and $\{\tRTDofsT{i}\}$ in our \fenics\ implementation of the $\RT{}$ space.
Transformations to and from the desired dofs $\{\RTDofsT{i}\}$ will be performed for all operations manipulating directly the coefficients of an $\RT{r+1}$ function.
For instance, the projection operation in \eqref{eq:definition_of_prox_tauGstar} (for the Chambolle--Pock \cref{alg:Chambolle-Pock_s_arbitrary}) in the case $s = 2$ would be implemented as
\begin{multline*}
	\ifthenelse{\boolean{ispreprint}}{\hfill}{}
	\tRTDofsT{i}(\bp)
	=
	B_T^{-\top} \min \left\{ \abs{B_T^\top \, \tRTDofsT{i}(\bar \bp)}_2, \; \beta \, c_{T,i} \right\} 
	\ifthenelse{\boolean{ispreprint}}{}{\\}
	\cdot
	\frac{B_T^\top \, \tRTDofsT{i}(\bar \bp)}{\abs{B_T^\top \, \tRTDofsT{i}(\bar \bp)}_2}
	.
	\ifthenelse{\boolean{ispreprint}}{\hfill}{}
\end{multline*}

\section{Numerical Results for \eqref{eq:DTV-L2}}
\label{sec:Numerical_results_DTV-L2}

In this section we present some numerical results for \eqref{eq:DTV-L2} \added{in the isotropic case ($s = 2$).}
\added{Our goals are to compare the convergence behavior and computational efficiency for \cref{alg:split_Bregman_s_arbitrary,alg:Chambolle-Pock_s_arbitrary} w.r.t.\ varying polynomial degree $r \in \{0,1,2\}$, and to exhibit the benefits of polynomial orders $r \ge 1$ for image quality, both for denoising and inpainting applications.}

\begin{figure}[htp]
	\begin{center}
		\includegraphics[width=0.32\linewidth]{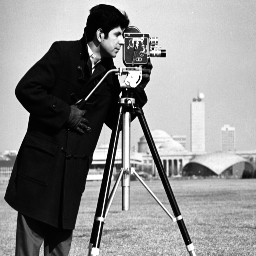}
		\hfill
		\includegraphics[width=0.32\linewidth]{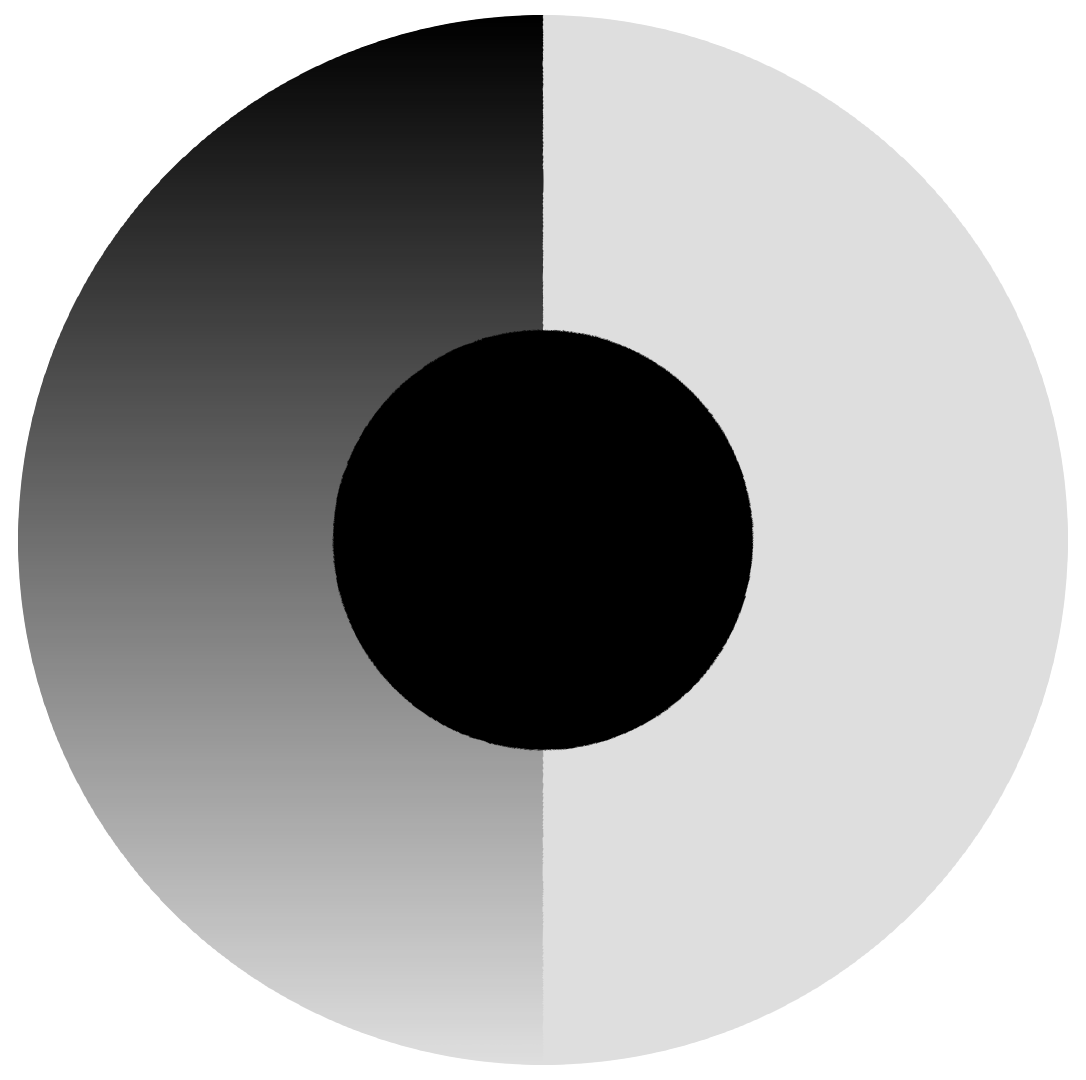}
		\hfill
		\includegraphics[width=0.32\linewidth]{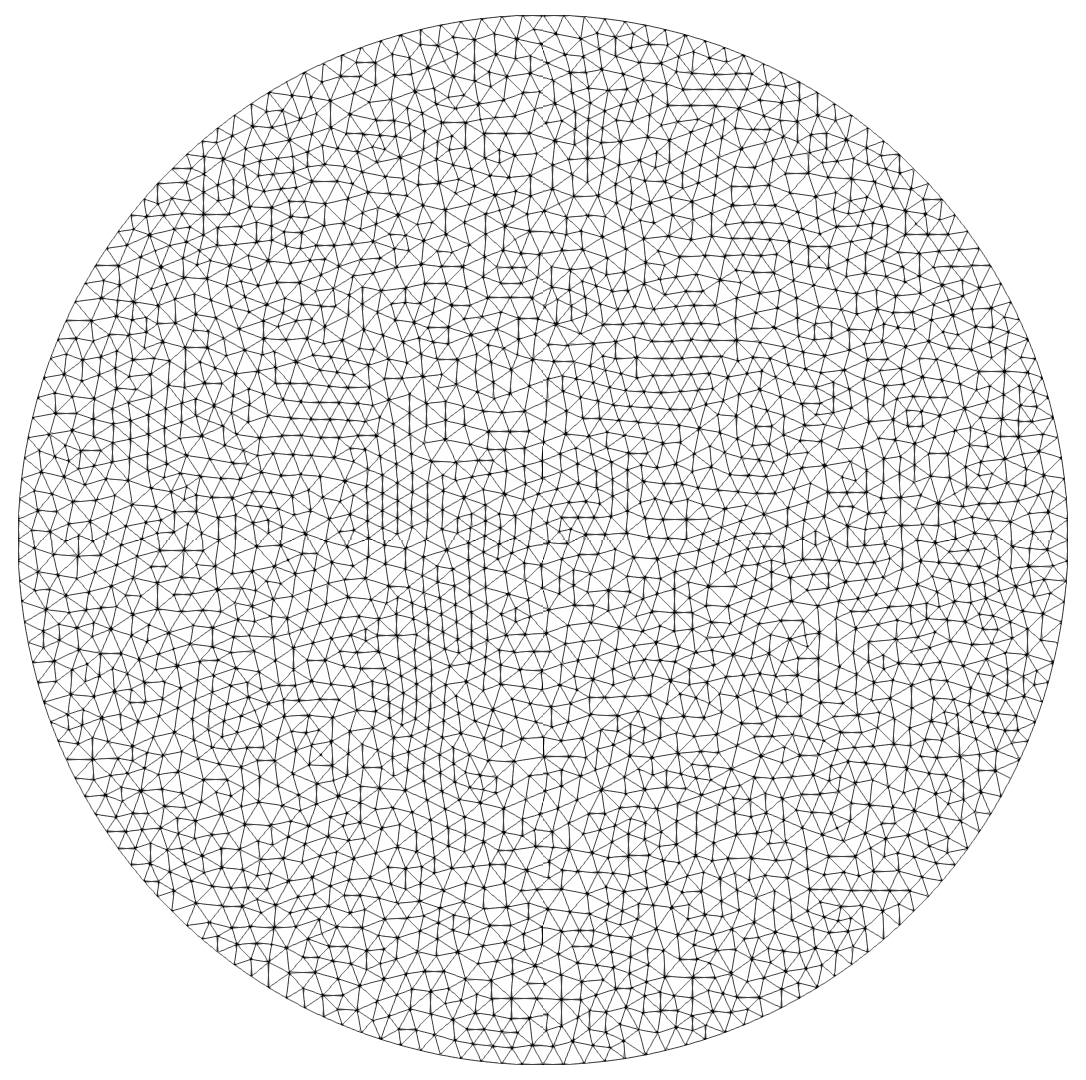}
	\end{center}
	\caption{\added{Left: Cameraman pixel test image. Middle: Non-discrete test image. Right: Mesh used to represent the image in the middle.}}
	\label{fig:InitImage}
\end{figure}
\added{In our tests, we use the two images displayed in \cref{fig:InitImage}. 
	Both have data in the range~$[0,1]$.
	The discrete cameraman image has a resolution of $256 \times 256$ square pixels and will be interpolated onto a $\DG{r}(\Omega)$ space on a triangular grid with crossed diagonals, so the mesh has \num[scientific-notation=false]{262144}~cells and \num[scientific-notation=false]{131585}~vertices.
	We are also using a low resolution version of the cameraman image on a $64 \times 64$~square grid in \cref{subsec:denoising_on_low_resolution_meshes}.
	The second is a non-discrete image on a circle of radius~0.5.
	The corresponding discrete problems are set up on a mesh consisting of \num[scientific-notation=false]{5460} cells and \num[scientific-notation=false]{2811} vertices.
	For each problem, the dimension of the finite element space for the image~$u$ is given in \cref{tab:dimensions_of_test_problems}.
	In all of the following tests, noise is added to each degree of freedom in the form of a normally distributed random variable with standard deviation $\sigma = \num{1e-1}$ and zero mean. 
	Our implementation uses the finite element framework \fenics\ (version~2017.2).
	All experiments were conducted on a standard desktop PC \added{with an Intel~i5-4690 CPU running at 3.50~Ghz, 16~GB RAM and Linux openSUSE Leap~42.1.}
	Visualization was achieved in \paraview.
}

\begin{table*}[htbp]
	\sisetup{scientific-notation=false}
	\centering
	\begin{tabular}{@{}lrrrrr@{}}
		\toprule
		image       & \# of cells $N_T$ & \# of vertices $N_V$ & $\dim \DG{0}(\Omega)$ & $\dim \DG{1}(\Omega)$ & $\dim \DG{2}(\Omega)$ \\
		\midrule
		cameraman   & \num{262144}      & \num{131585}         & \num{262144}          & \num{786432}          & \num{1572864} \\
		cameraman64 & \num{16384}       & \num{8321}           & \num{16384}           & \num{49152}           & \num{98304} \\
		ball        & \num{5460}        & \num{2811}           & \num{5460}            & \num{16380}           & \num{32760} \\
		\midrule 
		\bottomrule \\
	\end{tabular}
	\caption{Dimensions of the $\DG{r}$ spaces for our test images depending on the polynomial degree $r \in \{0,1,2\}$.}
	\label{tab:dimensions_of_test_problems}
\end{table*}

A stopping criterion for%
\ifthenelse{\boolean{ispreprint}}%
	{
		\cref{alg:split_Bregman_s_arbitrary,alg:Chambolle-Pock_s_arbitrary,alg:Chambolle_Projection_s=2,alg:PDAS_s=1} 
	}
	{
		\added{\cref{alg:split_Bregman_s_arbitrary,alg:Chambolle-Pock_s_arbitrary}} 
	}%
can be based on the primal-dual gap
\begin{equation}
	\label{eq:primal-dual_gap_dualTV-L2}
	F(u) + \beta \, G(\Lambda u) + F^*(\Lambda^* \bp) + \beta \, G^*(\bp/\beta)
	.
\end{equation}
Notice that since $G^* = I_{\bP}$ is the indicator function of the constraint set $\bP$, the last term is either~$0$ or $\infty$, and \eqref{eq:primal-dual_gap_dualTV-L2} can therefore not directly serve as a meaningful stopping criterion.
Instead, we omit the last term in \eqref{eq:primal-dual_gap_dualTV-L2} and introduce a distance-to-feasibility measure for $\bp$ as a second criterion.
For the latter, we utilize the difference of $\bp$ and its $Y^*$-orthogonal projection onto $\beta \bP$, measured in the $Y^*$-norm squared.
This expression can be easily evaluated when $s \in \{1,2\}$.
Straightforward calculations then show that we obtain the following specific expressions:
\begin{subequations}
	\label{eq:stopping_criterion}
	\begin{multline}
		\primaldualgap{u}{\bp}
		\coloneqq
		\frac{1}{2} \norm{u-f}_{L^2(\added{\Omega_0})}^2
		+
		\frac{1}{2} \norm{\div \bp+f}_{L^2(\added{\Omega_0})}^2
		\\
		-
		\frac{1}{2} \norm{f}_{L^2(\added{\Omega_0})}^2
		+ 
		\beta \sum_T \int_T \II_T \big\{ \abs{\nabla u}_s \big\} \, \dx
		\ifthenelse{\boolean{ispreprint}}{}{\\}
		+ 
		\beta \sum_E \int_E \II_E \big\{ \bigabs{\vjump{u}}_s \big\} \, \ds
		\label{eq:stopping_criterion_1}
	\end{multline}
	and
	\begin{multline}
		\label{eq:stopping_criterion_2}
		\infeasibility{2}{\bp}
		\coloneqq
		\sum_{T,i} \frac{1}{c_{T,i}\added{\scaling}} \max \big\{ \abs{\RTDofsT{i}(\bp)}_2 - \beta \, c_{T,i}, \; 0 \big\}^2
		\\
		+ 
		\sum_{E,j} \frac{1}{c_{E,j}} \max \big\{ \abs{\RTDofsE{j}(\bp)} - \beta \, c_{E,j}, \; 0 \big\}^2 
	\end{multline}
	when $s = 2$, as well as
	\begin{multline}
		\label{eq:stopping_criterion_3}
		\infeasibility{1}{\bp}
		\ifthenelse{\boolean{ispreprint}}{}{\\}
			\coloneqq
			\sum_{T,i} \frac{1}{c_{T,i}\added{\scaling}} \sum_{\ell=1}^2 \max \big\{ \bigabs{[\RTDofsT{i}(\bp)]_\ell} - \beta \, c_{T,i}, \; 0 \big\}^2
			\\
			+
			\sum_{E,j} \frac{1}{c_{E,j}} \max \big\{ \abs{\RTDofsE{j}(\bp)} - \beta \, \abs{\bn_E}_s \, c_{E,j}, \; 0 \big\}^2 
		\end{multline}
	when $s = 1$.
\end{subequations}
\added{In our numerical experiments, we focus on the case $s = 2$ and we stop either algorithm as soon as the iterates $(u,\bp)$ satisfy the following conditions:}
\begin{equation}
	\label{stopping_criterion}
	\begin{aligned}
		\abs{\primaldualgap{u}{\bp}} 
		&
		\le 
		\varepsilon_\text{rel} \primaldualgap{f}{\bnull}\\
		\infeasibility{2}{\bp} 
		&
		\le 
		\num{1e-11}
	\end{aligned}
\end{equation}
\added{with $\varepsilon_\text{rel} = \num{1e-3}$. 
	As a measurement for the quality of our results we use the common peak signal-to-noise ratio, defined by
}
\begin{equation}
	\label{PSNR}
	\psnr(u,\uref) = 10 \log_{10} \left(\frac{M^2 \, \abs{\Omega}}{\norm{u-\uref}_{L^2(\Omega)}^2}\right) ,
\end{equation}
\added{where $u$ is the recovered image, $\uref$ is the reference image, and $\abs{\Omega}$ is the area of the image.
	Moreover, $M = 1$ is the maximum possible image value.}

\subsection{Denoising of $\DG{r}$-Images}
\label{subsec:denoising_of_DGr_in_DGr}

\added{This section addresses the denoising of $\DG{r}$~images and it also serves as a comparative study of \cref{alg:split_Bregman_s_arbitrary,alg:Chambolle-Pock_s_arbitrary}.
We represent (interpolate) the \added{non-discrete} image displayed in~\cref{fig:InitImage}~(middle) in the space $\DG{r}(\Omega)$ for $r=0,1,2$. 
Noise is added to each degree of freedom as described above. 
We show the denoising results for the split Bregman method (\cref{alg:split_Bregman_s_arbitrary}) in \cref{fig:TestCase1}.
The results for the Chambolle-Pock approach (\cref{alg:Chambolle-Pock_s_arbitrary}) are very similar and are therefore not shown.
In either case, the noise is removed successfully. 
The infeasibility criterion \eqref{eq:stopping_criterion_2} in the final iteration was smaller than $10^{-37}$ for \cref{alg:split_Bregman_s_arbitrary} and smaller than $10^{-11}$ for \cref{alg:Chambolle-Pock_s_arbitrary} in all cases $r \in \{0,1,2\}$.
\Cref{tab:data_test_1} summarizes the convergence bevahior of both methods.
Since the split Bregman method performed slightly better w.r.t.\ iteration count and run-time in our implementation, we will use only \cref{alg:split_Bregman_s_arbitrary} for the subsequent denoising examples (\cref{subsec:Comparison_to_pixel_grids,subsec:denoising_on_low_resolution_meshes}).}
\begin{figure}[htbp]
	\begin{center}
		\includegraphics[width=0.32\linewidth]{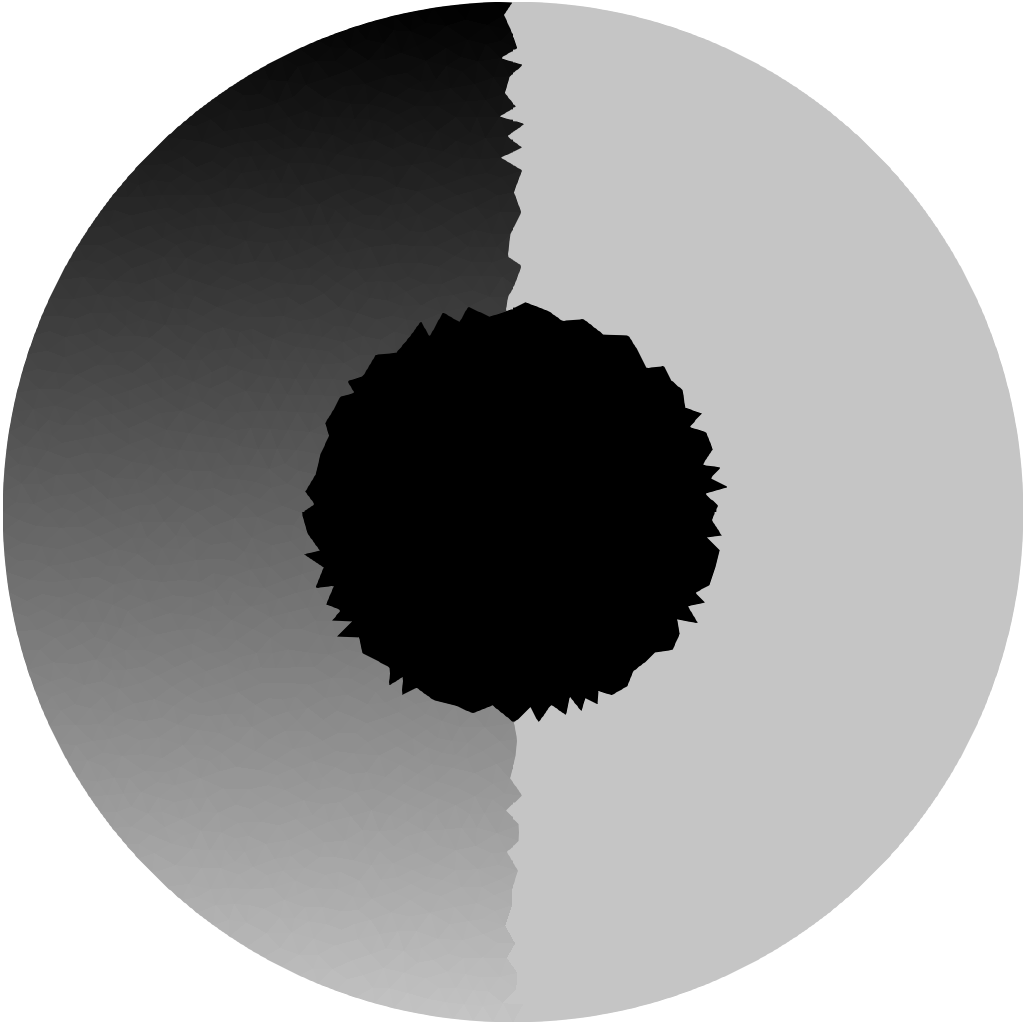}
		\hfill
		\includegraphics[width=0.32\linewidth]{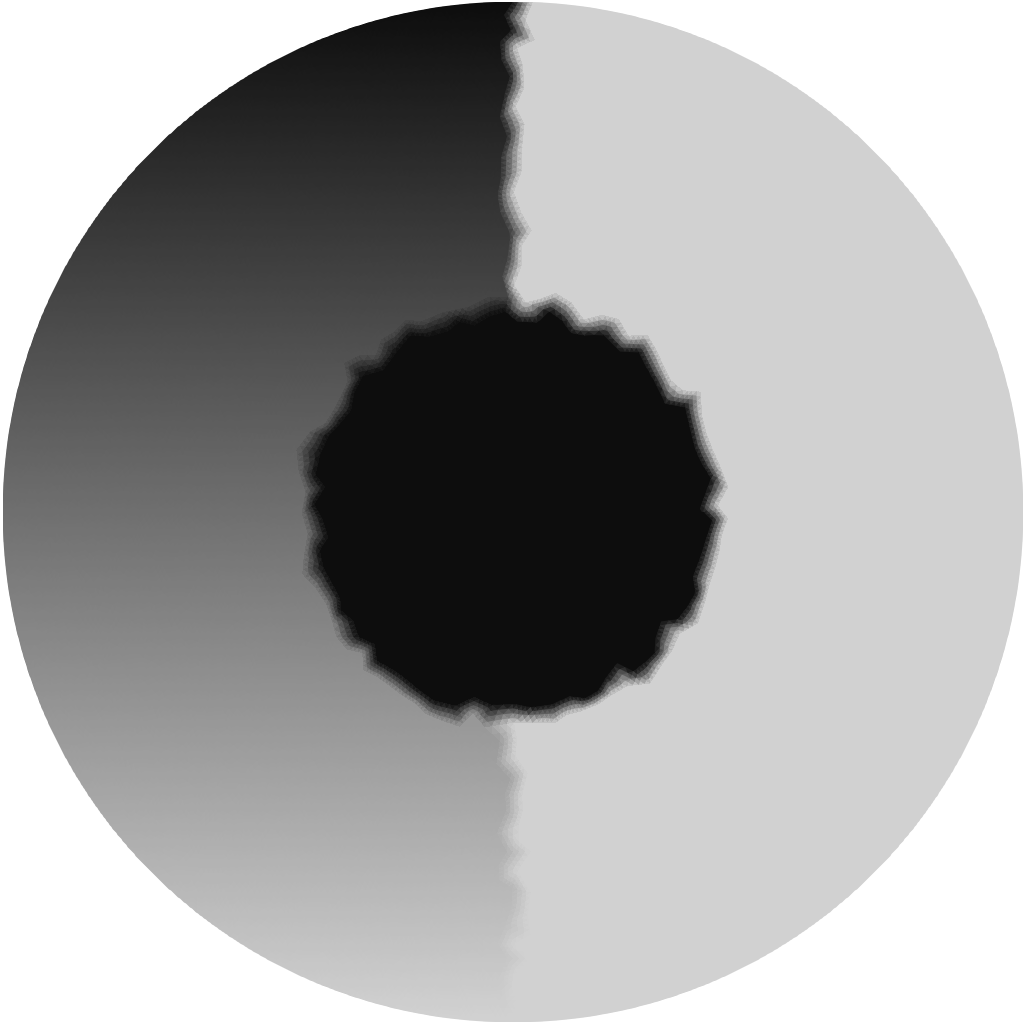}
		\hfill
		\includegraphics[width=0.32\linewidth]{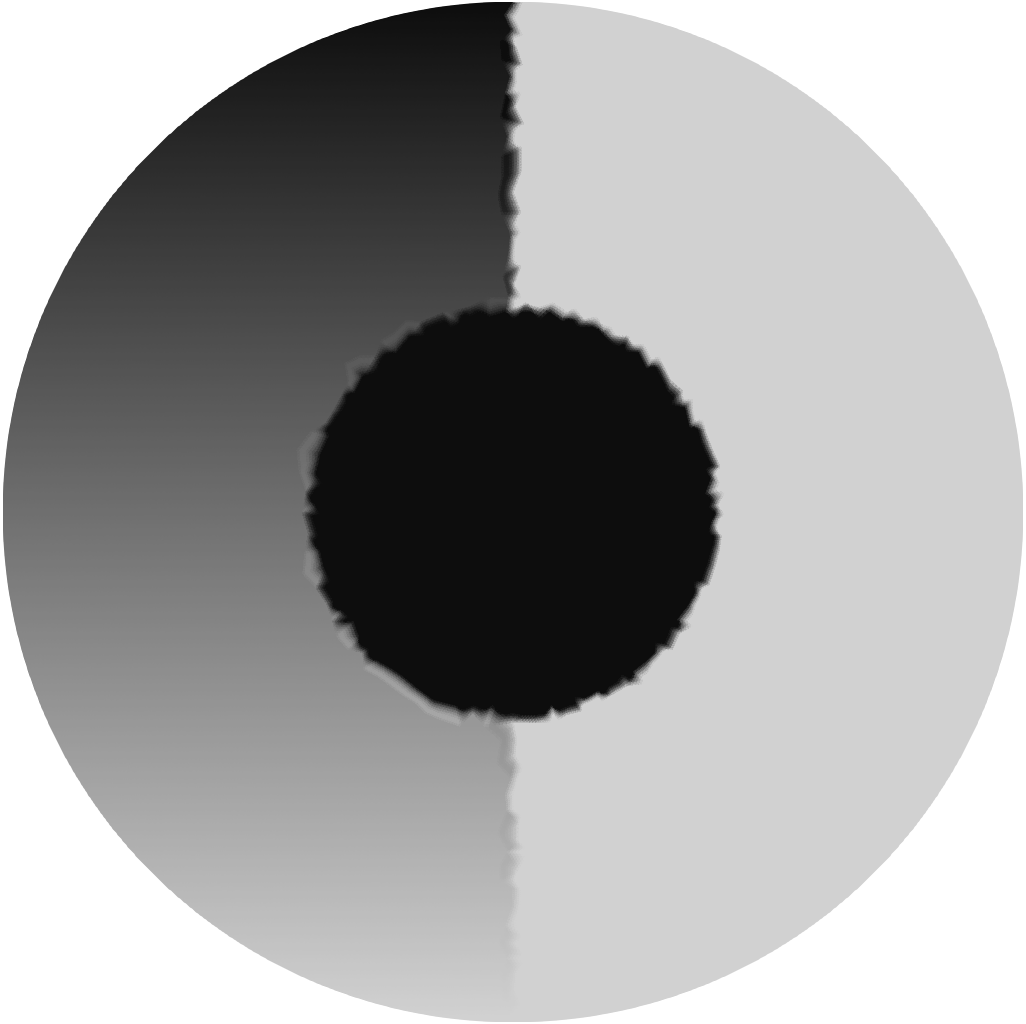}
		\\
		\includegraphics[width=0.32\linewidth]{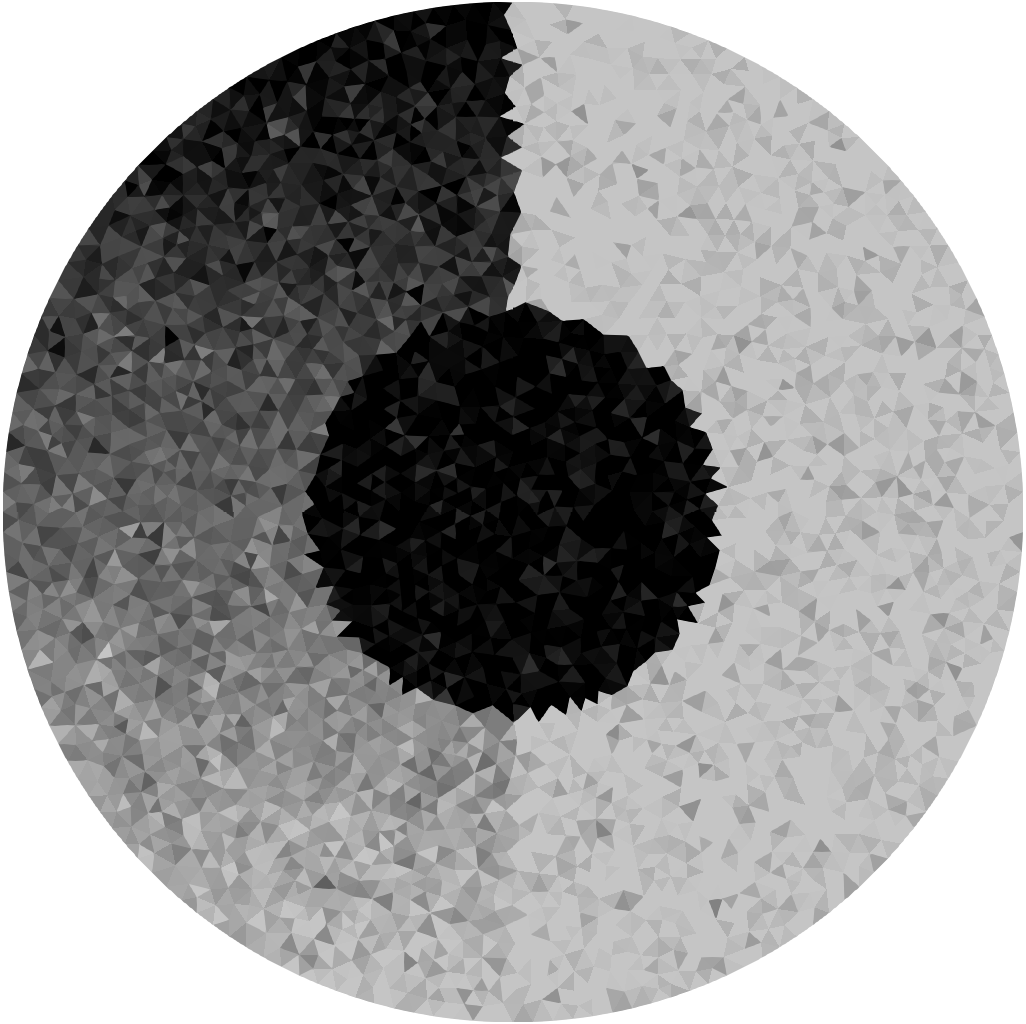}
		\hfill
		\includegraphics[width=0.32\linewidth]{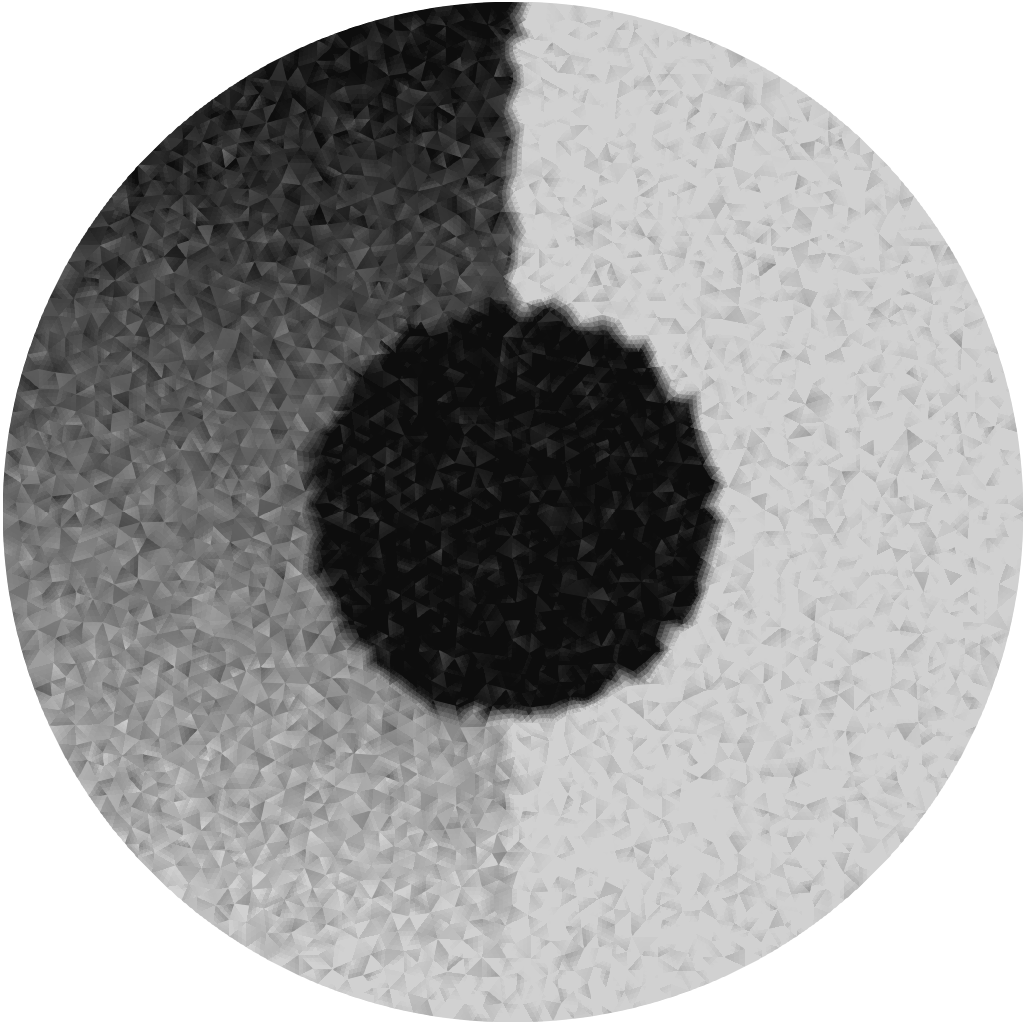}
		\hfill
		\includegraphics[width=0.32\linewidth]{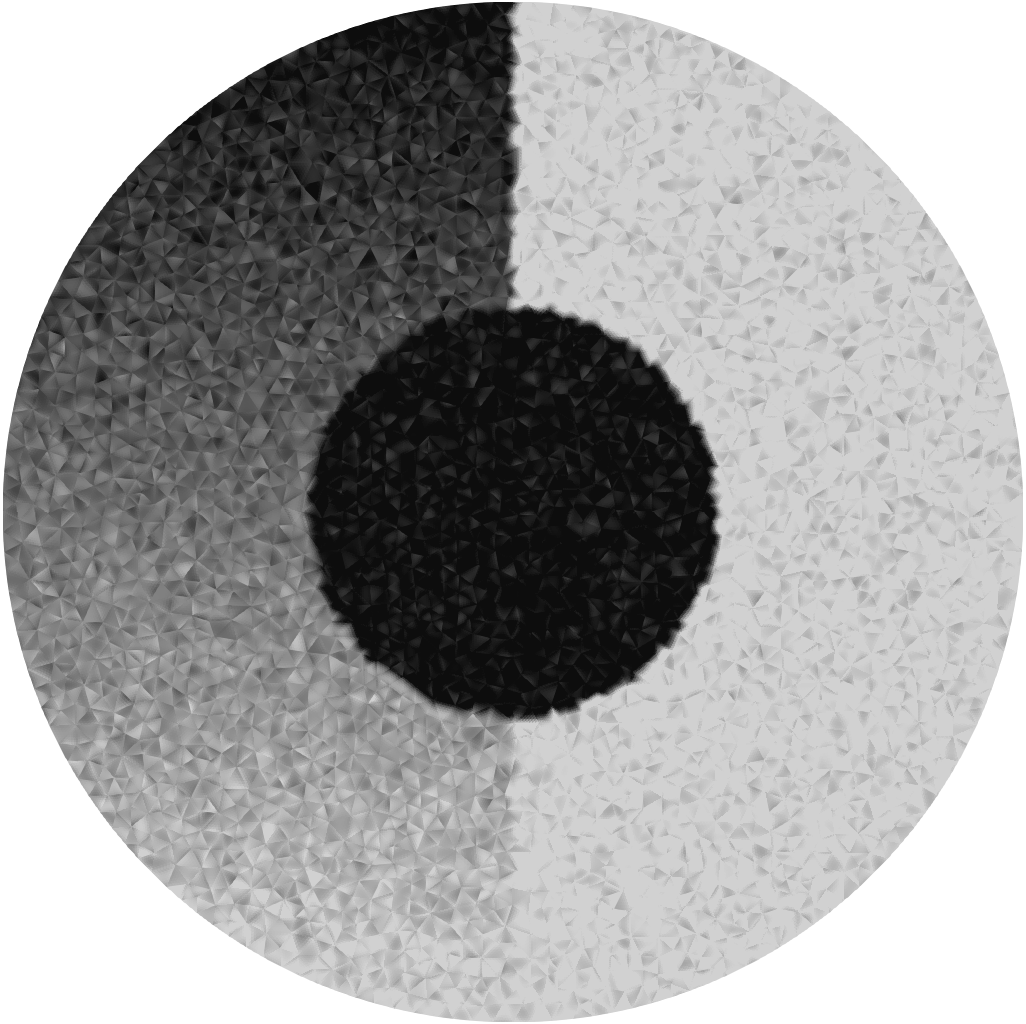}
		\\
		\includegraphics[width=0.32\linewidth]{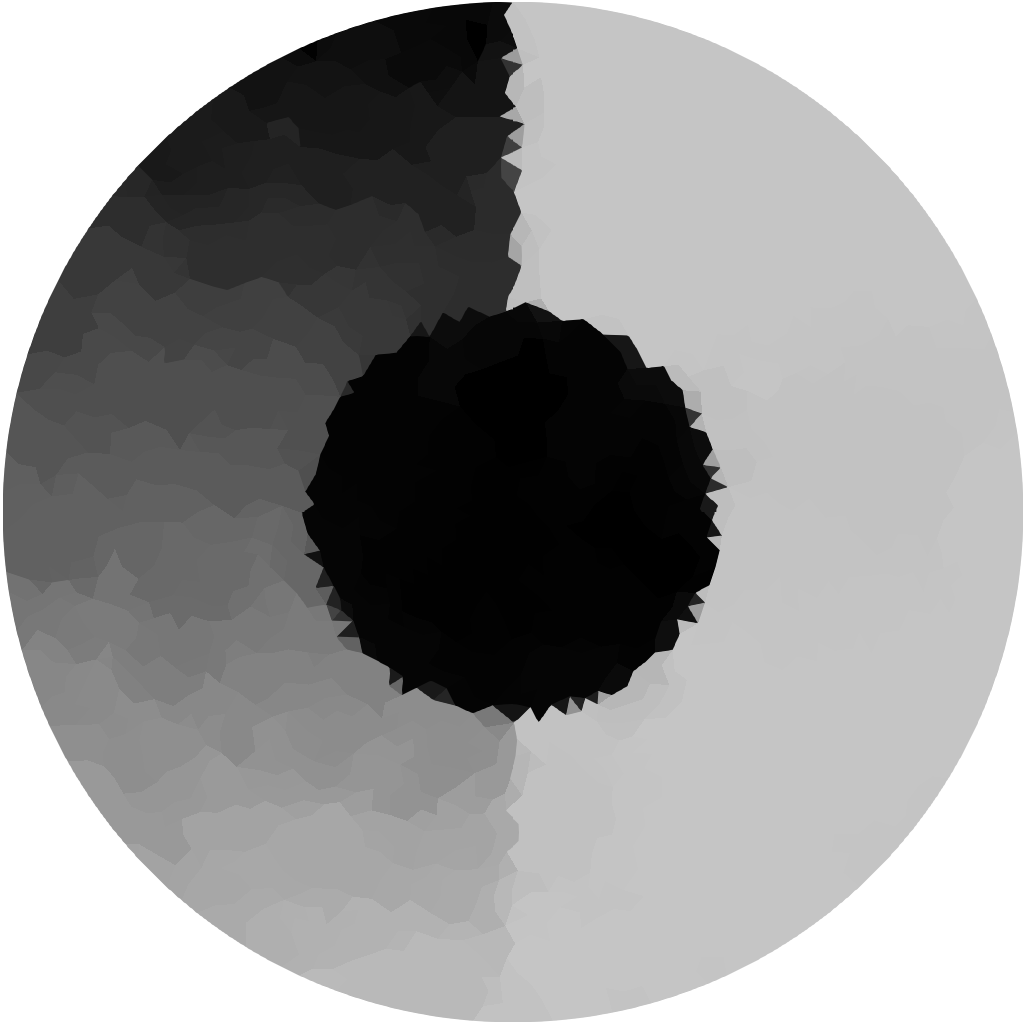}
		\hfill
		\includegraphics[width=0.32\linewidth]{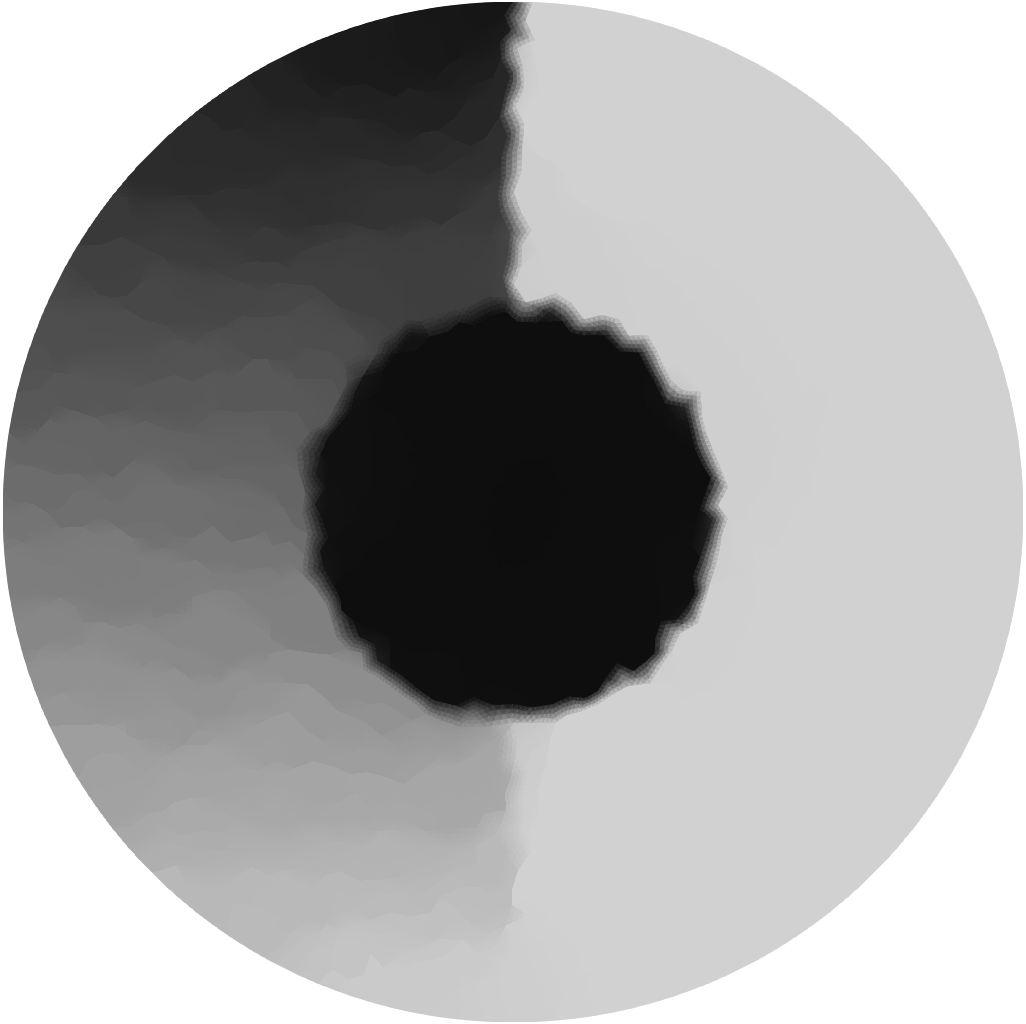}
		\hfill
		\includegraphics[width=0.32\linewidth]{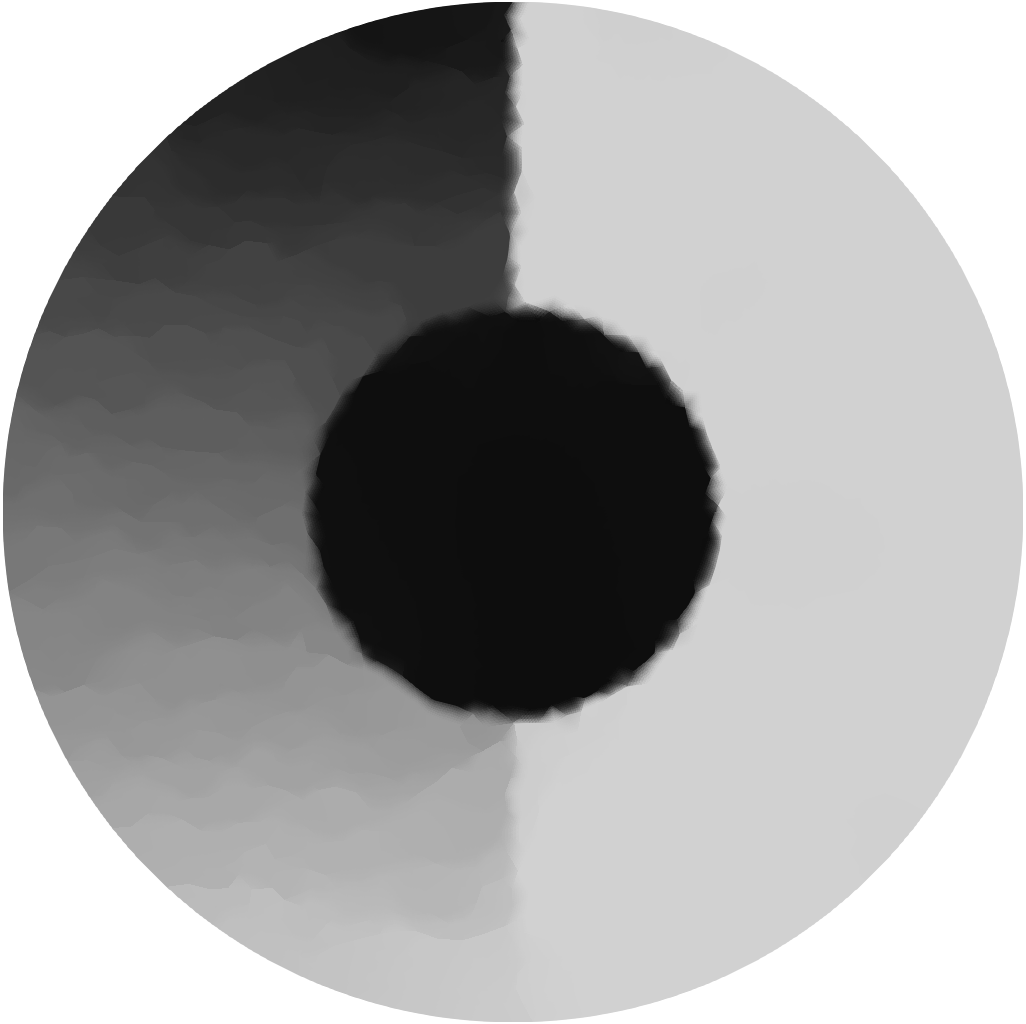}
	\end{center}
	\caption{Original, noisy and denoised images (top to bottom) for $\DG{0}$ (left column), $\DG{1}$ (middle column) and $\DG{2}$ (right column) for \eqref{eq:DTV-L2} with parameter $\beta = \num{1e-3}$ in the isotropic setting ($s = 2$). Results obtained using \cref{alg:split_Bregman_s_arbitrary} (split Bregman), see \cref{tab:data_test_1}. The results obtained by the Chambolle--Pock method are similar and not shown.}
	\label{fig:TestCase1}
\end{figure}

\begin{table*}[htbp]
	\sisetup{scientific-notation=false}
	\centering
	\ifthenelse{\boolean{ispreprint}}{\scriptsize}{}
	\begin{tabular}{@{}llrrll@{}}
		\toprule
		Space      & Algorithm                                                                                    & Iterations & Time [s]   & $\psnr$      & Objective \\
		\midrule
		$\DG{0}$   & Split Bregman ($\lambda = \num{1e-3}$)                                                       & \num{37}   & \num{1.6}  & \num{32.031} & \num{5.51e-3} \\
		& Chambolle-Pock ($\sigma = \num{0.016}, \tau = \num{1e-1}$)                                              & \num{128}  & \num{3.4}  & \num{31.987} & \num{5.51e-3} \\
		\midrule
		$\DG{1}$   & Split Bregman ($\lambda = \num{1e-3}$, $\scaling = \num{1e-2}$)                              & \num{57}   & \num{5.8}  & \num{36.092} & \num{3.46e-3} \\
		& Chambolle-Pock ($\sigma = \num{0.025}, \tau = \num{1e-2}$, $\theta = \num{1}$, $\scaling = \num{1e-2}$) & \num{91}   & \num{6.7}  & \num{33.480} & \num{3.66e-3} \\
		\midrule
		$\DG{2}$   & Split Bregman ($\lambda = \num{1e-3}$, $\scaling = \num{1e-2}$)                              & \num{41}   & \num{9.3}  & \num{31.896} & \num{4.14e-3} \\
		& Chambolle-Pock ($\sigma = \num{0.030}, \tau = \num{1e-3}$, $\theta = \num{1}$, $\scaling = \num{1e-2}$) & \num{223}  & \num{35.1} & \num{31.066} & \num{4.32e-3} \\
		\midrule 
		\bottomrule \\
	\end{tabular}
	\caption{Comparison of the performance of \cref{alg:split_Bregman_s_arbitrary,alg:Chambolle-Pock_s_arbitrary} for the denoising problem shown in~\cref{fig:TestCase1} in various discretizations.
	}
	\label{tab:data_test_1}
\end{table*}

\added{\Cref{fig:TestCase1} visualizes the benefits of higher-order finite elements in particular in the case where the discontinuities in the image are not resolved by the computational mesh.
	In addition, the $\DG{1}$ and $\DG{2}$ solutions exhibit less staircasing.
	Further evidence for the benefits of higher-order polynomial spaces for the cameraman test image is given in \cref{subsec:denoising_on_low_resolution_meshes}.
}

\added{Before continuing, we mention that all results in $\DG{1}$ were interpolated onto $\DG{0}$ on a twice refined mesh merely for visualization since $\DG{1}$ functions cannot directly be displayed in \paraview. 
Likewise, results in $\DG{2}$ were interpolated onto $\DG{0}$ on a three times refined mesh for visualization.}

\subsection{Comparison to $\DG{0}$ Image Denoising on Pixel Grids}
\label{subsec:Comparison_to_pixel_grids}

\added{In this section we provide a comparison of our approach, using $\DG{r}$ representations of an image for $r \in \{0,1,2\}$ and the discrete problem \eqref{eq:DTV-L2}, with the classical representation by constant pixels.
	We refer to the latter as $\DG{0}$ on pixels.
	In this example, we use the discrete cameraman test image on a $256 \times 256$ pixel grid.
For the finite element spaces, each pixel is refined into four triangles with crossed diagonals.}

\added{For this problem we do not expect higher-order discretization to be particularly beneficial since the 'original' image data is only piecewise constant itself.
	In addition, we cannot directly compare run-times since the $\DG{0}$ pixel problem was solved with an implementation of the split Bregman method in \matlab, since \fenics\ does not support all function spaces on quadraliteral meshes.
In any case, the same starting guess and stopping criterion \eqref{stopping_criterion} was used in each case.}

\added{The denoising results are shown in \cref{fig:TestCase2} and the convergence behavior of the split Bregman method is displayed in \cref{tab:data_test_2}.}

\begin{figure}[htbp]
	\begin{center}
		\hfill
		\ifthenelse{\boolean{ispreprint}}{\includegraphics[width=0.30\linewidth]{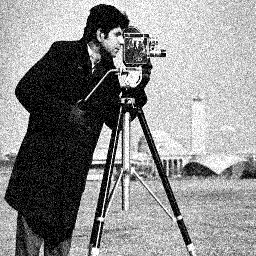}}{\includegraphics[width=0.48\linewidth]{./pix/Test2/T2_SB_DG0pixel_1.png}}
		\hfill
		\ifthenelse{\boolean{ispreprint}}{\includegraphics[width=0.30\linewidth]{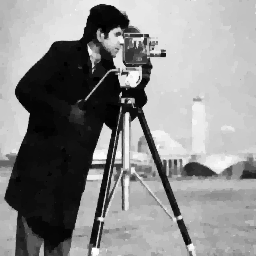}}{\includegraphics[width=0.48\linewidth]{./pix/Test2/T2_SB_DG0pixel_2.png}}
		\hspace*{\fill}
		\\
		\hfill
		\ifthenelse{\boolean{ispreprint}}{\includegraphics[width=0.30\linewidth]{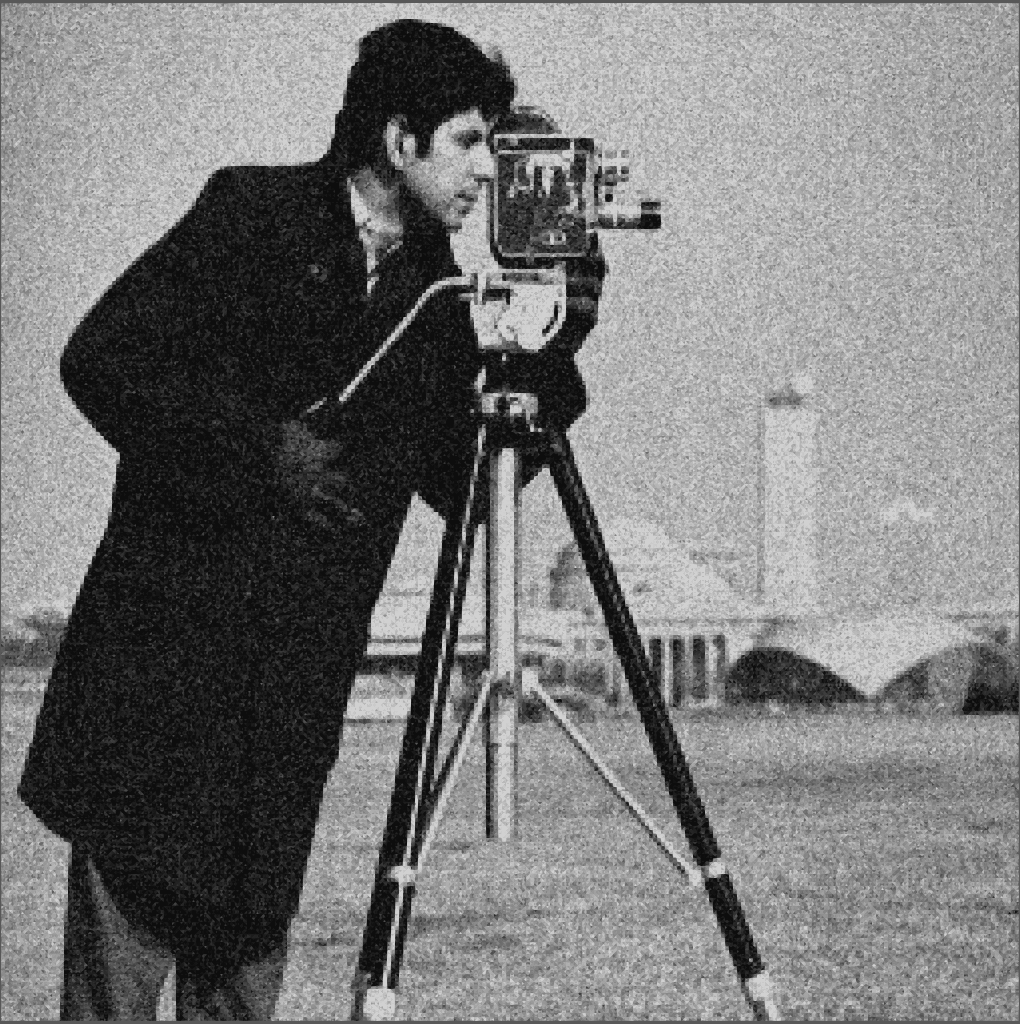}}{\includegraphics[width=0.48\linewidth]{./pix/Test2/T2_SB_DG0_1.png}}
		\hfill
		\ifthenelse{\boolean{ispreprint}}{\includegraphics[width=0.30\linewidth]{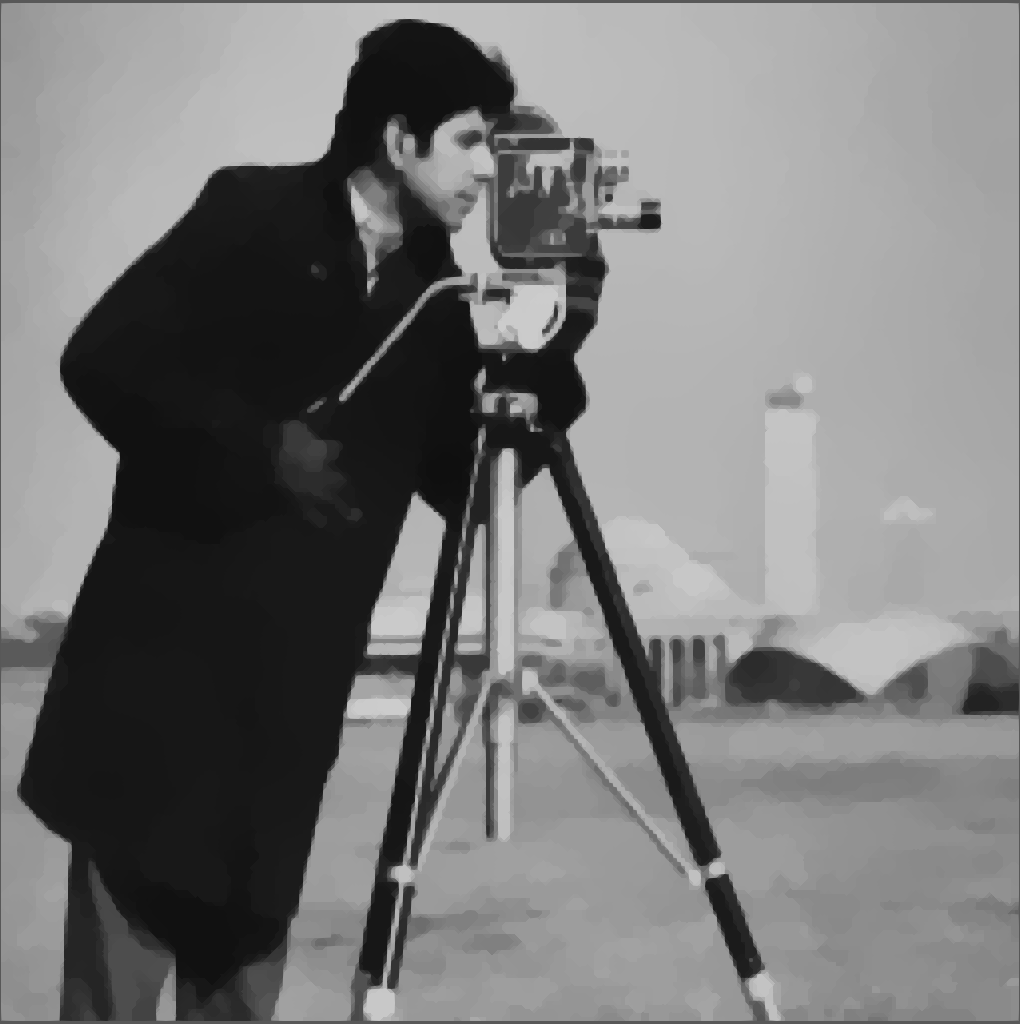}}{\includegraphics[width=0.48\linewidth]{./pix/Test2/T2_SB_DG0_2.png}}
		\hspace*{\fill}
		\\
		\hfill
		\ifthenelse{\boolean{ispreprint}}{\includegraphics[width=0.30\linewidth]{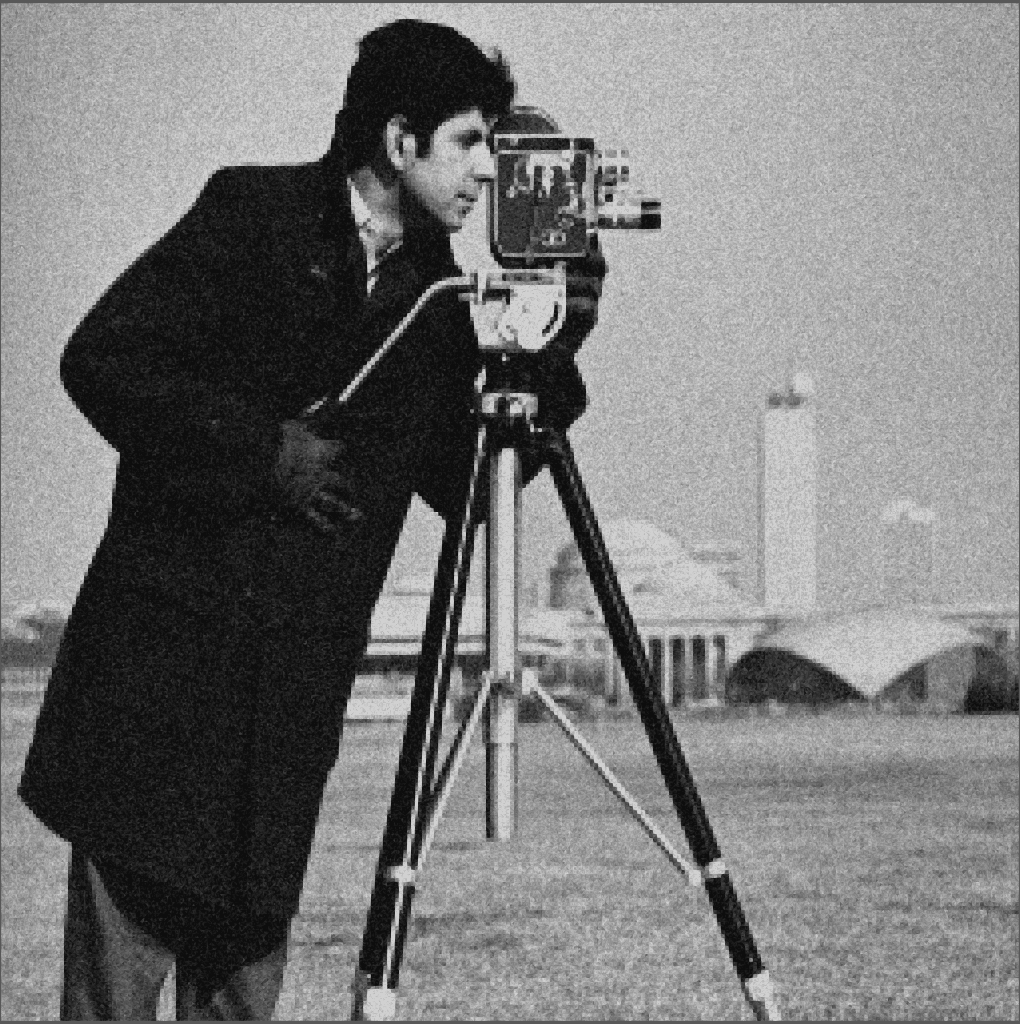}}{\includegraphics[width=0.48\linewidth]{./pix/Test2/T2_SB_DG1_1.png}}
		\hfill
		\ifthenelse{\boolean{ispreprint}}{\includegraphics[width=0.30\linewidth]{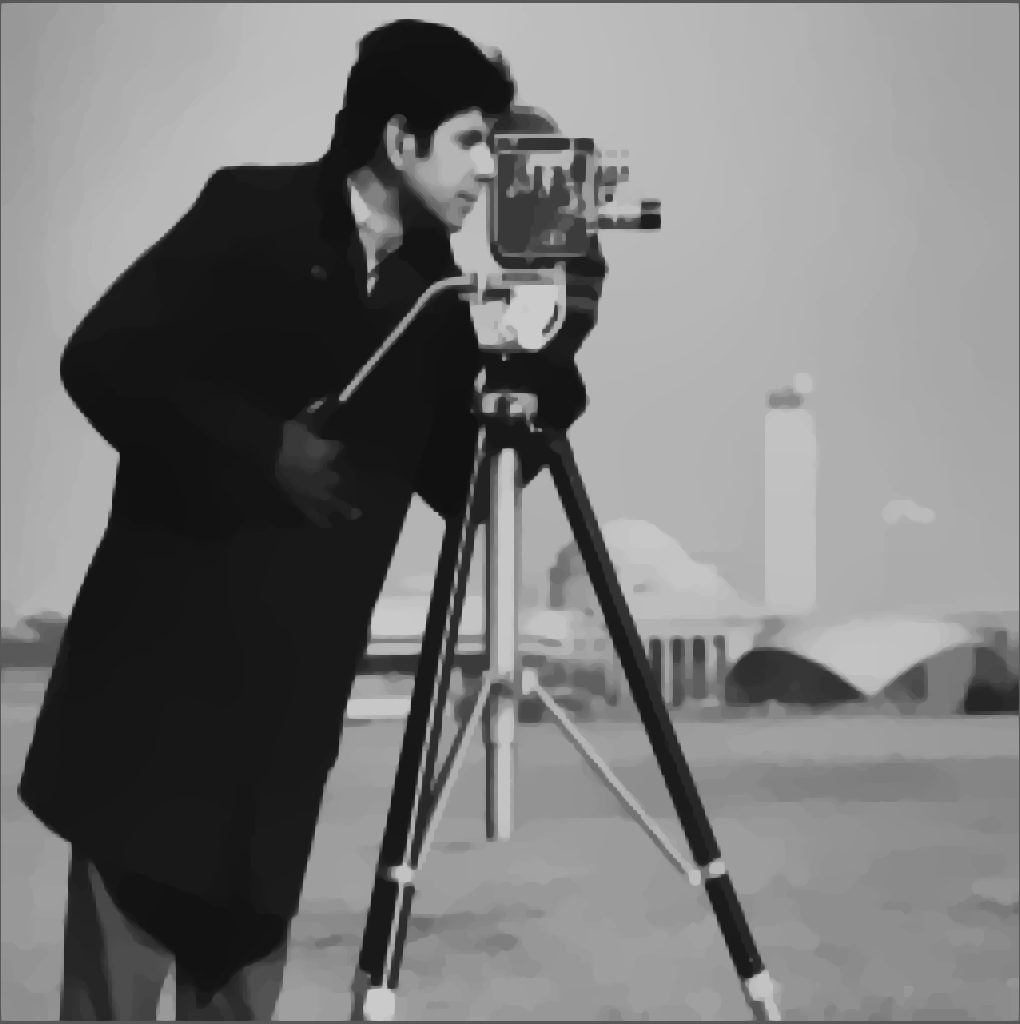}}{\includegraphics[width=0.48\linewidth]{./pix/Test2/T2_SB_DG1_2.png}}
		\hspace*{\fill}
		\\
		\hfill
		\ifthenelse{\boolean{ispreprint}}{\includegraphics[width=0.30\linewidth]{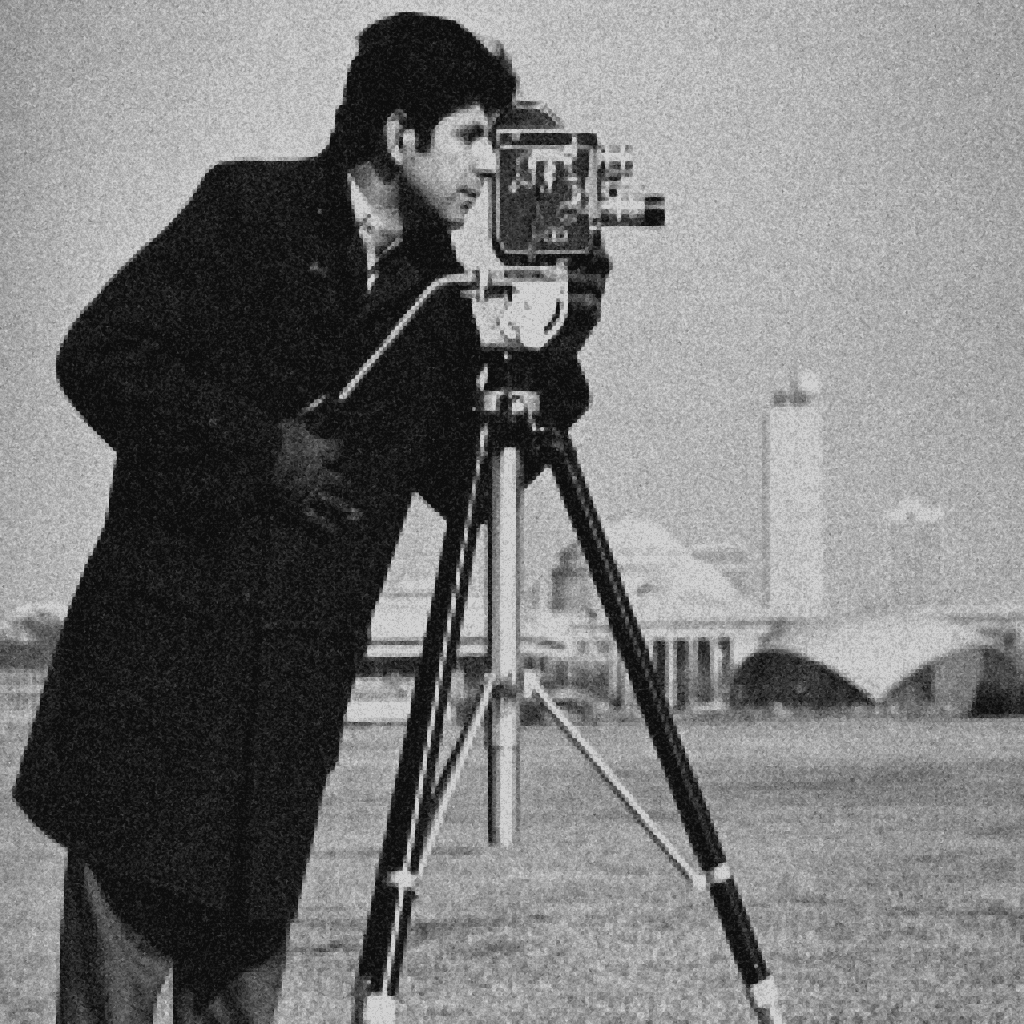}}{\includegraphics[width=0.48\linewidth]{./pix/Test2/T2_SB_DG2_1.png}}
		\hfill
		\ifthenelse{\boolean{ispreprint}}{\includegraphics[width=0.30\linewidth]{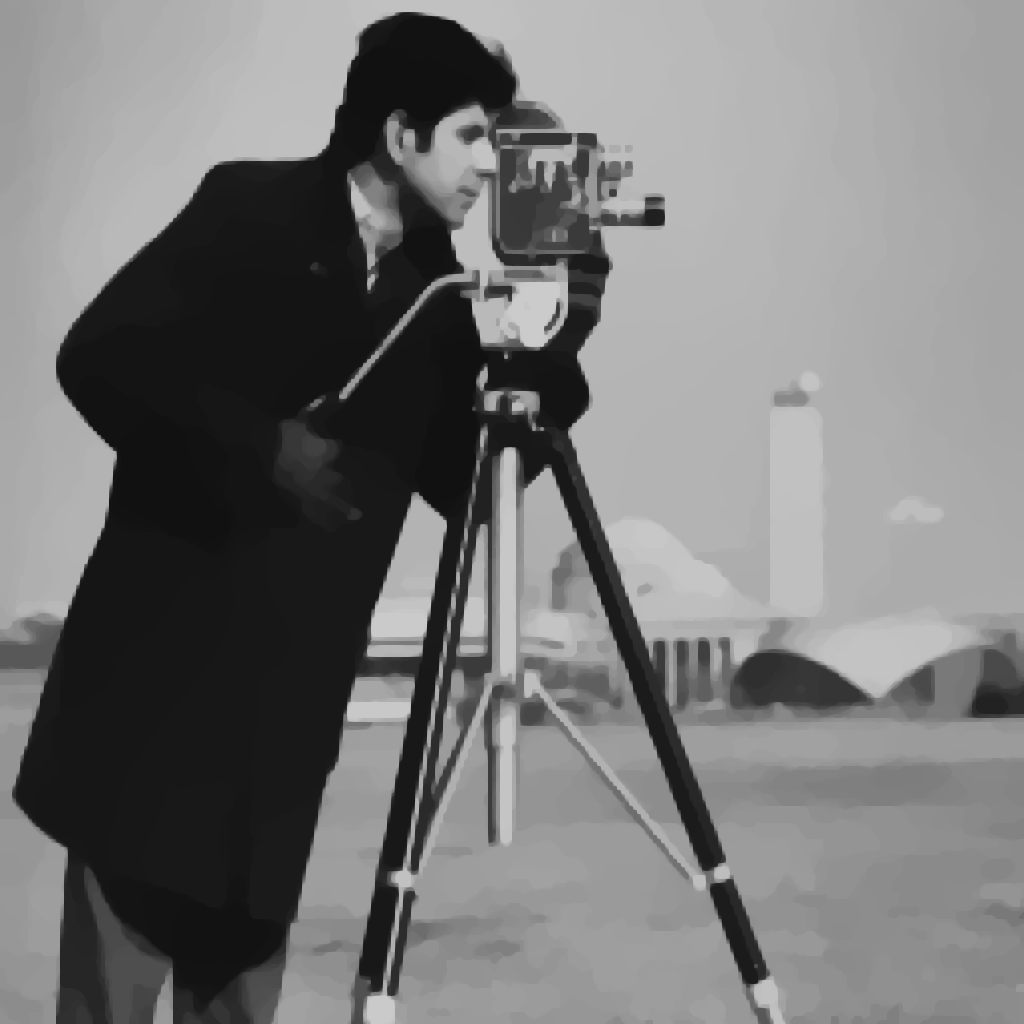}}{\includegraphics[width=0.48\linewidth]{./pix/Test2/T2_SB_DG2_2.png}}
		\hspace*{\fill}
		\\
	\end{center}
	\caption{Noisy (left) and denoised (right) images for classical $\DG{0}$ on pixels (top row), and finite element solutions in $\DG{0}$ (second row), $\DG{1}$ (third row) and $\DG{2}$ (bottom row) for \eqref{eq:DTV-L2} with parameter $\beta = \num{3e-4}$ in the isotropic setting ($s = 2$). Results obtained using \cref{alg:split_Bregman_s_arbitrary} (split Bregman), see \cref{tab:data_test_2}.}
	\label{fig:TestCase2}
\end{figure}

\begin{table*}[htbp]
	\sisetup{scientific-notation=false}
	\centering
	\ifthenelse{\boolean{ispreprint}}{\scriptsize}{}
	\begin{tabular}{@{}llrrll@{}}
		\toprule
		Space              & Algorithm                                                       & Iterations & Time [s]       & $\psnr$      & Objective \\
		\midrule
		$\DG{0}$ on pixels & Split Bregman ($\lambda = \num{1e-2}$)                          & \num{24}   & \num{2.8}      & \num{26.236} & \num{8.58e-3} \\
		\midrule
		$\DG{0}$           & Split Bregman ($\lambda = \num{1e-2}$)                          & \num{32}   & \num{49.1}     & \num{26.641} & \num{8.75e-3} \\
		\midrule
		$\DG{1}$           & Split Bregman ($\lambda = \num{1e-2}$, $\scaling = \num{1e-2}$) & \num{63}   & \num{516.6}    & \num{26.882} & \num{6.30e-3} \\
		\midrule
		$\DG{2}$           & Split Bregman ($\lambda = \num{1e-2}$, $\scaling = \num{1e-2}$) & \num{138}  & \num{3610.1}   & \num{26.911} & \num{6.94e-3} \\
		\midrule 
		\bottomrule \\
	\end{tabular}
	\caption{Comparison of the performance of \cref{alg:split_Bregman_s_arbitrary} (split Bregman) for the denoising problem shown in~\cref{fig:TestCase2} in various discretizations.}
	\label{tab:data_test_2}
\end{table*}

\subsection{Denoising of Low-Resolution Images}
\label{subsec:denoising_on_low_resolution_meshes}

\added{In this section we consider a low resolution of the cameraman image, which was obtained by interpolating the $256 \times 256$ pixel image onto a $64 \times 64$~square pixel grid with crossed diagonals.
	\added{Again, noise is added per coefficient in the respective space.}
	Subsequently the denoising problem is solved in the $\DG{r}(\Omega)$ spaces for $r \in \{0,1,2\}$ on the coarse grid.
	The goal is to demonstrate that the use of higher-order polynomial functions can partially compensate the loss of geometric resolution.
	In \cref{fig:TestCase3} we show the results obtained using the split Bregman method, whose performance was similar as in \cref{subsec:denoising_of_DGr_in_DGr}, as can be seen in \cref{tab:data_test_3}. 
The $\psnr$ values were evaluated using the full resolution image as $\uref$.}

\added{As can be seen from the results in \cref{fig:TestCase3} and \cref{tab:data_test_3}, the recovered image in $\DG{2}(\Omega)$, see \cref{fig:TestCase3}~(bottom right), exceeds the $\DG{0}$ image both in visual quality and PSNR value.}

\begin{figure}[htbp]
	\begin{center}
		\hfill
		\ifthenelse{\boolean{ispreprint}}{\includegraphics[width=0.35\linewidth]{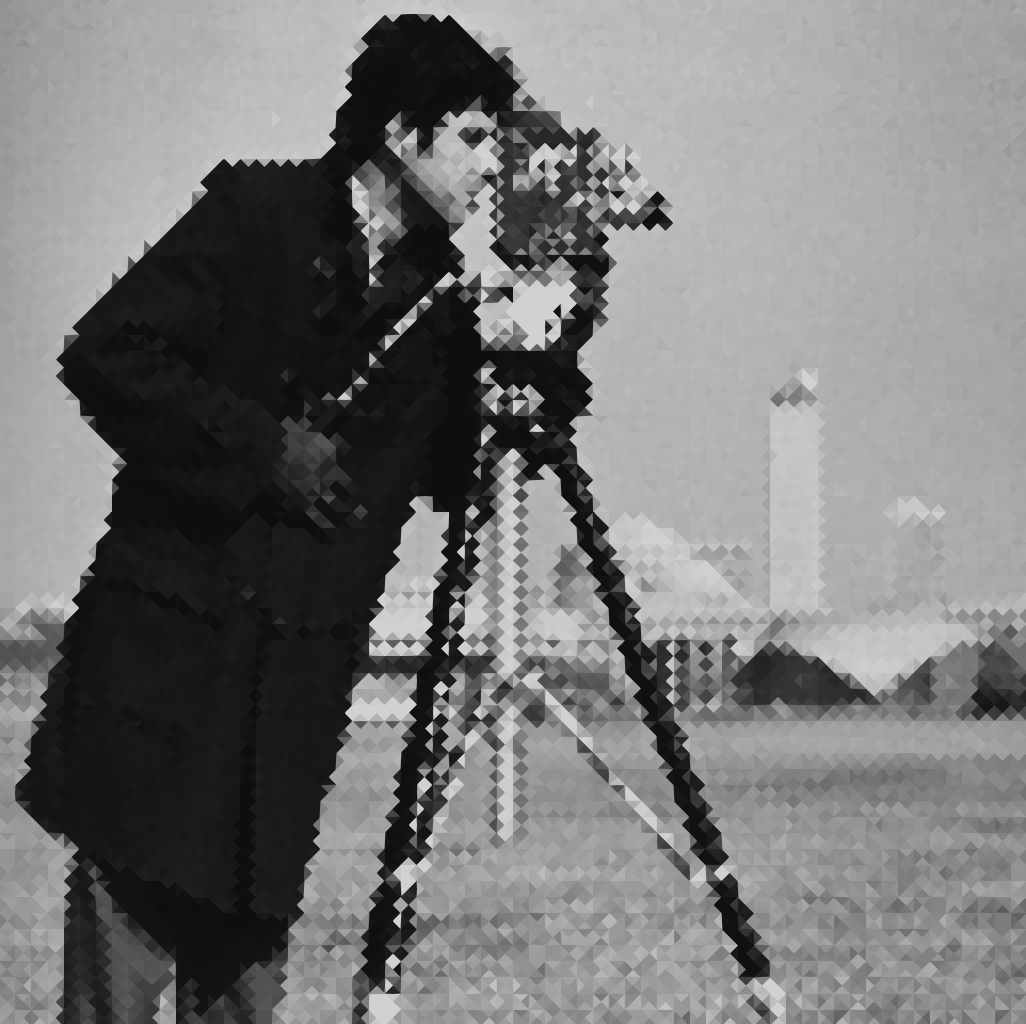}}{\includegraphics[width=0.48\linewidth]{./pix/Test3/T3_SB_DG0_1.png}}
		\hfill
		\ifthenelse{\boolean{ispreprint}}{\includegraphics[width=0.35\linewidth]{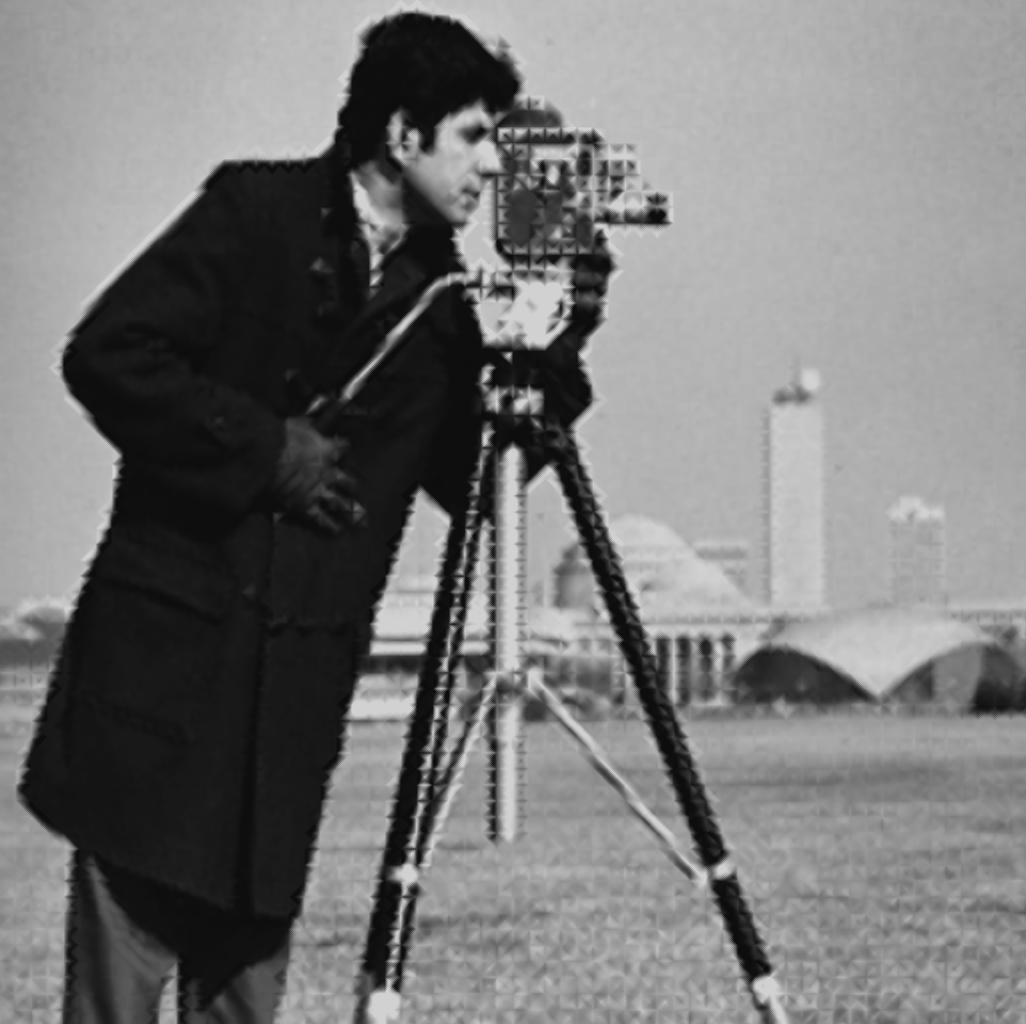}}{\includegraphics[width=0.48\linewidth]{./pix/Test3/T3_SB_DG2_1.png}}
		\hspace*{\fill}
		\\
		\hfill
		\ifthenelse{\boolean{ispreprint}}{\includegraphics[width=0.35\linewidth]{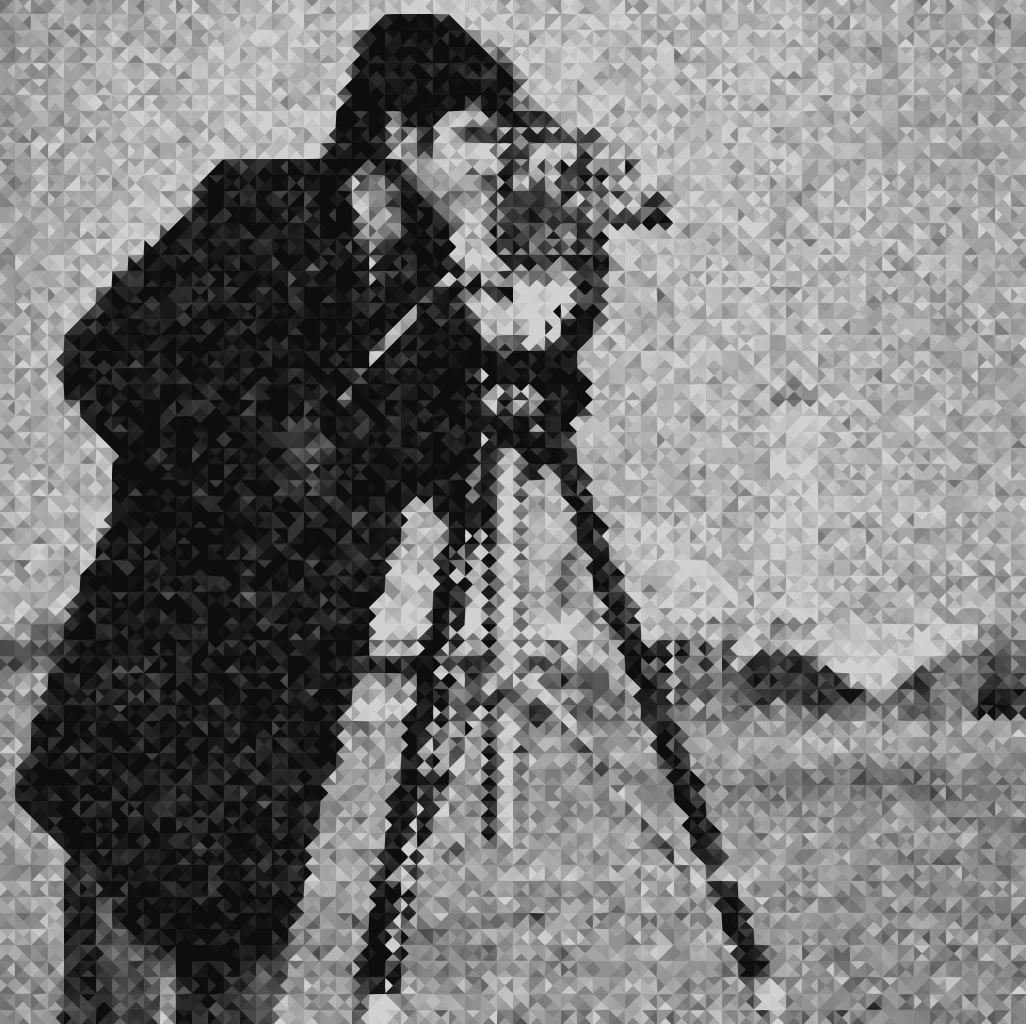}}{\includegraphics[width=0.48\linewidth]{./pix/Test3/T3_SB_DG0_2.png}}
		\hfill
		\ifthenelse{\boolean{ispreprint}}{\includegraphics[width=0.35\linewidth]{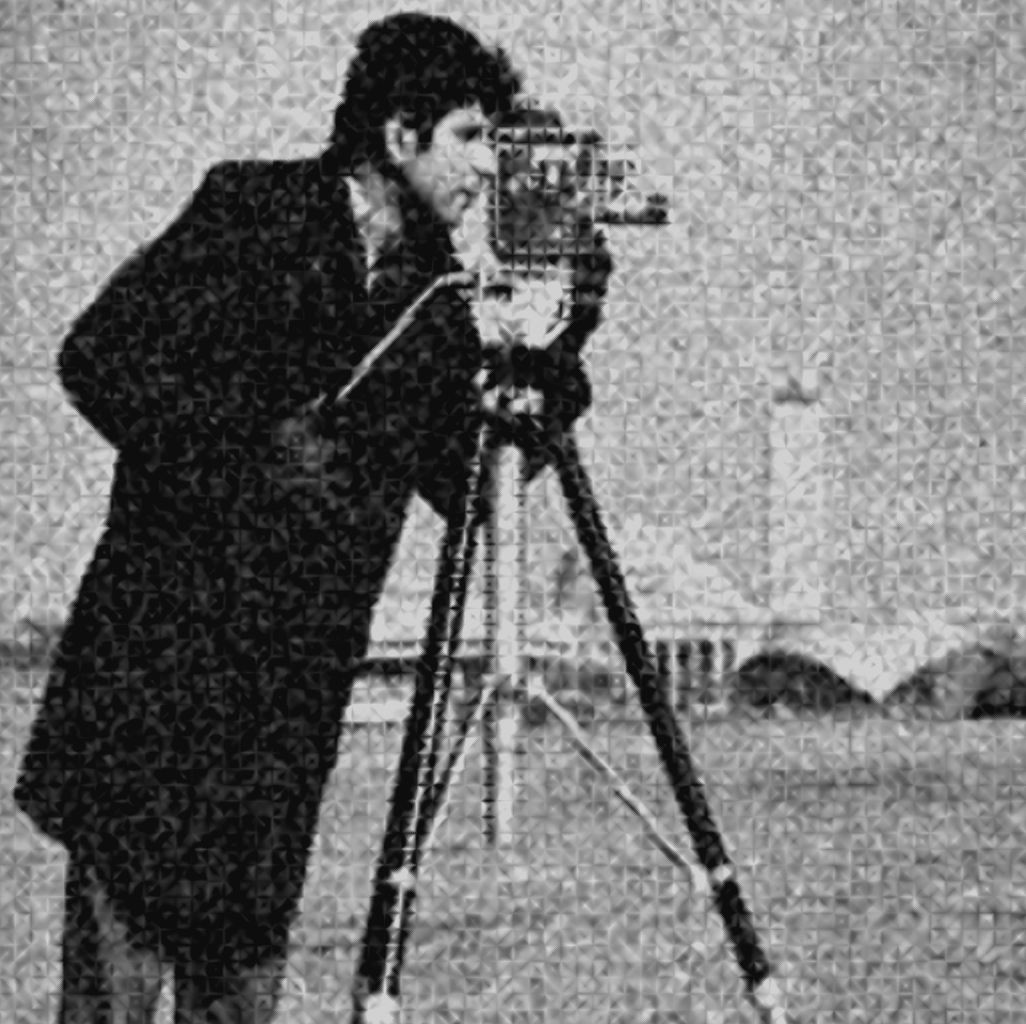}}{\includegraphics[width=0.48\linewidth]{./pix/Test3/T3_SB_DG2_2.png}}
		\hspace*{\fill}
		\\
		\hfill
		\ifthenelse{\boolean{ispreprint}}{\includegraphics[width=0.35\linewidth]{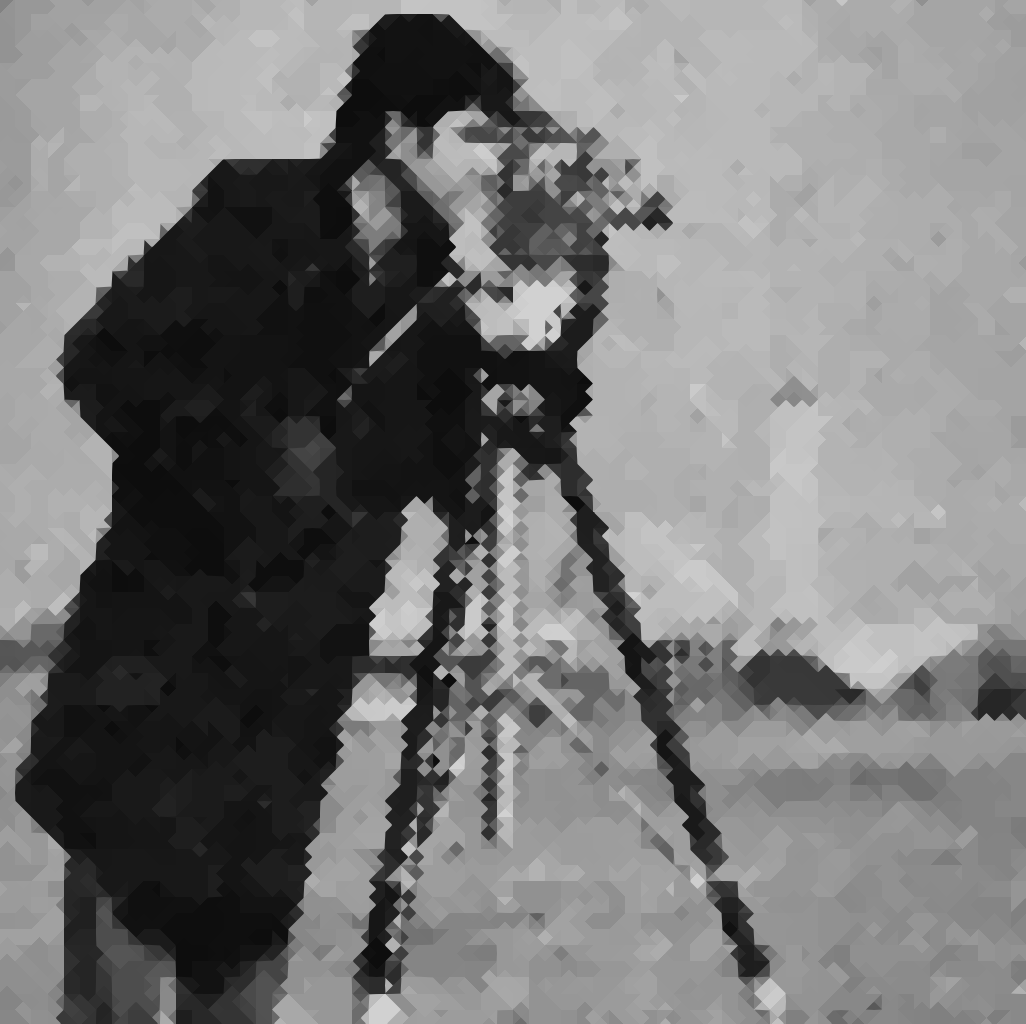}}{\includegraphics[width=0.48\linewidth]{./pix/Test3/T3_SB_DG0_3.png}}
		\hfill
		\ifthenelse{\boolean{ispreprint}}{\includegraphics[width=0.35\linewidth]{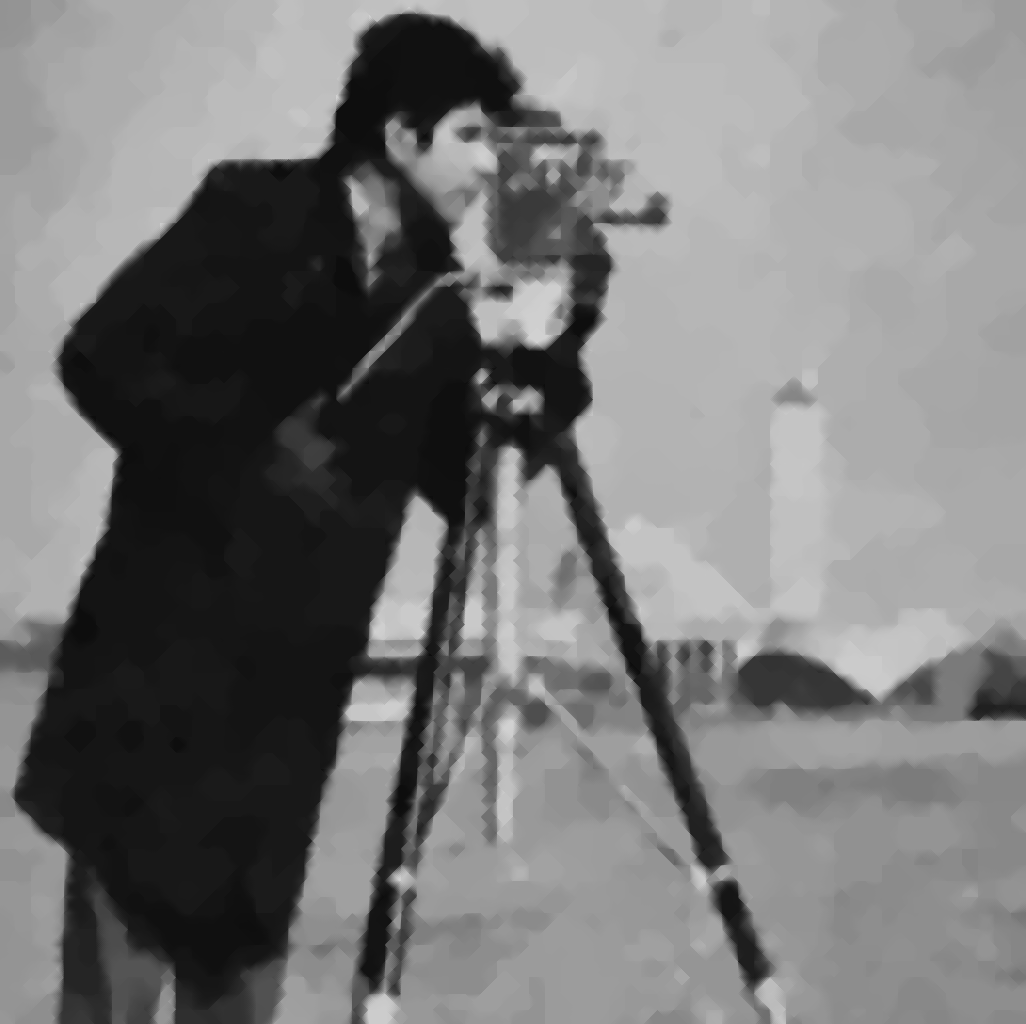}}{\includegraphics[width=0.48\linewidth]{./pix/Test3/T3_SB_DG2_3.png}}
		\hspace*{\fill}
	\end{center}
	\caption{Original (interpolated), noisy and denoised images (top to bottom) for $\DG{0}$ (left column) and $\DG{2}$ (right column) for \eqref{eq:DTV-L2} with parameter $\beta = \num{0.0004}$ for the isotropic setting ($s = 2$) on a coarse grid. Results obtained using \cref{alg:split_Bregman_s_arbitrary} (split Bregman), see \cref{tab:data_test_3}.}
	\label{fig:TestCase3}
\end{figure}

\begin{table*}[htbp]
	\sisetup{scientific-notation=false}
	\centering
	\begin{tabular}{@{}llrrll@{}}
		\toprule
		Space      & Algorithm                                                       & Iterations & Time [s]   & $\psnr$      & Objective \\
		\midrule
		$\DG{0}$   & Split Bregman ($\lambda = \num{1e-2}$)                          & \num{20}   & \num{6.3}  & \num{19.333} & \num{8.97e-3} \\
		\midrule
		$\DG{2}$   & Split Bregman ($\lambda = \num{1e-2}$, $\scaling = \num{1e-2}$) & \num{101}  & \num{84.3} & \num{20.855} & \num{7.18e-3} \\
		\midrule 
		\bottomrule \\
	\end{tabular}
	\caption{Performance of \cref{alg:split_Bregman_s_arbitrary} (split Bregman) for the low-resolution denoising problem shown in~\cref{fig:TestCase3} in various discretizations.}
	\label{tab:data_test_3}
\end{table*}

\subsection{Inpainting of $\DG{r}$-Images}
\label{subsec:Inpainting}

\added{In this and the following section we demonstrate the utility of higher-order polynomial function spaces for the purpose of denoising and inpainting.
	To this end, we consider the non-discrete 'ball' image and randomly delete two thirds of all cells, which subsequently serve as the inpainting region $\Omega \setminus \Omega_0$.
	Noise is added to the remaining data and problem \eqref{eq:DTV-L2} solved in $\DG{r}(\Omega)$ for $r \in \{0,1,2\}$; see \cref{fig:TestCase4}.
	For this test, we found the Chambolle--Pock method (\cref{alg:Chambolle-Pock_s_arbitrary}) to perform better than split Bregman; see \cref{tab:data_test_4}.
}

\begin{figure}[htbp]
	\begin{center}
		\hfill
		\ifthenelse{\boolean{ispreprint}}{\includegraphics[width=0.35\linewidth]{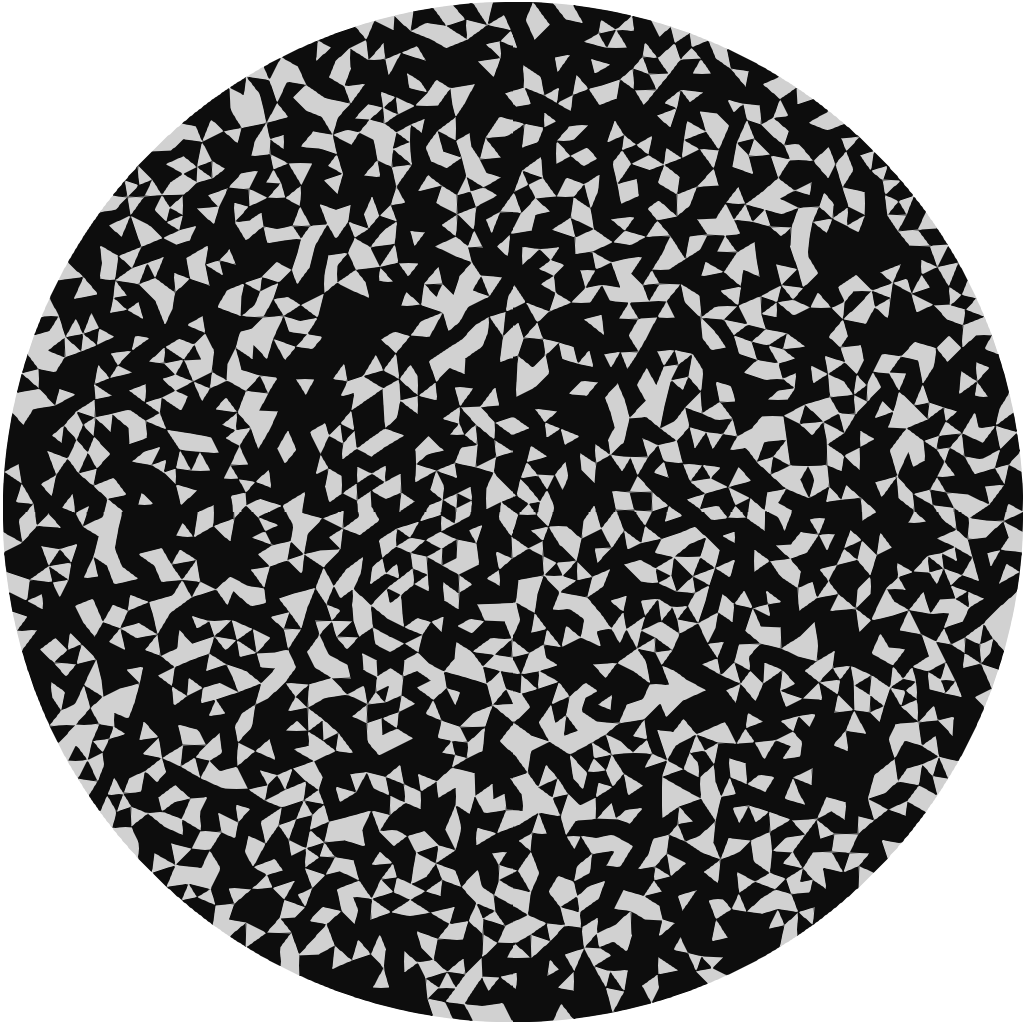}}{\includegraphics[width=0.48\linewidth]{./pix/Test4/Inpainting_region.png}}
		\hfill
		\ifthenelse{\boolean{ispreprint}}{\includegraphics[width=0.35\linewidth]{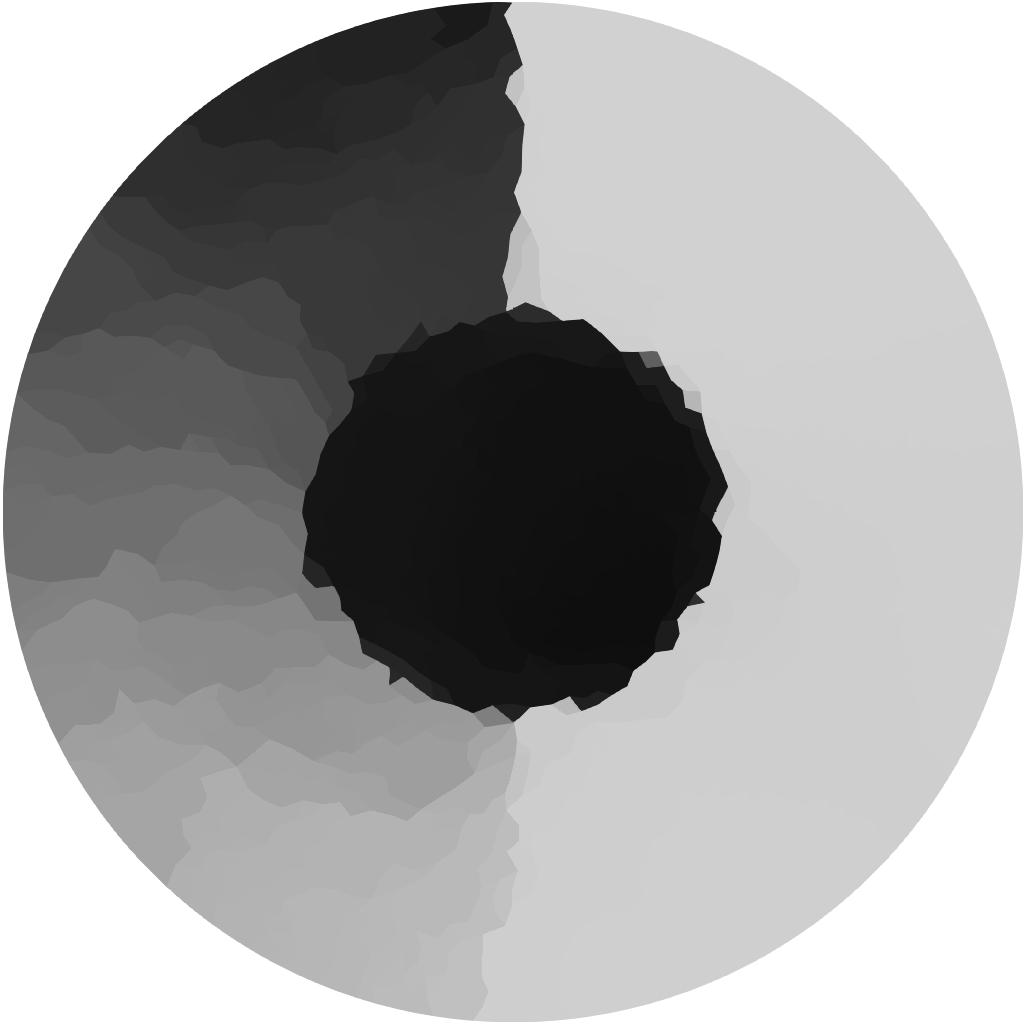}}{\includegraphics[width=0.48\linewidth]{./pix/Test4/T4_CP_DG0_1.png}}
		\hspace*{\fill}
		\\
		\hfill
		\ifthenelse{\boolean{ispreprint}}{\includegraphics[width=0.35\linewidth]{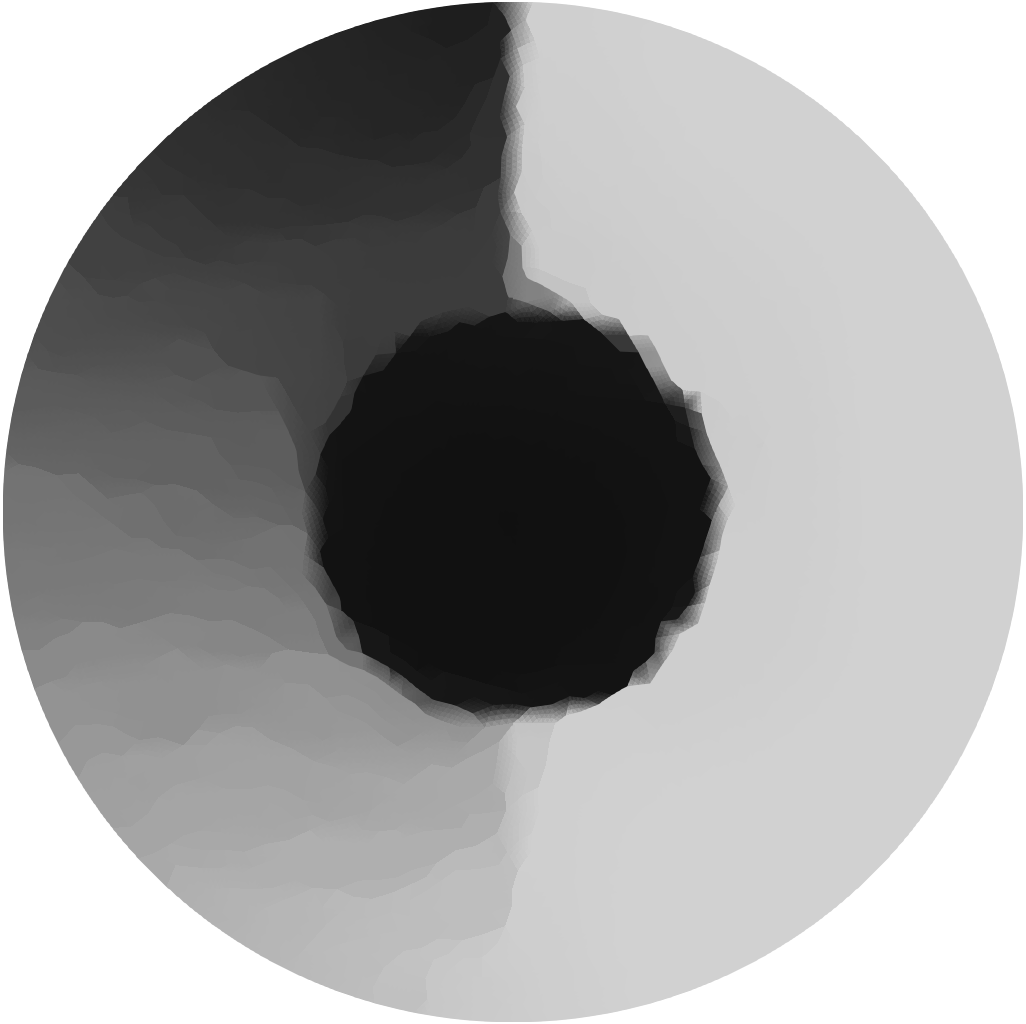}}{\includegraphics[width=0.48\linewidth]{./pix/Test4/T4_CP_DG1_1.png}}
		\hfill
		\ifthenelse{\boolean{ispreprint}}{\includegraphics[width=0.35\linewidth]{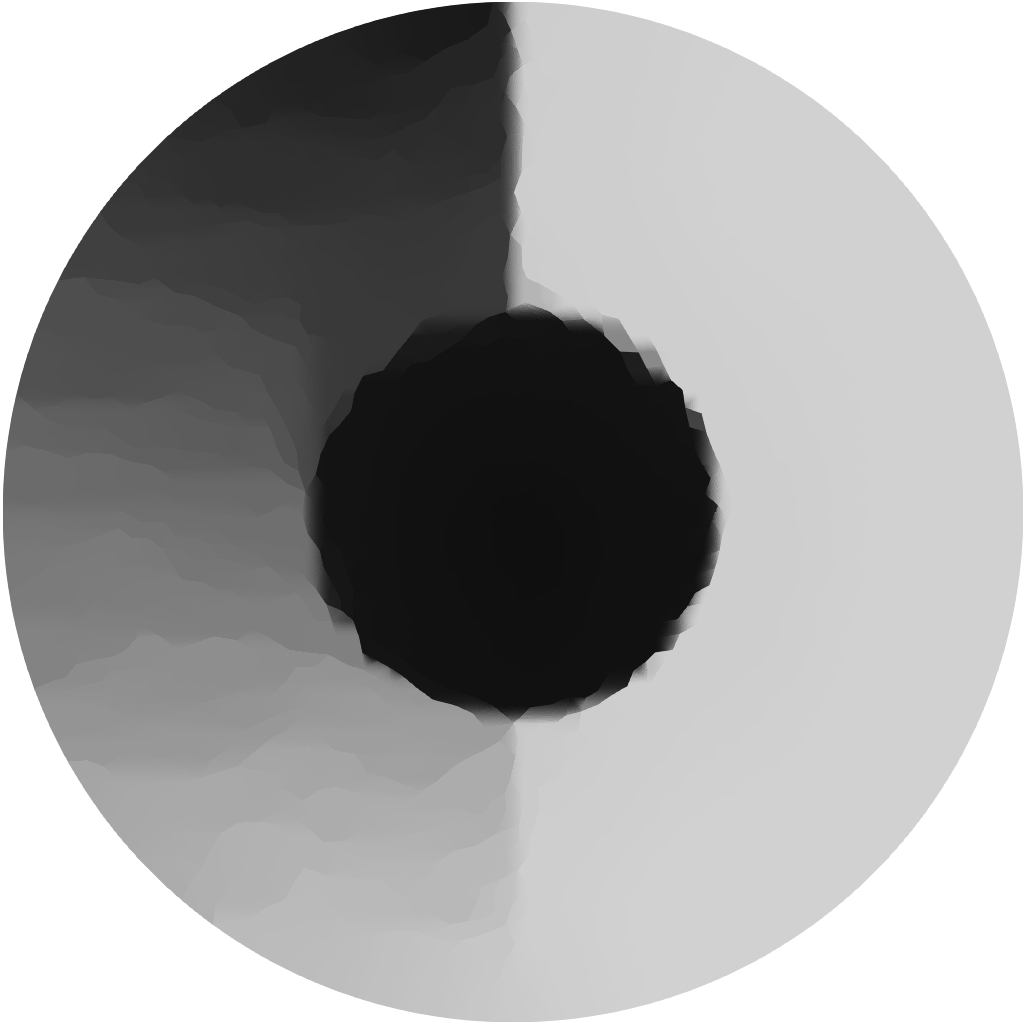}}{\includegraphics[width=0.48\linewidth]{./pix/Test4/T4_CP_DG2_1.png}}
		\hspace*{\fill}
	\end{center}
	\caption{Inpainting with $66.6\%$ of the cells erased (shown in black in the upper left image). The noisy images are not shown. Inpainting and denoising results for $\DG{0}$ (upper right), $\DG{1}$ (lower left) and $\DG{2}$(lower right) for \eqref{eq:DTV-L2} with parameter $\beta = \num{1e-3}$ for the isotropic setting ($s = 2$). Results obtained using \cref{alg:Chambolle-Pock_s_arbitrary} (Chambolle--Pock), see \cref{tab:data_test_4}.}
	\label{fig:TestCase4}
\end{figure}
	
\begin{table*}[htbp]
	\sisetup{scientific-notation=false}
	\centering
	\ifthenelse{\boolean{ispreprint}}{\scriptsize}{}
	\begin{tabular}{@{}llrrll@{}}
		\toprule
		Space      & Algorithm                                                                                                     & Iterations  & Time [s]    & $\psnr$      & Objective \\
		\midrule
		$\DG{0}$   & Chambolle--Pock ($\sigma = \num{0.70}$, $\tau = \num{1.25e-4}$, $\theta = \num{1}$, $\scaling = \num{1e-2}$)  & \num{2031}  & \num{47.7}  & \num{23.617} & \num{2.80e-3} \\
		\midrule
		$\DG{1}$   & Chambolle--Pock ($\sigma = \num{0.50}$, $\tau = \num{5.00e-4}$, $\theta = \num{1}$, $\scaling = \num{1e-2}$)  & \num{697}   & \num{49.0}  & \num{26.788} & \num{2.23e-3} \\
		\midrule 
		$\DG{2}$   & Chambolle--Pock ($\sigma = \num{0.07}$, $\tau = \num{1.50e-4}$, $\theta = \num{1}$, $\scaling = \num{1e-2}$)  & \num{2286}  & \num{354.0} & \num{26.385} & \num{2.47e-3} \\
		\midrule 
		\bottomrule \\
	\end{tabular}
	\caption{Performance of \cref{alg:Chambolle-Pock_s_arbitrary} (Chambolle--Pock) for the \eqref{eq:DTV-L2} inpainting problem shown in~\cref{fig:TestCase4} in various discretizations.}
	\label{tab:data_test_4}
\end{table*}

\added{The results for this combined inpainting and denoising problem are similar to those for the pure denoising case (\cref{subsec:denoising_of_DGr_in_DGr}).
Clearly, the higher-order results produce images closer to the original than the recovery in $\DG{0}$, which is also reflected in the $\psnr$ values.}

\section{Solving the \eqref{eq:DTV-L1} Problem}
\label{sec:DTV-L1}

We briefly discuss the implementation of two algorithms for 
\begin{equation*}
	\tag{DTV-L1}
	\text{Minimize} \quad \sum_{\added{T \subset \Omega_0}} \int_T \JJ_T \big\{ \abs{u-f} \big\} \, \dx + \beta \, \abs{u}_{DTV(\Omega)}
\end{equation*}
with $s \in [1,\infty]$.
They too can be realized equally efficiently as their original counterparts devised for images on Cartesian grids with low-order finite difference approximations of the gradient and divergence.
For simplicity, we restrict the discussion to the polynomial degrees $r \in \{0,1\}$ in this section so that all weights $c_{T,i}$, $c_{E,j}$ as well as $C_{T,k}$ are strictly positive.
The cases $r = \{2,3\}$ can be included provided that zero weights are properly treated and we come back to this in \cref{subsec:Missing_polynomial_degrees}.

\subsection{Chambolle--Pock Method}
\label{subsec:Chambolle-Pock_TV-L1}

We focus on the changes compared to the method for \eqref{eq:DTV-L2} discussed in \cref{subsec:Chambolle-Pock}.
As in \cref{subsec:Discrete_dual_problem_TV-L1}, we need to replace $F$ by \eqref{eq:definition_of_F_for_DTV-L1} and use the lumped inner product \eqref{eq:scalar_product_DGr_lumped} in $U = \DG{r}(\Omega)$.
Due to the diagonal structure of both $F$ and the inner product, the $F$-prox operator is easily seen to be $\alert{u} = \prox_{\sigma F}(\bar u)$ if and only if
\begin{equation}
	\label{eq:definition_of_prox_sigmaF_DTV-L1}
	\alert{u_{T,k}} 
	=
	f_{T,k} 
	+ 
	\shrink \big( \abs{\bar u_{T,k} - f_{T,k}}, \; \sigma \big)
	,
\end{equation}
\added{in case $T \subset \Omega_0$,} similarly as in \cite[Sect.~6.2.2]{ChambollePock2011}.
\added{In case $T \subset \Omega \setminus \Omega_0$, we have $u_{T,k} = \bar u_{T,k}$.}
The remaining steps in \cref{alg:Chambolle-Pock_s_arbitrary} are unaffected.

\subsection{ADMM Method}
\label{subsec:ADMM_TV-L1}

Finally we consider the Alternating Direction Method of Multipliers (ADMM) for the primal problem \eqref{eq:DTV-L1} as in \cite{TaoYang2009:1_preprint}.
In our context, similar as for the split Bregman method (\cref{subsec:split_Bregman}), one introduces variables $u \in \DG{r}(\Omega)$ and $\bd, \bb \in Y$.
A second splitting $e = u - f$ is required, so we additionally introduce $e \in \DG{r}(\Omega)$ as well as a multiplier $g \in \DG{r}(\Omega)$.
The corresponding augmented Lagrangian functional reads
\begin{multline}
	\label{eq:split_Bregman_ALM_functional_DTV-L1}
	\sum_{\added{T \subset \Omega_0},k} \abs{e_{T,k}} \, C_{T,k}
	+
	\beta \sum_{T,i} c_{T,i} \, \bigabs{\bd_{T,i}}_s 
	\ifthenelse{\boolean{ispreprint}}{}{\\}
	+
	\beta \sum_{E,j} \abs{\bn_E}_s \, c_{E,j} \, \abs{d_{E,j}}
	+
	\frac{\lambda}{2} \norm{\bd - \Lambda u - \bb}_Y^2
	\\
	+ 
	\frac{\lambda\added{\scaling}}{2} \sum_{T,k} C_{T,k} \, \bigabs{e_{T,k} - u_{T,k} + f_{T,k} - g_{T,k}}^2 
	.
\end{multline}
Let us briefly consider the individual minimization problems w.r.t.\ $u$, $\bd$, and $e$.
The $u$-problem is to minimize
\begin{multline}
	\label{eq:split_Bregman_ALM_functional_u_DTV-L1}
	\frac{\lambda\added{\scaling}}{2} \sum_{T,k} C_{T,k} \, \bigabs{e_{T,k} - \alert{u_{T,k}} + f_{T,k} - g_{T,k}}^2 
	\\
	+
	\frac{\lambda\added{\scaling}}{2} \sum_{T,i} c_{T,i} \, \bigabs{\bd_{T,i} - \nabla \alert{u}(\LagrangeNodesTrmone{i}) - \bb_{T,i}}_2^2 
	\ifthenelse{\boolean{ispreprint}}{}{\\}
	+
	\frac{\lambda}{2} \sum_{E,j} c_{E,j} \, \bigabs{d_{E,j} - \jump{\alert{u}}(\LagrangeNodesEr{j}) - b_{E,j}}^2
\end{multline}
w.r.t.\ $\alert{u} \in \DG{r}(\Omega)$.
This problem is similar to \eqref{eq:split_Bregman_ALM_functional_u} and it leads to a coupled linear system for $\alert{u}$.
The minimization of \eqref{eq:split_Bregman_ALM_functional_DTV-L1} w.r.t.\ $\bd \in \RT{r+1}^0(\Omega)$ is identical to \eqref{eq:split_Bregman_ALM_functional_d} and the $e$-problem is to minimize
\begin{equation}
	\label{eq:split_Bregman_ALM_functional_e_DTV-L1}
	\sum_{\added{T \subset \Omega_0},k} \abs{\alert{e}_{T,k}} \, C_{T,k}
	+ 
	\frac{\lambda\added{\scaling}}{2} \sum_{T,k} C_{T,k} \, \bigabs{\alert{e}_{T,k} - u_{T,k} + f_{T,k} - g_{T,k}}^2 
	.
\end{equation}
This problem can be easily solved via shrinkage, cf.\ \eqref{eq:split_Bregman_ALM_functional_dE}.
Finally, the multiplier update for $\bb$ is as in \cref{subsec:split_Bregman}, and the update for $g$ is similar; see \cref{alg:ADMM_DTV-L1_s_arbitrary}.
Once again, the solution $\bp$ of the dual \eqref{eq:dualDTV-L1} can be recovered from the multipliers $\bb_{T,i}$ and $b_{E,j}$ as in \eqref{eq:recover_p_from_discrete_multipliers}.
Moreover, it can be easily checked that 
\begin{equation}
	\label{eq:interpret_discrete_multiplier_g}
	\lambda \, g_{T,k} 
	=
	\frac{1}{C_{T,k}} \int_T (\div \bp) \, \LagrangeBasisTr{k} \, \dx
\end{equation}
holds, where the quantity on the right appears as a constraint in \eqref{eq:dualDTV-L1} and thus it satisfies $\abs{\lambda \, g_{T,k}} \le 1$ in the limit\added{ where $T \subset \Omega_0$ and $\abs{\lambda \, g_{T,k}} = 0$ where $T \subset \Omega \setminus \Omega_0$}.

\begin{algorithm}
	\caption{ADMM algorithm for \eqref{eq:DTV-L1} with $s \in [1,\infty]$}
	\label{alg:ADMM_DTV-L1_s_arbitrary}
	\begin{algorithmic}[1]
		\STATE Set $u^{(0)} \coloneqq f \in \DG{r}(\Omega)$, $\bb^{(0)} \coloneqq \bnull \in Y$ and $\bd^{(0)} \coloneqq \bnull \in Y$ 
		\STATE Set $e^{(0)} \coloneqq 0 \in \DG{r}(\Omega)$ and $g^{(0)} \coloneqq 0 \in \DG{r}(\Omega)$
		\STATE Set $n \coloneqq 0$
		\WHILE{not converged}
		\STATE Minimize \eqref{eq:split_Bregman_ALM_functional_u_DTV-L1} for $u^{(n+1)}$ with data $\bb^{(n)}$, $\bd^{(n)}$, $e^{(n)}$ and $g^{(n)}$
		\STATE Minimize \eqref{eq:split_Bregman_ALM_functional_dTE} for $\bd^{(n+1)}$ with data $u^{(n+1)}$ and $\bb^{(n)}$
		\STATE Minimize \eqref{eq:split_Bregman_ALM_functional_e_DTV-L1} for $e^{(n+1)}$ with data $u^{(n+1)}$ and $g^{(n)}$
		\STATE Set $\bb_{T,i}^{(n+1)} \coloneqq \bb_{T,i}^{(n)} + \nabla u^{(n+1)}(\LagrangeNodesTrmone{i}) - \bd_{T,i}^{(n+1)}$ 
		\STATE Set $b_{E,j}^{(n+1)} \coloneqq b_{E,j}^{(n)} + \jump{u^{(n+1)}}(\LagrangeNodesEr{j}) - d_{E,j}^{(n+1)}$
		\STATE Set $g_{T,k}^{(n+1)} \coloneqq g_{T,k}^{(n)} + u_{T,k}^{(n+1)} - f_{T,k} - e_{T,k}^{(n+1)}$
		\STATE Set $n \coloneqq n+1$
		\ENDWHILE
		\STATE Set $\bp^{(n)}$ by \eqref{eq:recover_p_from_discrete_multipliers} with data $\bb^{(n)}$ 
	\end{algorithmic}
\end{algorithm}

\section{Extensions}
\label{sec:Extensions}

In this section we collect a number of extensions showing that problems more general than those based on the TV-$L^2$ and TV-$L^1$ models and discontinuous functions can be dealt with efficiently by generalizations of the respective algorithms to our higher-order finite element setting.

\subsection{Huber TV-Seminorm}
\label{subsec:Huber_TV-Seminorm}

We consider the replacement of the TV-seminorm by its `Huberized' variant; see for instance \cite{Kuensch1994,PanReeves2006} and \cite[Ch.~4]{ScherzerGrasmairGrossauerHaltmeierLenzen2009}.
In the case $s = 2$, on which we focus here, the function $G$ in \eqref{eq:definition_of_G} can be written as
\begin{multline*}
	G(\bd) 
	=
	\sum_T \int_T \II_T \big\{ \abs{\bd_T}_2 \big\} \, \dx
	+
	\sum_E \int_E \II_E \big\{ \abs{d_E} \big\} \, \ds
	\ifthenelse{\boolean{ispreprint}}{}{\\}
	=
	\sum_{T,i} \, c_{T,i} \, \bigabs{\bd_{T,i}}_2
	+
	\sum_{E,j} \, c_{E,j} \, \abs{d_{E,j}}
	.
\end{multline*}
The corresponding Huber functional with parameter $\varepsilon > 0$ then becomes
\begin{multline}
	\label{eq:GHuber_in_case_s=2}
	G_\varepsilon(\bd) 
	=
	\sum_{T,i} \, c_{T,i} \, \max\left\{ \abs{\bd_{T,i}}_2 - \frac{\varepsilon}{2}, \; \frac{1}{2 \, \varepsilon} \abs{\bd_{T,i}}_2^2 \right\} 
	\\
	+
	\sum_{E,j} \, c_{E,j} \, \max\left\{ d_{E,j} - \frac{\varepsilon}{2}, \; - d_{E,j} - \frac{\varepsilon}{2}, \; \frac{1}{2 \, \varepsilon} (d_{E,j})^2 \right\}
	.
\end{multline}
It can be shown by straightforward calculations that the convex conjugate of $G_\varepsilon$ is
\begin{equation}
	\label{eq:GstarHuber_in_case_s=2}
	G_\varepsilon^*(\bp)
	=
	I_{\bP}\added{(\bp)} + \frac{\varepsilon}{2} \norm{\bp}_{Y^*}^2
	.
\end{equation}
We recall that $I_{\bP}$ is the indicator function of the constraint set $\bP$ in \eqref{eq:constraint_set}.

The `Huberized' discrete TV-seminorm is thus defined by $G_\varepsilon(\Lambda u)$ where $\Lambda$ is given in \eqref{eq:definition_of_Lambda}.
It can be combined with both the $L^2$ and $L^1$ loss terms,
\begin{equation*}
	F(u) = \frac{1}{2} \norm{u-f}_{L^2(\added{\Omega_0})}^2
\end{equation*}
and
\begin{equation*}
	F(u) = \sum_{\added{T \subset \Omega_0}} \int_T \JJ_T \big\{ \abs{u-f} \big\} \, \dx
	.
\end{equation*}
We refer to the corresponding primal problems, i.e., the minimization of $F(u) + \beta \, G_\varepsilon(\Lambda u)$, as (DTV$_{\!\varepsilon}$-L2) and (DTV$_{\!\varepsilon}$-L1).
The specific form of corresponding dual problems, where $F^*(-\Lambda^* \bp) + \beta \, G_\varepsilon^*(\bp/\beta)$ is minimized, should now also be clear.

\added{The Chambolle--Pock method (\cref{alg:Chambolle-Pock_s_arbitrary}) can be adapted in a straightforward way by replacing the $G^*$-prox by the one involving $G_\varepsilon^*$, i.e., by replacing \eqref{eq:prox_tauGstar} by}
\begin{equation}
	\label{eq:prox_tagGepsilonstar}
	\alert{\bp} = \argmin_{\bq \in \RT{r+1}^0(\Omega)} \frac{1}{2} \norm{\bq - \bar \bp}_{Y^*}^2 + \frac{\varepsilon}{2} \norm{\bq}_{Y^*}^2
	\text{ s.t.\ } \bq \in \bP.
\end{equation}
\added{In case $s = 2$, for instance, this amounts to}
\begin{equation}
	\begin{aligned}
		\RTDofsE{j}(\alert{\bp})
		&
		=
		\min \left\{ \added{\frac{1}{1+\varepsilon}} \abs{\RTDofsE{j}(\bar \bp)} , \; \beta \, \abs{\bn_E}_s \, c_{E,j} \right\} \frac{\RTDofsE{j}(\bar \bp)}{\abs{\RTDofsE{j}(\bar \bp)}}
		,
		\\
		\RTDofsT{i}(\alert{\bp})
		&
		=
		\min \left\{ \added{\frac{1}{1+\varepsilon}} \abs{\RTDofsT{i}(\bar \bp)}_2, \; \beta \, c_{T,i} \right\} \frac{\RTDofsT{i}(\bar \bp)}{\abs{\RTDofsT{i}(\bar \bp)}_2}
	\end{aligned}
\end{equation}
\added{in place of \eqref{eq:definition_of_prox_tauGstar}.}
\ifthenelse{\boolean{ispreprint}}%
{
	Chambolle's projection method (\cref{alg:Chambolle_Projection_s=2}) can also be adapted to (DTV$_{\!\varepsilon}$-L2) in the case $s = 2$ by modifying \eqref{eq:Chambolle_Projection_KKT_4} and \eqref{eq:Chambolle_Projection_gradient}.
	The modification necessary to the primal-dual active set method (\cref{alg:PDAS_s=1}) for $s = 1$ was already discussed in \cref{subsec:PDAS}.
}
{
}
\added{Similarly, the Chambolle--Pock method for \eqref{eq:DTV-L1} (\cref{subsec:Chambolle-Pock_TV-L1}) can be adapted to solve (DTV$_{\!\varepsilon}$-L1) with the same modification as above.}

\subsection{Polynomial Degrees}
\label{subsec:Missing_polynomial_degrees}

We recall that we restricted the discussion of algorithms for \eqref{eq:DTV-L2} and its dual \eqref{eq:dualDTV-L2} in \cref{sec:Algorithms} to the cases $r \in \{0,1,2,4\}$, each of which ensures that $c_{T,i}$ and $c_{E,j}$ are strictly positive; see \cref{lemma:basis_functions_positive_integrals}.
In the case $r = 3$, three of the six weights $c_{T,i}$ on each triangle are zero.
This is not a major issue but it requires some care when formulating the algorithms in \cref{sec:Algorithms} in this case.
Briefly, when $c_{T,i} = 0$, quantities bearing the same index $(T,i)$ are to be ignored.
This applies, in particular, to the inner product $\scalarprod{\cdot}{\cdot}_{Y^*}$ in \eqref{eq:inner_product_in_Ystar}.

Similarly, we excluded the cases $r \in \{2,3\}$ in the discussion of algorithms for \eqref{eq:DTV-L1} and its dual problem \eqref{eq:dualDTV-L1} in \cref{sec:DTV-L1} so that the weights $C_{T,k}\coloneqq \int_T \LagrangeBasisTr{k} \, \dx$ pertaining to the basis $\{\LagrangeBasisTr{k}\}$ of $\PP_r(T)$ are strictly positive as well.
In case $r = 3$, we proceed as discussed above, ignoring terms for which the corresponding weights $c_{T,i} = 0$.
When $r = 2$, we instead ignore terms for which $C_{T,k} = 0$ in any of the algorithms in \cref{sec:DTV-L1}.

\ifthenelse{\boolean{ispreprint}}%
{
\subsection{Images in $\CG{r}(\Omega)$}
\label{subsec:CG_instead_of_DG}

While we believe that the representation of images as discontinuous functions is rather natural, it is certainly useful to consider also the case when $u \in \CG{r}(\Omega)$.
This situation is meaningful only for $r \ge 1$, and hence we consider $r \in \{1,2,3,4\}$ in this section.
Clearly, for $u \in \CG{r}(\Omega)$, the TV-seminorm \eqref{eq:continuous_TV_for_DG} and its discrete counterpart \eqref{eq:discrete_TV_for_DG} reduce to 
\begin{subequations}
	\label{eq:continuous_and_discrete_TV_for_CG}
	\begin{align}
		\abs{u}_{TV(\Omega)} 
		&
		=
		\sum_T \int_T \abs{\nabla u}_s \, \dx
		\\
		\abs{u}_{DTV(\Omega)} 
		&
		=
		\sum_T \int_T \II_T \big\{ \abs{\nabla u}_s \big\} \, \dx
	\end{align}
\end{subequations}
since the terms related to edge jumps disappear.
As was mentioned in the introduction, the lowest-order case $r = 1$ has been considered in \cite{FengProhl2003,ElliottSmitheman2009,Bartels2012,BartelsNochettoSalgado2014,BerkelsEfflandRumpf2017,ClasonKruseKunisch2017}.
In this case, $\abs{u}_{TV(\Omega)} = \abs{u}_{DTV(\Omega)}$ holds.
Similarly as in \cref{cor:error_estimate_DTV}, a simple convexity argument shows that $\abs{u}_{TV(\Omega)} \le \abs{u}_{DTV(\Omega)}$ holds for all $u \in \CG{2}(\Omega)$.

Since $\CG{r}(\Omega)$ is a proper subspace of $\DG{r}(\Omega)$, it can be expected that it is enough to take the supremum in \cref{theorem:dual_representation_of_DTV} over a smaller set of test functions.
Indeed, as the image of the gradient operator $\Lambda: U = \CG{r}(\Omega) \to Y$ reduces to $Y = \prod_T \PP_{r-1}(T)^2$, the edge-based dofs of $\bp \in \RT{r+1}^0(\Omega)$ that can be dispensed with since no edge jumps need to be measured.
We thus obtain the following corollary of \cref{theorem:dual_representation_of_DTV}.

\begin{corollary}[Dual Representation of $\abs{u}_{DTV(\Omega)}$ for $u \in \CG{r}(\Omega)$]
	\label{corollary:dual_representation_of_DTV_for_CG}
	Suppose $r \in \{1,2,3,4\}$.
	Then for any $u \in \CG{r}(\Omega)$, the discrete TV-seminorm \eqref{eq:discrete_TV_for_DG} reduces to \eqref{eq:continuous_and_discrete_TV_for_CG} and it satisfies
	\begin{align}
		\MoveEqLeft
		\abs{u}_{DTV(\Omega)} 
		=
		\sup \Bigg\{ 
		\int_\Omega u \div \bp \, \dx: \bp \in \RT{r+1}^0(\Omega),
		\notag
		\\
		&
		\abs{\RTDofsT{i}(\bp)}_{s^*} \le c_{T,i} \text{ for all $T$, } i = 1, \ldots, r \, (r+1)/2,
		\notag
		\\
		&
		\RTDofsE{j}(\bp) = 0 \text{ for all $E$, } j = 1, \ldots, r+1 
		\Bigg\}
		.
		\label{eq:dual_representation_of_DTV_for_CG}
	\end{align}
\end{corollary}
It is straightforward to adopt the algorithms presented in \cref{sec:Algorithms,sec:DTV-L1} to this simpler situation.
In a nutshell, all edge-based quantities (such as $d_{E,j}$ and $b_{E,j}$ in the split Bregman method, \cref{alg:split_Bregman_s_arbitrary}) can be ignored, and the edge-based coefficients $\RTDofsE{j}(\bp)$ of any function $\bp \in \RT{r+1}^0(\Omega)$ would be left at zero.

We remark, however, that the gradient operator is not surjective onto $Y$\deleted{ unless $r = 1$ holds.}
\deleted{In case $r \in \{2,4\}$,}\added{so that} the set of test functions $\bp$ in \eqref{eq:dual_representation_of_DTV_for_CG} is unnecessarily large.
A more economical formulation for these cases remains open for future investigation.
}
{
}

\subsection{The 3D Case}
\label{subsec:3D}

\added{%
	When $\Omega \subset \R^3$ is triangulated by a mesh consisting of tetrahedra $K$ and interior facets $F$, then the former replace triangles $T$ and the latter replace interior edges $E$ throughout the paper.
	For instance, the definition \eqref{eq:discrete_TV_for_DG} of the discrete total variation becomes
}
\begin{multline}
	\label{eq:discrete_TV_for_DG_3D}
	\ifthenelse{\boolean{ispreprint}}{\hfill}{}
	\abs{u}_{DTV(\Omega)} 
	\ifthenelse{\boolean{ispreprint}}{}{\\}
	\coloneqq
	\sum_K \int_K \II_K \big\{ \abs{\nabla u}_s \big\} \, \dx
	+
	\sum_F \int_F \II_F \big\{ \bigabs{\vjump{u}}_s \big\} \, \ds
	.
	\ifthenelse{\boolean{ispreprint}}{\hfill}{}
\end{multline}
\added{%
	The definition of the jump \eqref{eq:definition_of_jump} across interior facets remains unchanged.
	The finite element spaces involved remain the same, except that their respective cell domains and thus their dimensions change; see \cref{tab:spaces_and_basis_functions_3D}.}
\begin{table}[htbp]
	\centering
	\ifthenelse{\boolean{ispreprint}}{}{\scriptsize}
	\begin{tabular}{@{}lll@{}}
		\toprule
		FE space           & local dimension               & global dimension \\
		\midrule
		$\CG{r}(\Omega)$   & $(r+1)(r+2)(r+3)/6$           & $N_K \, (r-3)^+ (r-2)(r-1)$
		\\
		($r \ge 1$)        &                               & ${} + N_F \, (r-2)^+ (r-1)/2$
		\\
		                   &                               & ${}+ N_E \, (r-1)^+ + N_V$
		\\
		\midrule
		$\DG{r}(\Omega)$   & $(r+1)(r+2)(r+3)/6$           & $N_K \, (r+1)(r+2)(r+3)/6$ 
		\\
		\midrule
		$\DG{r-1}(\Omega)$ & $r \, (r+1)(r+2)/6$           & $N_K \, r \, (r+1)(r+2)/6$
		\\
		\midrule
		$\DG{r}(\cup F)$   & $(r+1)(r+2)/2$                & $N_F \, (r+1)(r+2)/2$
		\\
		\midrule
		$\RT{r+1}^0(\Omega)$ & $(r+1)(r+2)(r+4)/2$         & $N_K \, r\,(r+1)(r+2)/2$ 
		\\
		                   &                               & ${}+N_F \, (r+1)(r+2)/2$
		\\
		\midrule 
		\bottomrule \\
	\end{tabular}
	\caption{Finite element spaces, their degrees of freedom and corresponding bases in 3D. Here $N_K$, $N_F$, $N_E$ and $N_V$ denote the number of tetrahedra, interior facets, interior edges and vertices in the triangular mesh; compare \cref{tab:spaces_and_basis_functions}.}
	\label{tab:spaces_and_basis_functions_3D}
\end{table}

\added{%
	The operator $\Lambda$, which represents the gradient and was defined in \eqref{eq:definition_of_Lambda}, now maps the space $U = \DG{r}(\Omega)$ onto $Y = \prod_K \PP_{r-1}(K)^3 \times \prod_F \PP_r(F)$.
	It is important to realize for our approach that $Y^*$ can still be identified with $\RT{r+1}^0(\Omega)$ with the duality product given by \eqref{eq:duality_product}, mutatis mutandis.
	From here, all results can be derived as in the 2D case. 
	We only mention that the analogue of \cref{lemma:basis_functions_positive_integrals} in 3D limits the polynomial degrees with non-negative weights to $r \in \{0,1,2,4\}$ in case of \eqref{eq:DTV-L2}; see \cite[Tab.~II]{Silvester1970}.
When \eqref{eq:DTV-L1} is considered, only the choices $r \in \{0,1\}$ remain.
}

\section{Conclusion and Outlook}
\label{sec:Conclusion_outlook}

In this paper we have introduced a discrete version (DTV) of the TV-seminorm for globally discontinuous ($\DG{r}$) Lagrangian finite element functions on \added{simplicial} grids in $\R^2$\added{ and $\R^3$}.
\added{Since continuous ($\CG{r}$) functions form a subspace of $\DG{r}$, all considerations apply to images represented as continuous finite element functions as well.}
We have shown that $\abs{\,\cdot\,}_{DTV(\Omega)}$ has a convenient dual representation in terms of the supremum over the space of Raviart--Thomas finite element functions, subject to a set of simple constraints.
This allows for the efficient realization of a variety of algorithms, e.g., \eqref{eq:dualDTV-L2} and \eqref{eq:dualDTV-L1}\added{ for image denoising and inpainting}, both with low and higher-order finite element functions available in finite element libraries.

Since we admit higher-order polynomial functions, it would be natural to extend our analysis to a discrete version of the total generalized variation (TGV) functional introduced in \cite{BrediesKunischPock2010}.
Another generalization that could be of interest is to consider finite element functions defined on more general cells than the simplices considered here.
Clearly rectangles are of particular interest in imaging applications, but also hexagons; see \cite{KnaupSteckmannBockenbachKachelriess2007,ColemanScotneyGardiner2016}, as mentioned in the introduction.
\added{We remark that $\RT{}$ finite element spaces on parallelograms were already discussed in the original contribution} \cite{RaviartThomas1977:1}\added{, and we refer to \cite{LeeParkPark2018_preprint} for an application to imaging, but only for the lowest-order case.
	The generalization to higher-order finite elements, as well as to more general element geometries, is left for future research.}

The polynomial degree in our \added{2D} study was limited to $0 \le r \le 4$ (or $0 \le r \le 3$ for \eqref{eq:DTV-L1}), which should be sufficient for most applications.
The limitation in the degree arises due to the requirement that the quadrature weights, i.e., the integrals over the standard Lagrangian basis functions, have to be non-negative; see \cref{lemma:basis_functions_positive_integrals}.
\deleted{%
Our approach carries over with no changes to tetrahedral grids on 3D~domains.
In this case the analogue of Lemma~\ref{lemma:basis_functions_positive_integrals} limits the polynomial degrees with non-negative weights to $r \in \{0,1,2,4\}$; see [Tab.~II]{Silvester1970}.
When \eqref{eq:DTV-L1} is considered, only the choices $r \in \{0,1\}$ remain.}
This brings up the question whether a Lagrangian basis for higher-order polynomial functions on triangles or tetrahedra exists, such that the integrals of the basis functions are (strictly) positive.
This is answered in the affirmative by results in \cite{WandzuraXiao2003,TaylorWingateBos2007} for the triangle and \cite{GellertHarbord1991,ZhangCuiLiu2009}\added{ for tetrahedra}, where interpolatory quadrature formulas with positive weights are constructed.
However, it remains to be investigated whether a Lagrangian finite element with a modified basis admits an appropriate Raviart--Thomas type counterpart such that a dual representation of $\abs{\,\cdot\,}_{DTV(\Omega)}$ parallel to \cref{theorem:dual_representation_of_DTV} continues to hold.
Moreover, such non-standard finite element spaces certainly incur an overhead in implementation.

One may also envision applications where it would be beneficial to allow for locally varying polynomial degrees and mesh sizes in imaging applications, so that the resolution can be chosen adaptively.
Finally, we mention possible extensions to vectorial TV-seminorms, see for instance \cite{GoldlueckeStrekalovskiyCremers2012}.
These topics remain for future research.

\appendix
\section{Examples of Finite Element Basis Functions}
\label{sec:Examples_of_FE_basis_functions}

\begin{figure}[htp]
	\centering
	\includegraphics[width=0.30\linewidth]{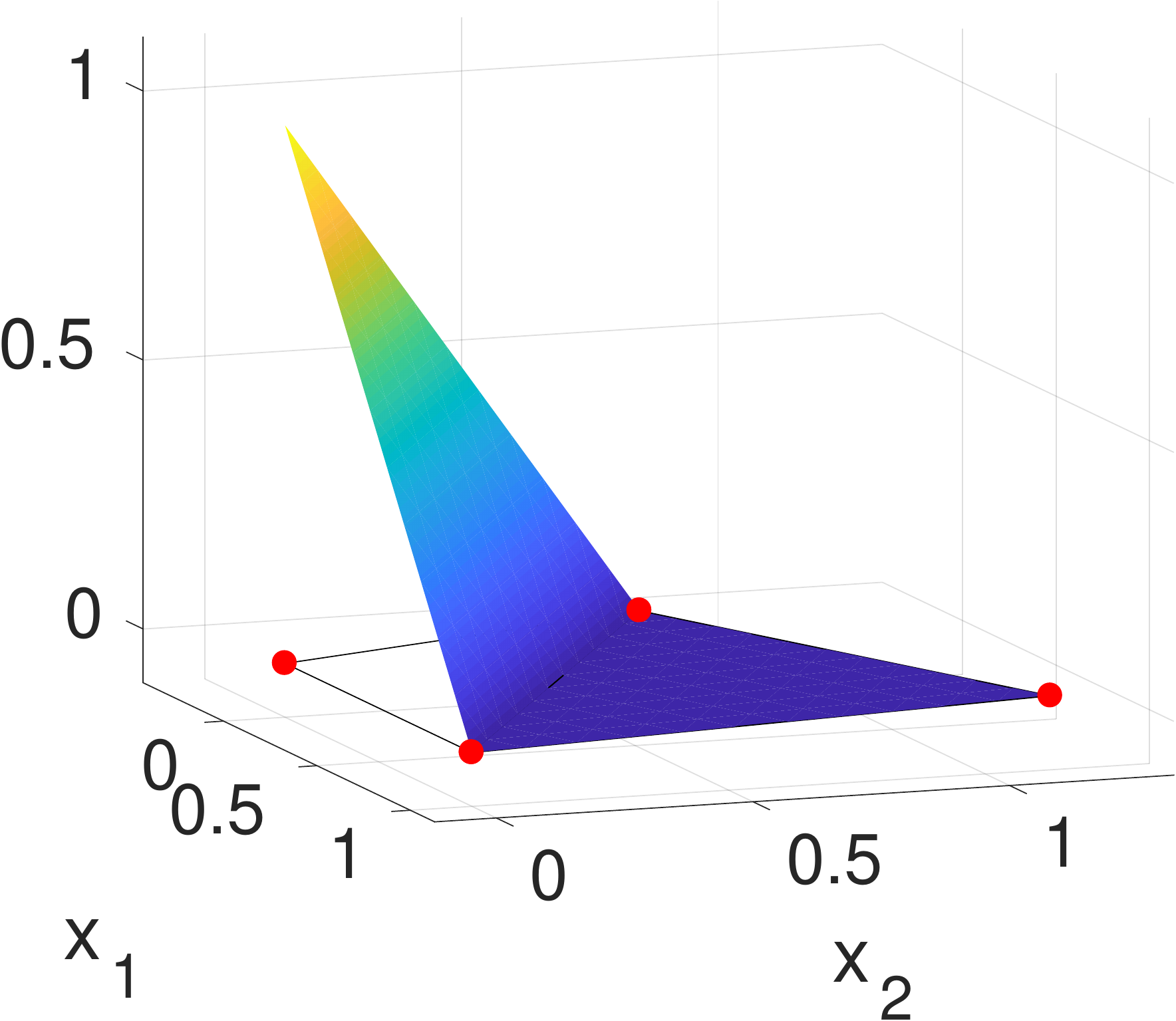}
	\hfill
	\includegraphics[width=0.30\linewidth]{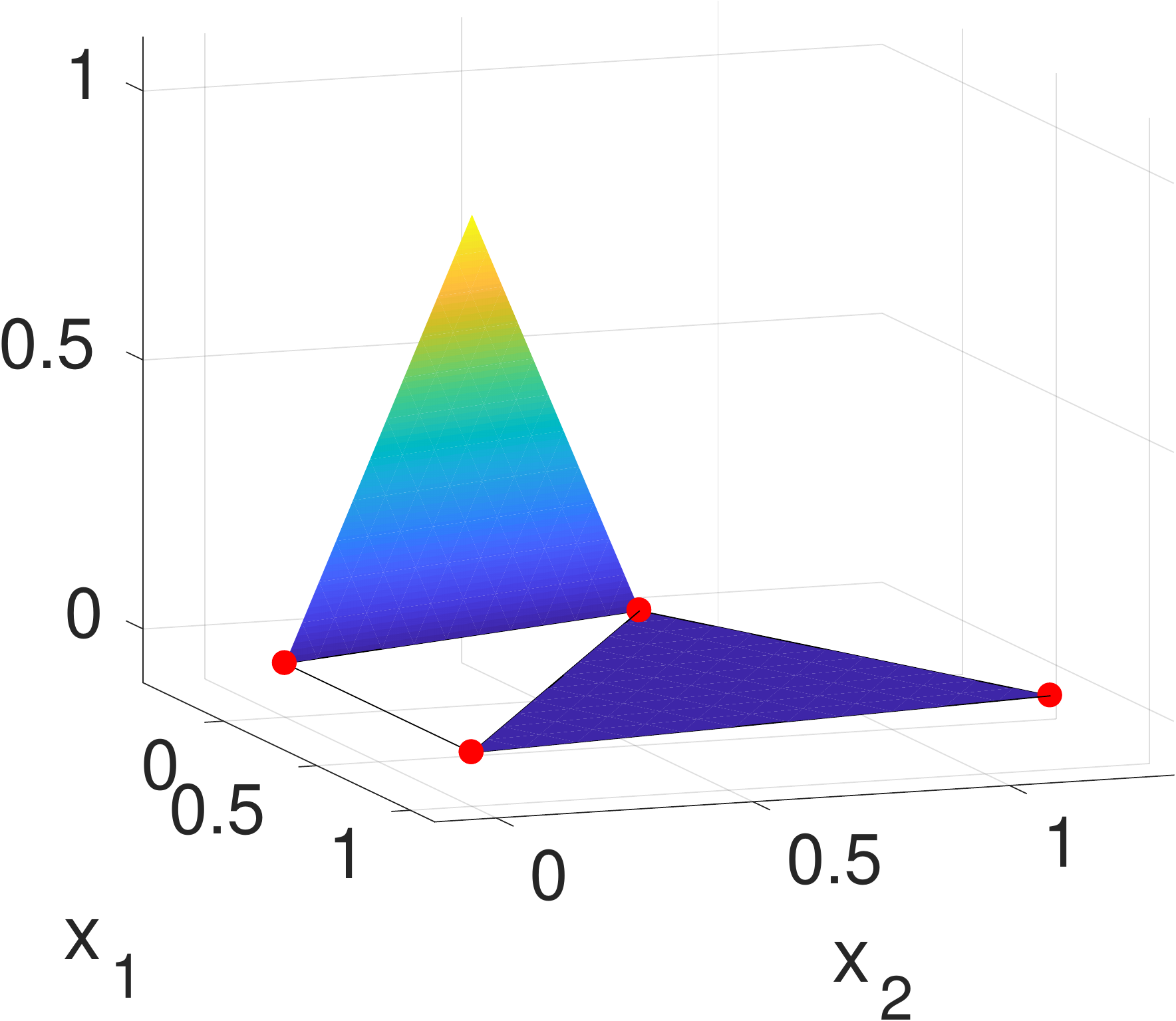}
	\hfill
	\includegraphics[width=0.30\linewidth]{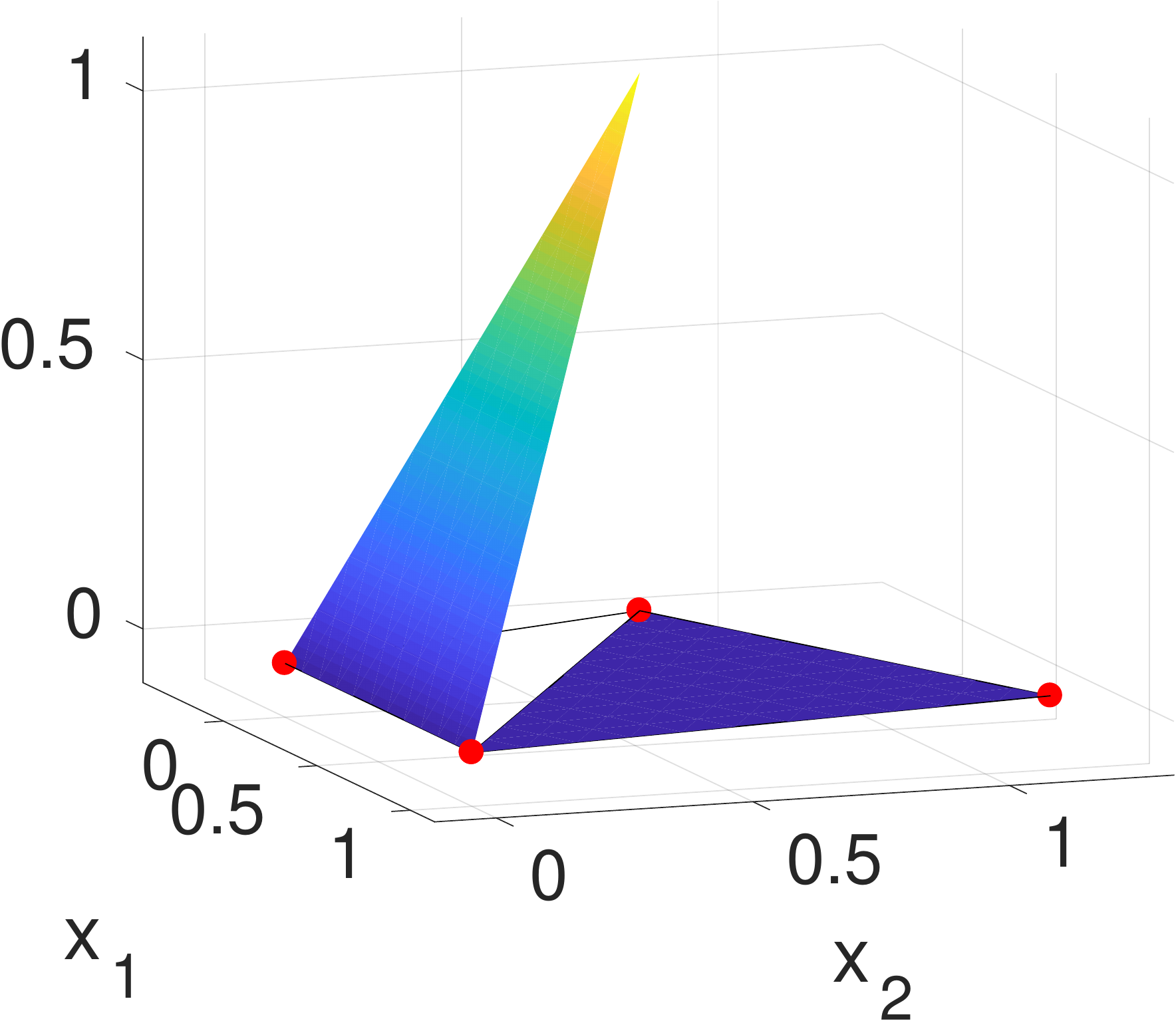}
	\\
	\includegraphics[width=0.30\linewidth]{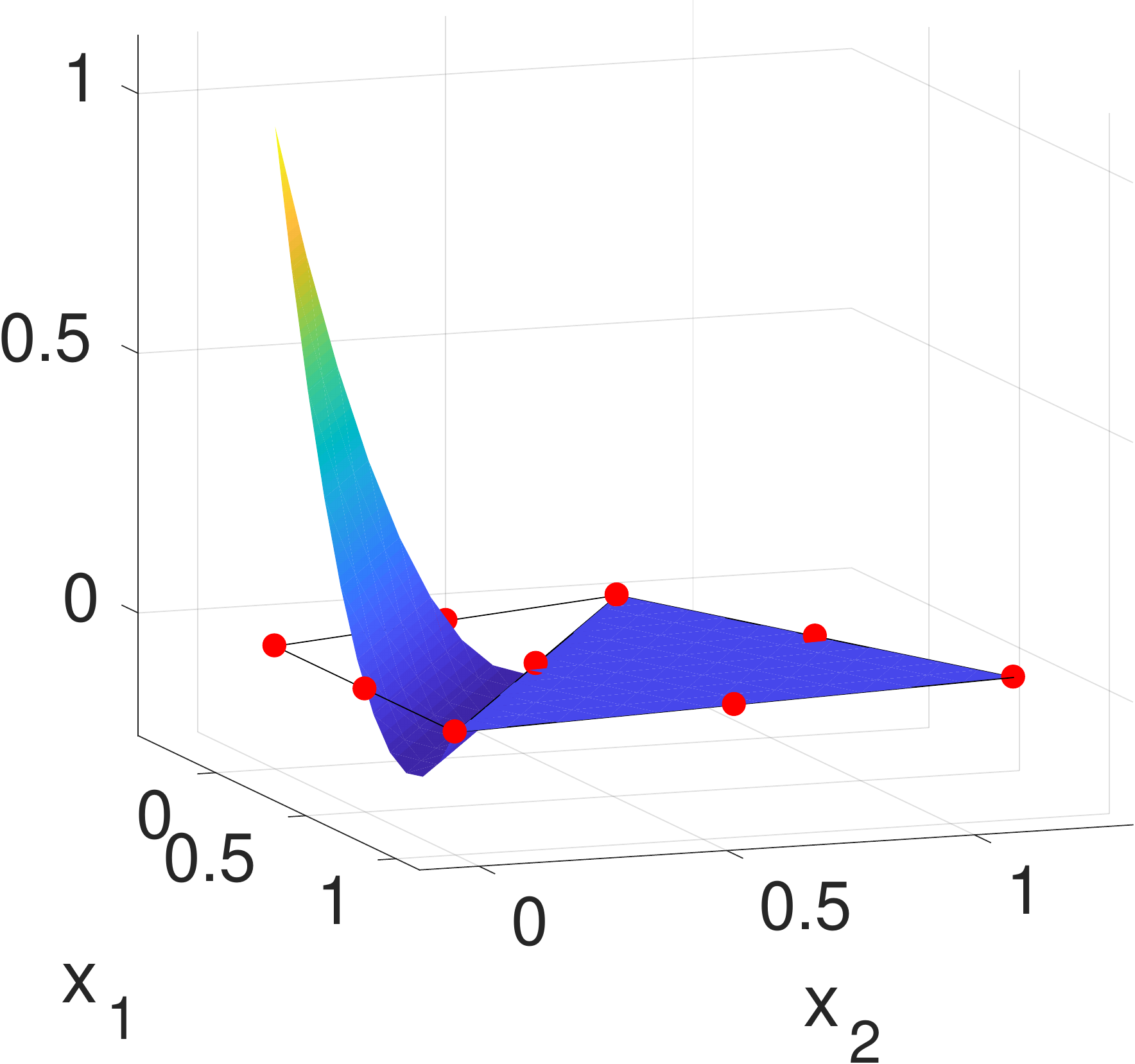}
	\hfill
	\includegraphics[width=0.30\linewidth]{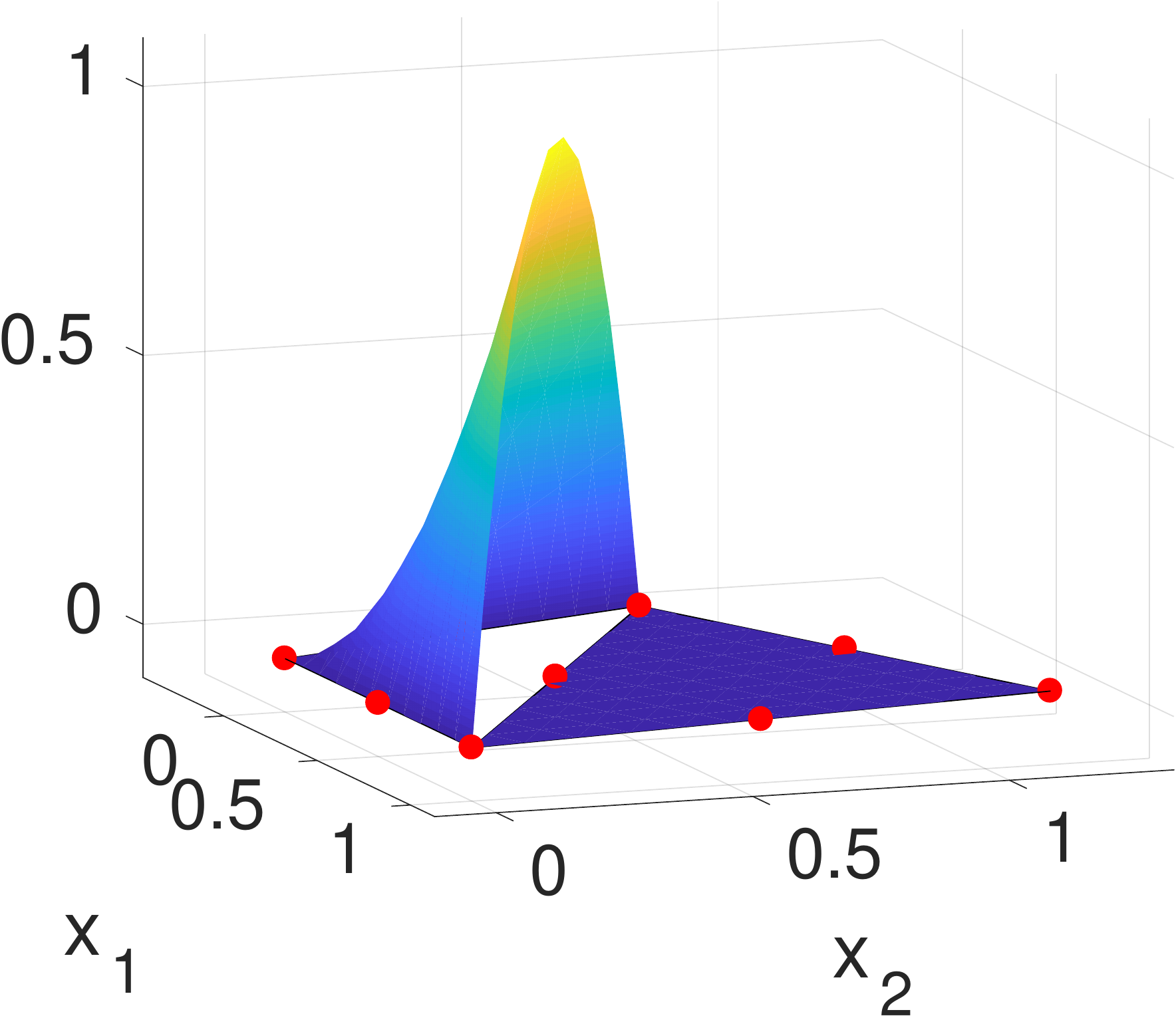}
	\hfill
	\includegraphics[width=0.30\linewidth]{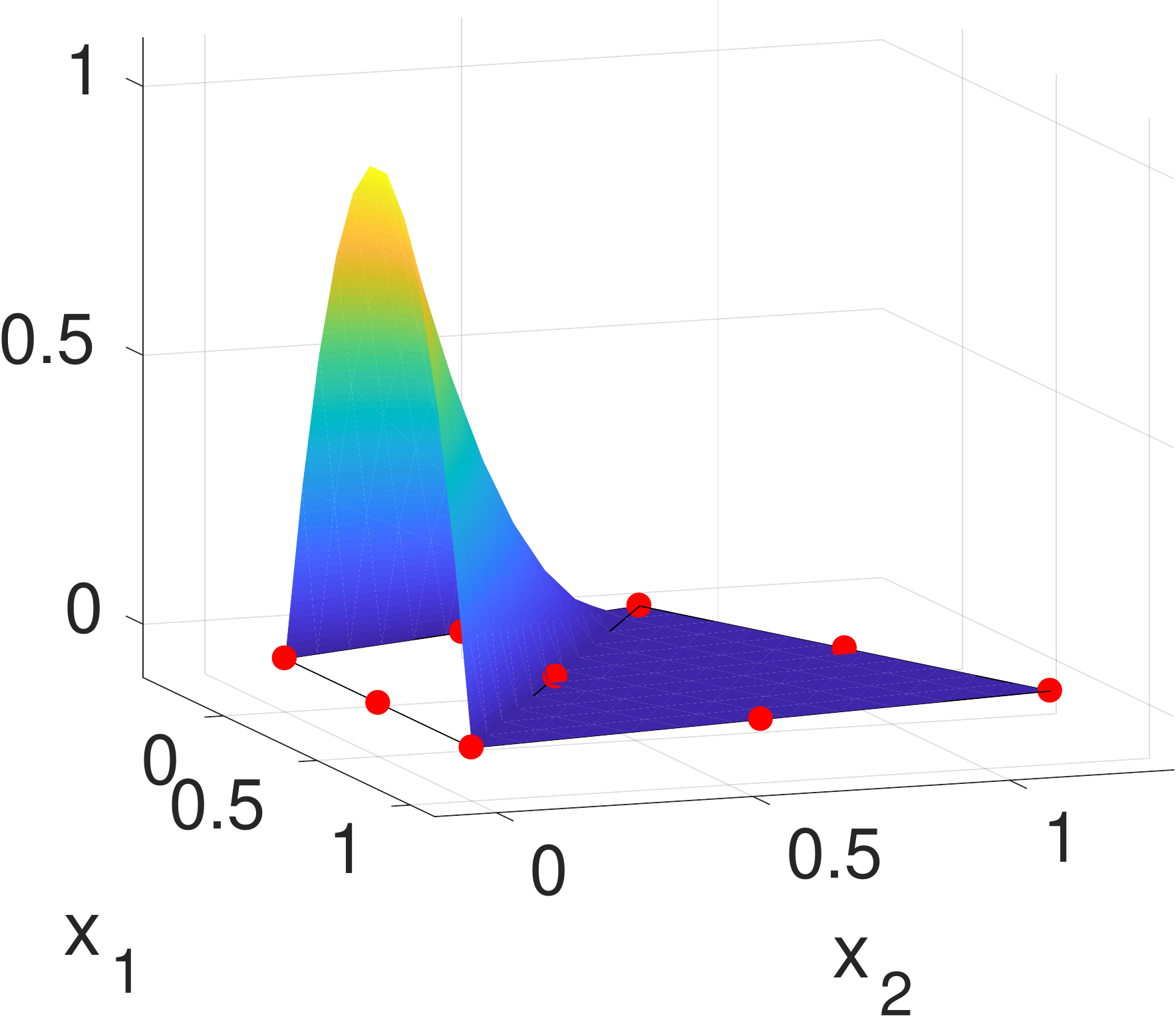}
	\caption{Some basis functions $\{\LagrangeBasisTr{k}\}$ of $\DG{r}(\Omega)$ for $r = 1$ (top row) and $r = 2$ (bottom row).}
	\label{fig:illustration_of_DG_basis_functions}
\end{figure}

\begin{figure}[htp]
	\centering
	\includegraphics[width=0.30\linewidth]{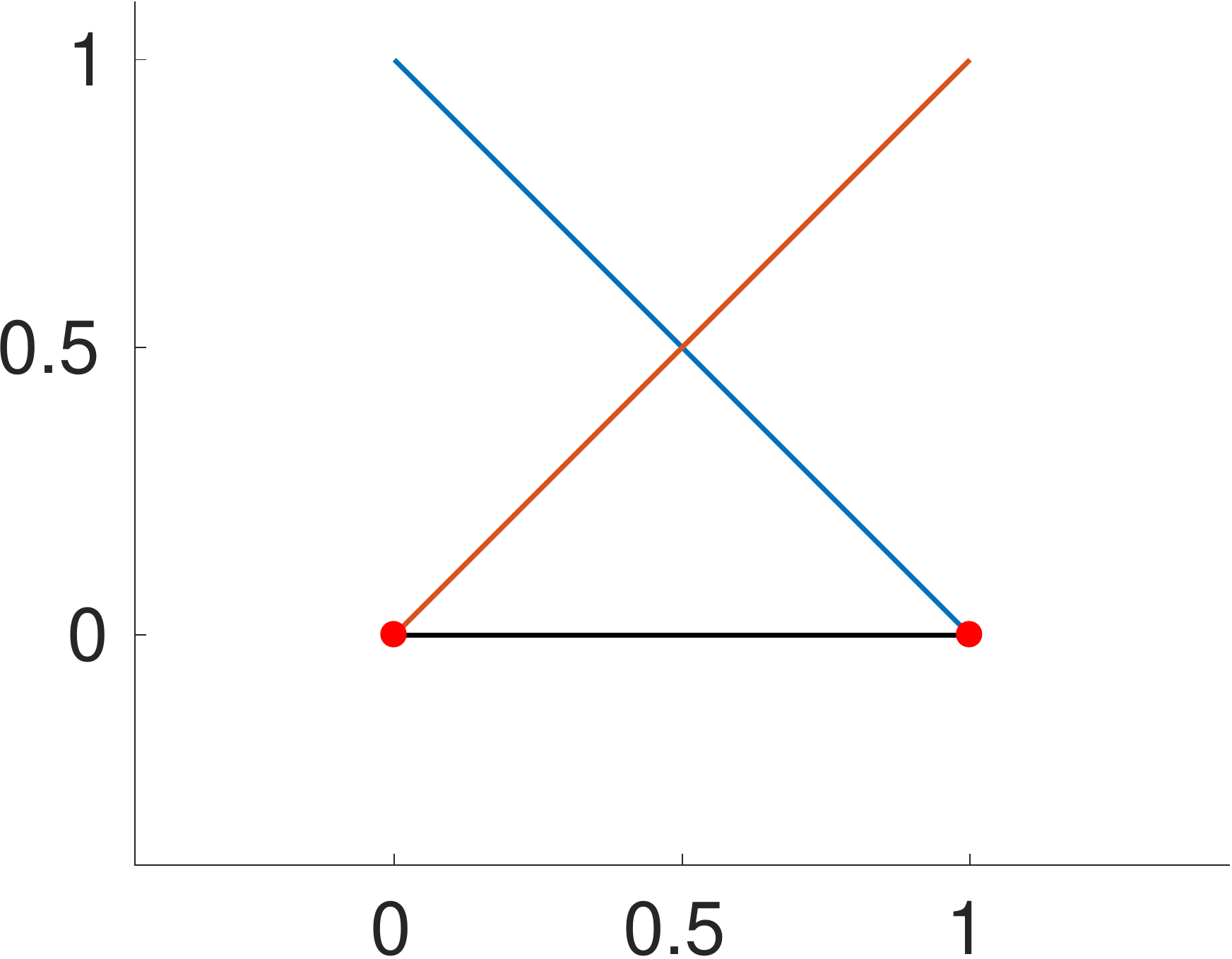}
	\hfill
	\includegraphics[width=0.30\linewidth]{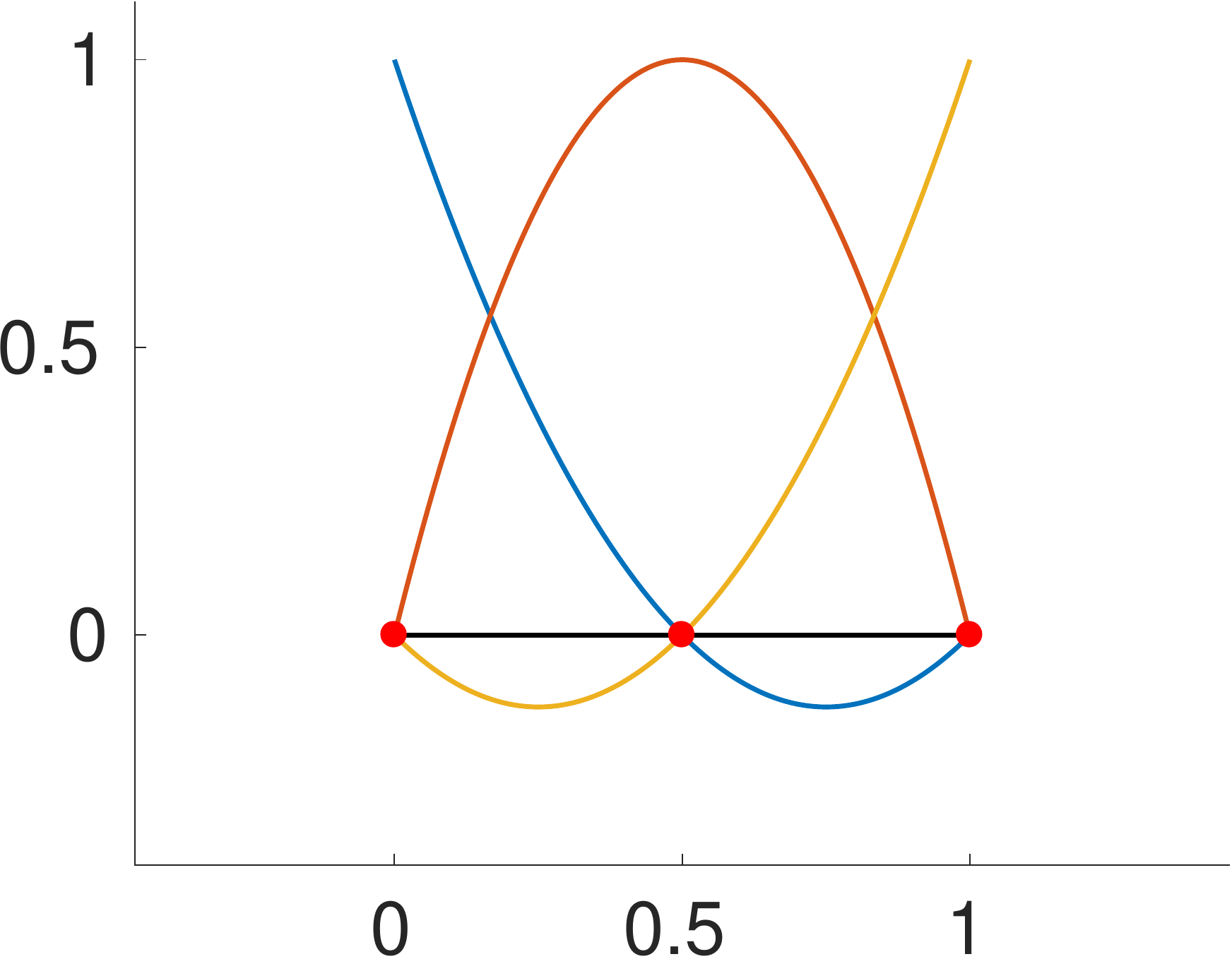}
	\hfill
	\includegraphics[width=0.30\linewidth]{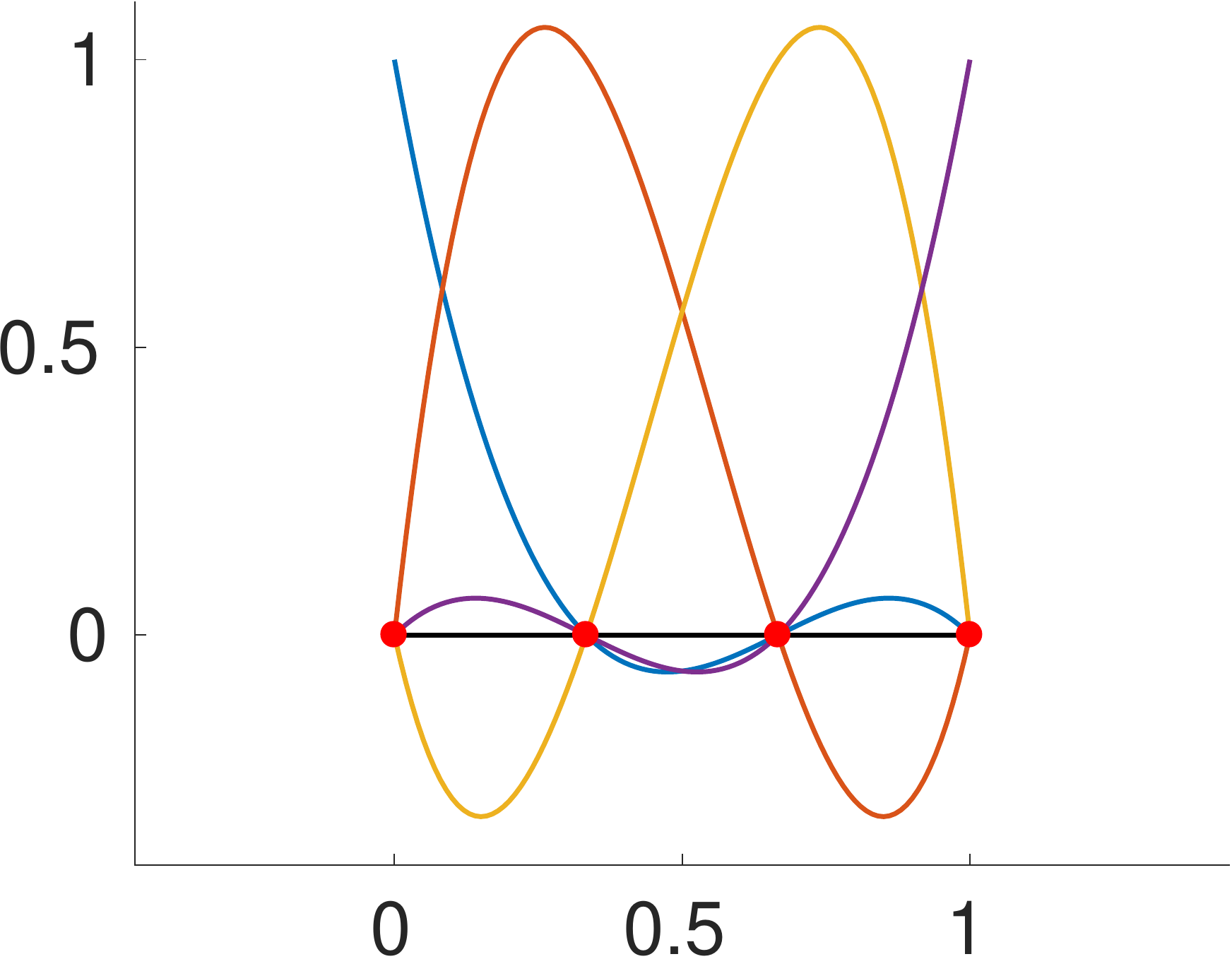}
	\caption{Complete set of basis functions $\{\LagrangeBasisEr{j}\}$ of $\PP_r(E)$ for $r \in \{1,2,3\}$ (from left to right).}
	\label{fig:illustration_of_Pr_basis_functions}
\end{figure}

\begin{figure}[htp]
	\centering
	\includegraphics[width=0.30\linewidth]{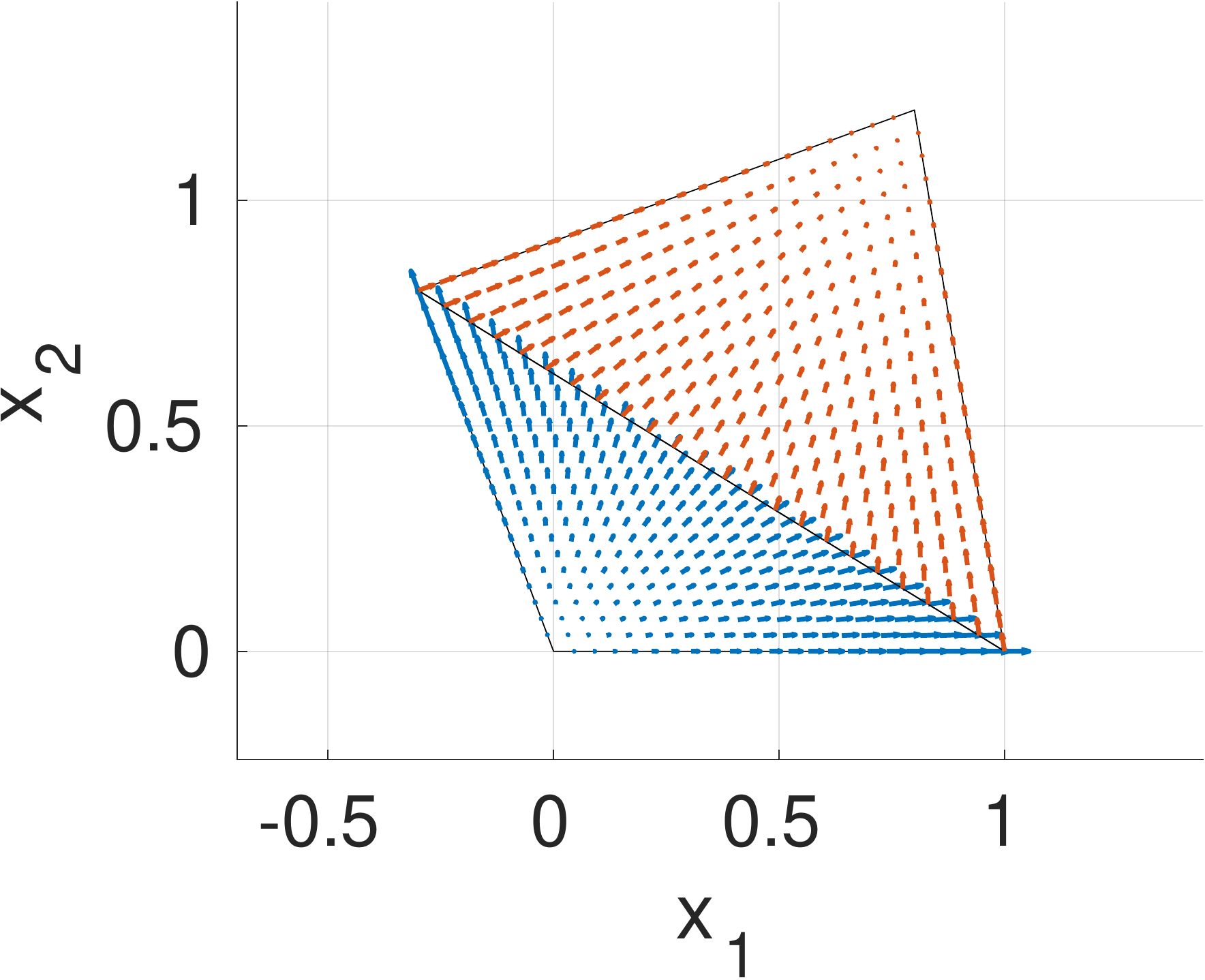}
	\hfill
	\includegraphics[width=0.30\linewidth]{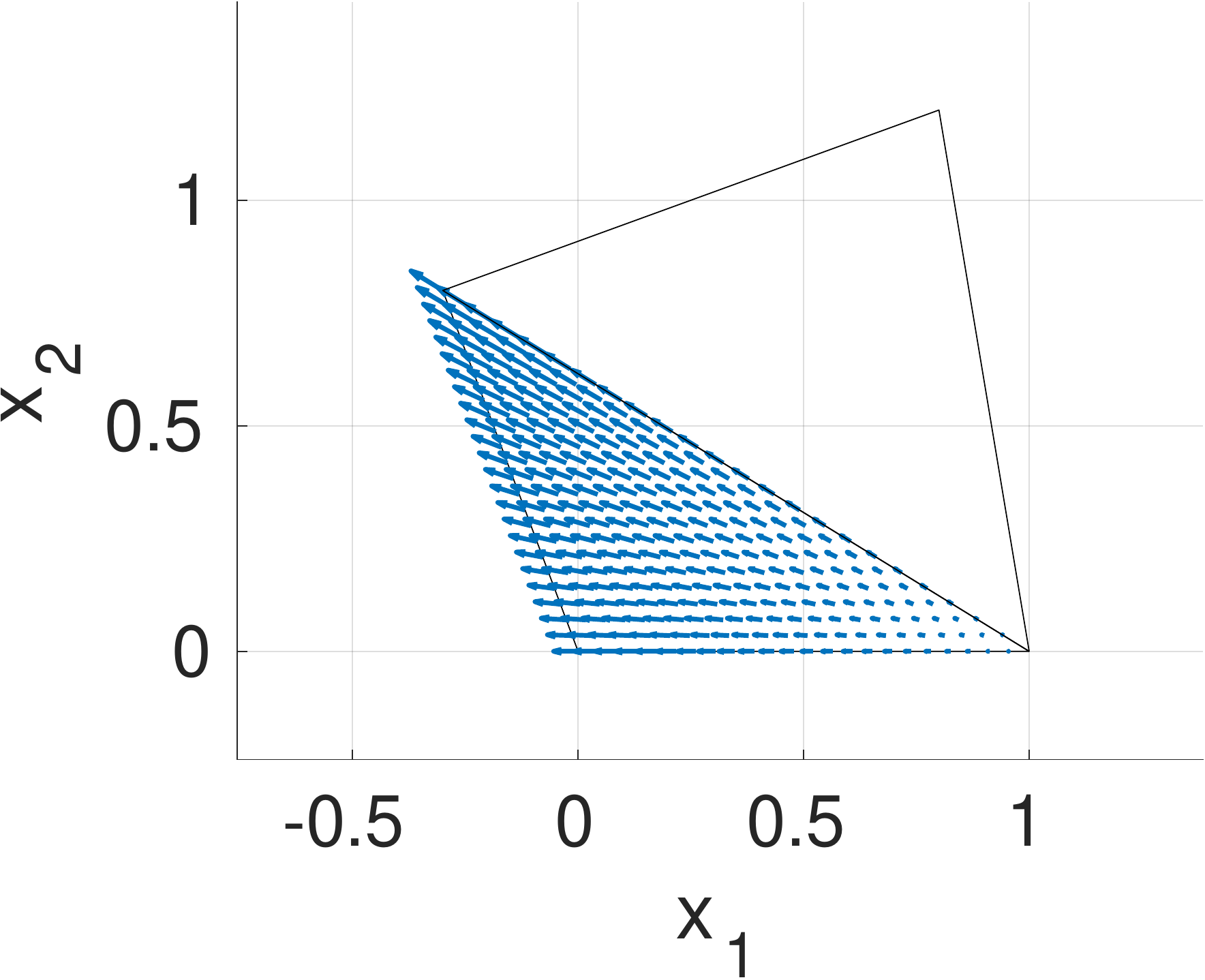}
	\hfill
	\includegraphics[width=0.30\linewidth]{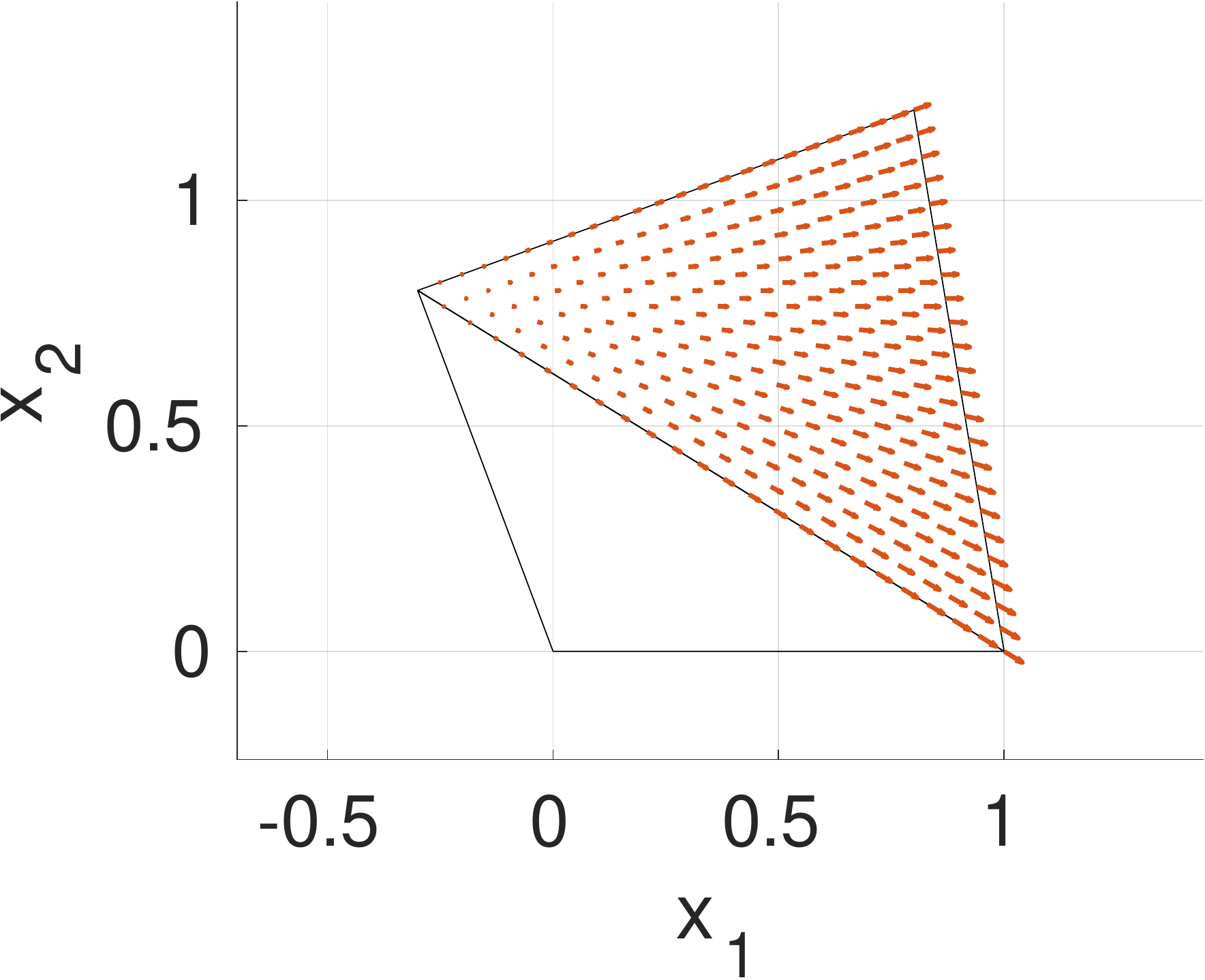}
	\\
	\includegraphics[width=0.30\linewidth]{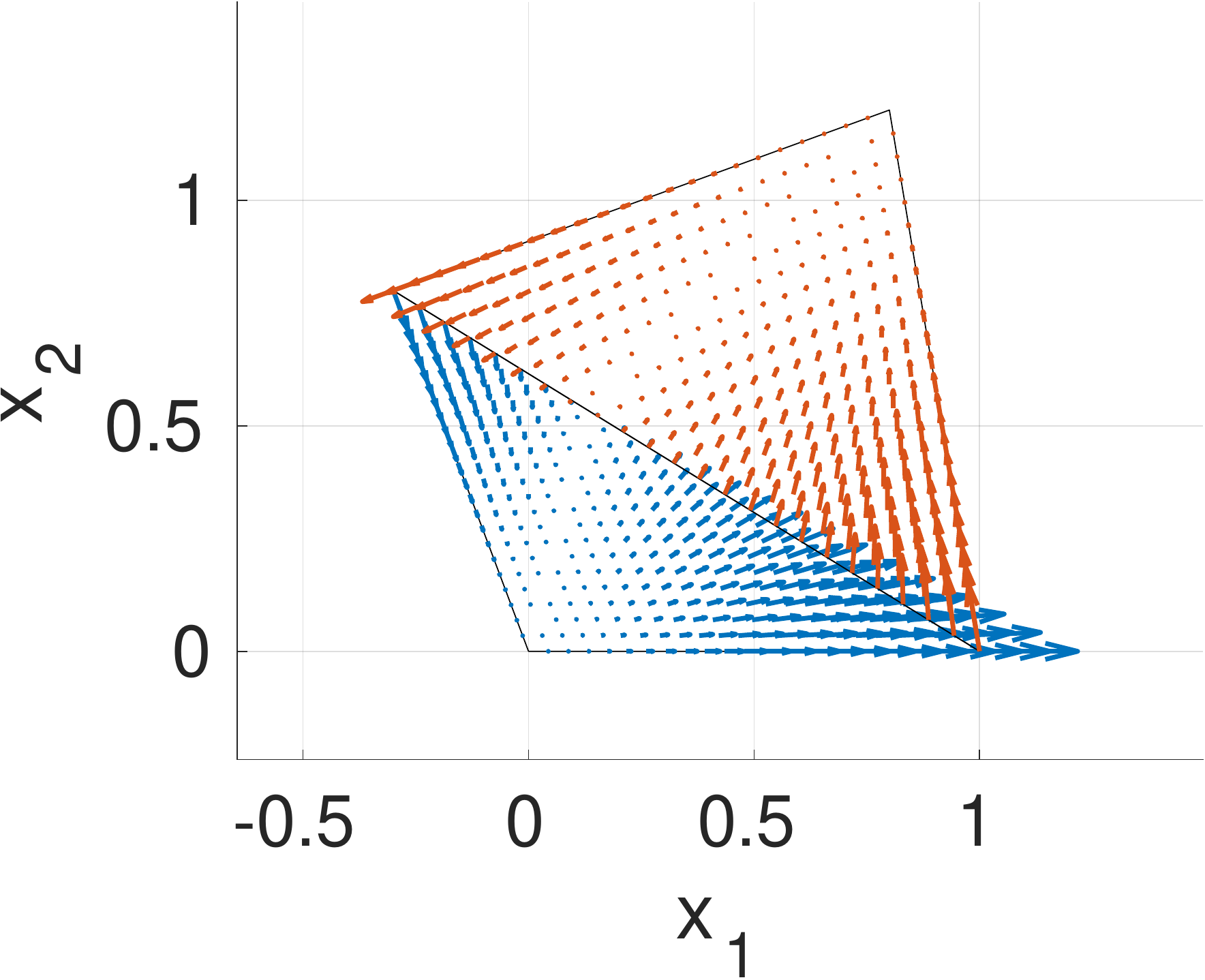}
	\hfill
	\includegraphics[width=0.30\linewidth]{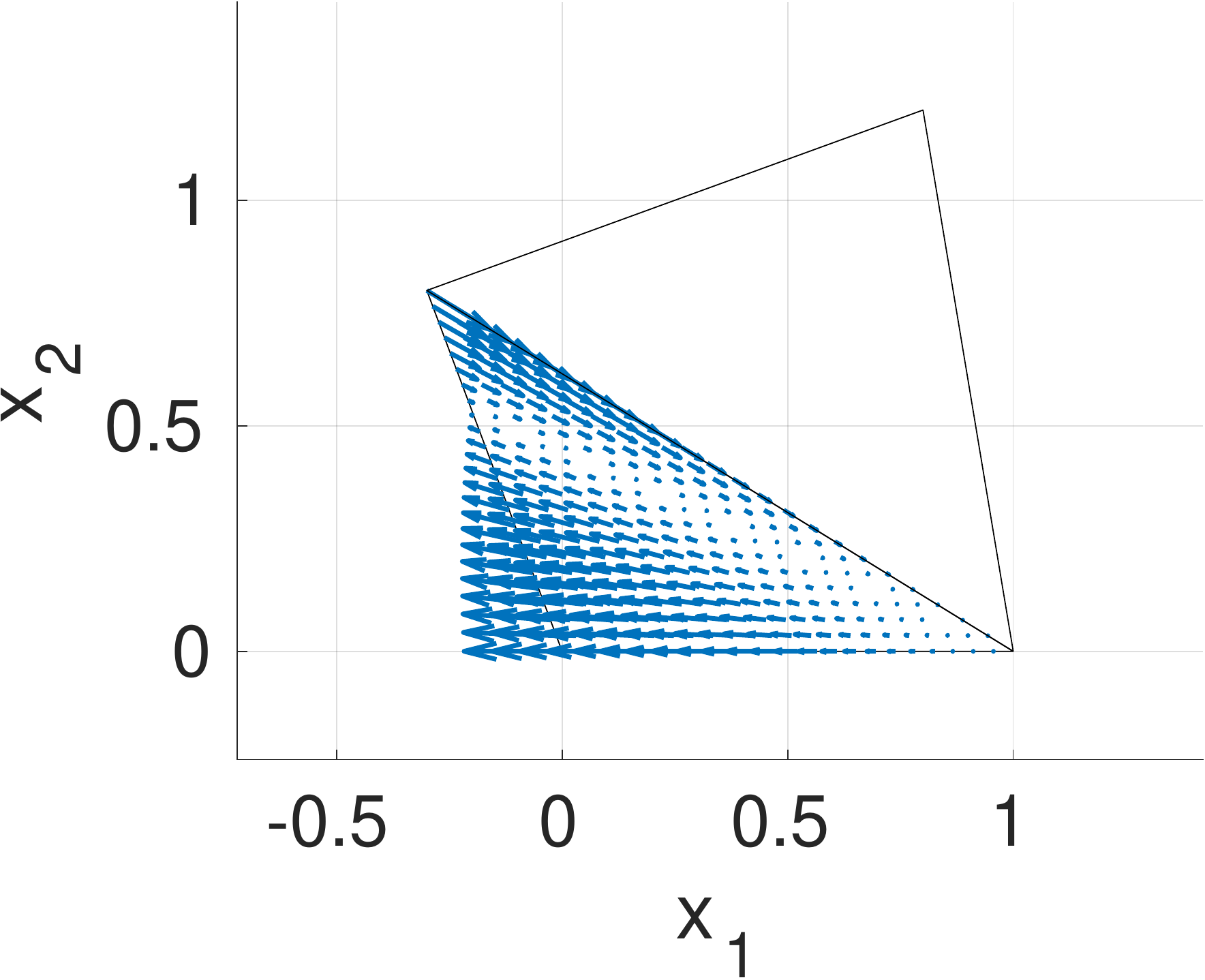}
	\hfill
	\includegraphics[width=0.30\linewidth]{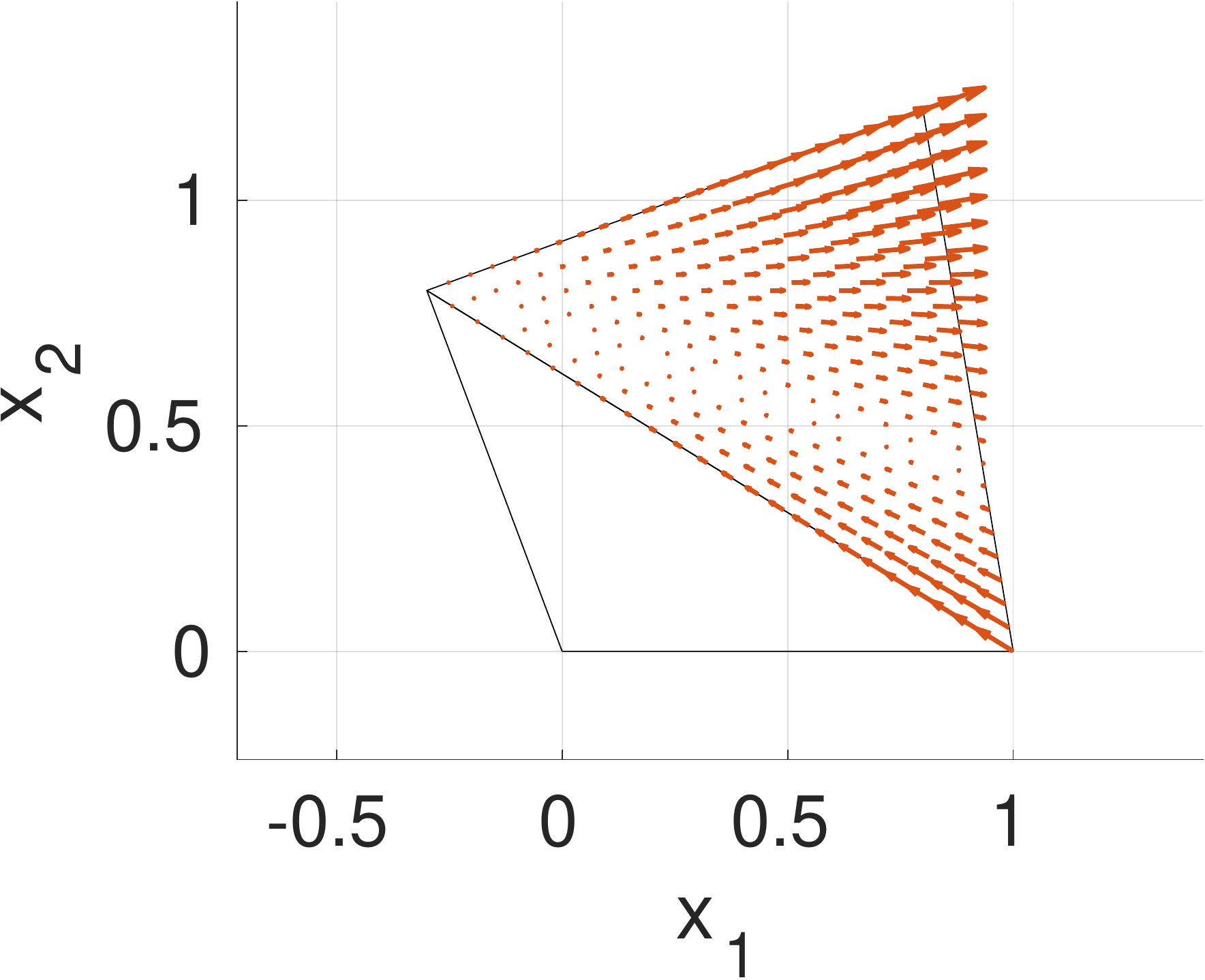}
	\caption{Some basis functions $\{\RTBasisE{j}\}$ of $\RT{r+1}(\Omega)$ for $r = 0$ (top row) and $r = 1$ (bottom row).}
	\label{fig:illustration_of_RT_basis_functions}
\end{figure}

\ifthenelse{\boolean{ispreprint}}%
	{\section*{Acknowledgments}}{\begin{acknowledgements}}
	This work was supported by DFG grants HE~6077/10--1 and SCHM~3248/2--1 within the \href{https://spp1962.wias-berlin.de}{Priority Program SPP~1962} (\emph{Non-smooth and Complementarity-based Distributed Parameter Systems: Simulation and Hierarchical Optimization}), which is gratefully acknowledged.
	The authors would like to thank Jan Blechta for help with custom quadrature schemes in \fenics.
	Part of this research was contrived while the second author was visiting the University of British Columbia, Vancouver.
	He would like to thank the Department of Computer Science for their hospitality.
	\added{Moreover, the authors would like to thank two anonymous reviewers for their constructive comments, which helped improve the presentation of the paper.}
\ifthenelse{\boolean{ispreprint}}%
{}{\end{acknowledgements}}